\documentclass[openany,10pt]{mt}
\usepackage{amssymb,amsmath,amsthm}
\usepackage{color}

\usepackage{enumerate}
\usepackage{todonotes}
\usepackage[all]{xy}
\usepackage[margin=1.3in]{geometry} 

\newcommand{\fe}{\mathfrak{e}}
\renewcommand{\preceq}{\preccurlyeq}
\renewcommand{\succeq}{\succcurlyeq}
\newcommand{\supp}{\operatorname{supp}}
\newcommand{\trunc}[2]{#1|_{#2}}
\newcommand{\cha}{\operatorname{char}}
\newcommand{\truncof}{\triangleleft}
\newcommand{\trunceq}{\trianglelefteq}

\def \<{\langle}
\def \>{\rangle}
\def \d{\operatorname{d}}
\def \h{\operatorname{h}}
\def \div{\operatorname{div}}
\def \ex{\operatorname{e}}
\def \smallexp{\operatorname{exp_s}}
\def \Li{\operatorname{Li}}
\def \No{\mathbf{No}}

\newcommand{\bigO}{\mathcal{O}}
\newcommand{\smallo}{{\scriptstyle\mathcal{O}}}
\newcommand{\smallk}{{\mathbf{k}}}
\newcommand{\ser}[1]{\sum_{\gamma} #1_\gamma t^\gamma}

\newcommand{\upper}{\hspace{-0.12cm}\uparrow}
\newcommand{\downer}{\hspace{-0.15cm}\downarrow}

\newcommand{\inv}{^{-1}}

\newtheorem{theorem}{Theorem}[chapter]
\newtheorem{lemma}[theorem]{Lemma}
\newtheorem{corollary}[theorem]{Corollary}
\newtheorem{proposition}[theorem]{Proposition}

\theoremstyle{definition}
\newtheorem*{example}{Example}

\numberwithin{section}{chapter}

\newcommand{\C}{\mathbb{C}}
\newcommand{\E}{\operatorname{E}}
\newcommand{\LE}{\operatorname{LE}}
\newcommand{\N}{\mathbb{N}}
\newcommand{\Q}{\mathbb{Q}}
\newcommand{\R}{\mathbb{R}}
\newcommand{\T}{\mathbb{T}}
\newcommand{\Z}{\mathbb{Z}}

\newcommand{\Texp}{\mathbb{T}_{\exp}}

\newcommand{\cF}{\mathcal{F}}

\newcommand{\cM}{\mathcal{M}}
\newcommand{\cP}{\mathcal{P}}

\newcommand{\cU}{\mathcal{U}}

\newcommand{\fG}{\mathfrak{G}}
\newcommand{\fH}{\mathfrak{H}}
\newcommand{\fM}{\mathfrak{M}}
\newcommand{\fN}{\mathfrak{N}}

\newcommand{\fd}{\mathfrak{d}}

\newcommand{\fg}{\mathfrak{g}}
\newcommand{\fh}{\mathfrak{h}}
\newcommand{\fm}{\mathfrak{m}}
\newcommand{\fn}{\mathfrak{n}}
\newcommand{\fp}{\mathfrak{p}}
\newcommand{\fq}{\mathfrak{q}}

\newcommand{\IP}{\operatorname{IP}}
\newcommand{\TP}{\operatorname{TP2}}
\newcommand{\SOP}{\operatorname{SOP}}

\DeclareFontFamily{U}{fsy}{}
\DeclareFontShape{U}{fsy}{m}{n}{<->s*[.9]psyr}{}
\DeclareSymbolFont{der@m}{U}{fsy}{m}{n}
\DeclareMathSymbol{\der}{\mathord}{der@m}{182}

\DeclareSymbolFont{der@m}{U}{fsy}{m}{n}
\DeclareMathSymbol{\derdelta}{\mathord}{der@m}{100}

\title{Truncation in Differential Hahn Fields}
\author{Santiago Camacho}

\usepackage{fancyhdr}   
\usepackage{lipsum}

\begin{document}
\frontmatter

\begin{abstract}
Being closed under truncation for subsets of generalized series fields is a robust property in the sense that it is preserved under various algebraic and transcendental extension procedures. Nevertheless, in Chapter 4 of this dissertation we show that generalized series fields with truncation as an extra primitive yields undecidability
in several settings. Our main results, however, concern the robustness of being truncation closed in generalized series fields equipped with a derivation, and under extension procedures that involve this derivation. In the last chapter we study this in the ambient field $\T$ of logarithmic-exponential transseries. It leads there to a theorem saying that under a natural ``splitting''  condition the Liouville closure of a truncation closed differential subfield of $\T$ is again truncation closed.

\end{abstract}

\maketitle

\pagebreak

This page is intentionally left blank

\tableofcontents

\mainmatter
\chapter{Introduction}

\noindent Consider the Taylor expansion
\[\log(1-t)\ =\ -t-\frac{t^2}{2} - \frac{t^3}{3} -\frac{t^4}{4}-\cdots\ =\ \sum_{n=1}^{\infty}\frac{-t^n}{n}.\]

\noindent
In a typical calculus course we use the Taylor polynomials $0$, $- t$, $- t - t^2/2,\ldots $ to approximate the value of $\log(1-t) $ for $t$ close to zero. These Taylor polynomials correspond to truncations of the Taylor series for $\log(1-t)$. Intuitively speaking, truncations of a series correspond to omitting smaller terms, and can be used to approximate the original series. There is an interest in embedding various classes of functions into series rings and fields. 
A classical instance is  given by associating Taylor series to analytic functions. 
In the case of complex meromorphic functions the association is made with Laurent series, and the corresponding truncations would correspond to Laurent polynomials, that is,
 elements of $\C[t,t^{- 1}]$. For other classes of functions, Laurent series are not enough, and a more general version is required. 
Given an ordered abelian group $\Gamma$ and a field $\smallk$, the Hahn series field $\smallk((t^\Gamma))$ over $\smallk$ consists of all formal series $\sum_\gamma f_\gamma t^\gamma$ where all $f_\gamma \in \smallk$ and the set of $\gamma\in \Gamma$ such that $f_\gamma$ is nonzero is a well ordered subset of $\Gamma$. For more details on $\smallk((t^\Gamma))$, see Section \ref{HS}. This is in fact a generalization since for $\smallk = \C$ and $\Gamma = \Z$ this gives the usual field $\C((t))$ of Laurent series over $\C$.
Truncation in this general setting is a lot more subtle than in the case of Laurent series. One result in this dissertation is that joining the notion of truncation to the valued field structure of a Hahn field results in undecidability; see Chapter \ref{TU}. Despite this ``negative'' feature of truncation, one can prove many ``positive'' things about it as we will
show in this dissertation.

\medskip\noindent
For various reasons (notational and traditional) we prefer to consider asymptotics at infinity, using a large variable $x$ rather than a small variable $t=x^{-1}$.  Thus for $x>1$,
\[\log(1-\frac{1}{x})\ =\ -\frac{1}{x} - \frac{1}{2x^2} - \frac{1}{3x^3} - \cdots\ =\ \sum_{n=1}^{\infty} \frac{-x^{-n}}{n}\]
and the truncations of $\log(1-\frac{1}{x})$ will be $0,\ -\frac{1}{x},-\frac{1}{x} -\frac{1}{2x^2},\ldots$ mimicking the concept of omitting ``small" terms. 


\medskip\noindent
Truncation first appears in a model theoretic setting when Mourgues and Ressayre \cite{MR} proved that every real closed field is isomorphic to a truncation closed
subfield of a Hahn field over $\R$. They use this result to prove that every real closed field $R$ has an \textit{integer part}, that is, a subring $Z$ of $R$ such that for every $x\in R$ there is a unique $z\in Z$ such that $z\le x < z+1$.   
This goes as follows: Let $R$ be a truncation closed subfield of $\R((t^\Gamma))$, and
let $\cU$ be the additive subgroup of $\R((t^\Gamma))$ consisting of the series $\sum_{\gamma<0} f_\gamma t^\gamma$. 
Then an integer part of $R$ is given by $(R\cap\cU) \oplus \Z$. 
Fornasiero~\cite{F} generalized the truncation result of Mourgues and Ressayre by showing that every henselian valued field of characteristic zero is isomorphic to
a truncation closed subfield of a Hahn field over its residue field \cite{F}. 
Van den Dries, Marker and Macintyre used the preservation of being truncation closed under various extension procedures to prove undefinability of certain functions in the o-minimal expansion $\R_{\operatorname{an},\exp}$ of the real field \cite{DMM}. Van den Dries gathered many of these results and expanded on them in~\cite{D}. 
Another potential role of truncation, prominent in \cite{Re}, is that it enables transfinite 
induction on the (well-ordered) support of series. This is a new
 tool in valuation theory, not fully exploited so far.  We do use this tool at various places. 

\medskip\noindent
There is increasing interest in Hahn fields $\smallk((t^\Gamma))$ with a `good'
derivation, and being truncation closed is significant 
in that context for similar reasons.  Investigating this is one of the goals of this dissertation.
The `good' derivations are introduced in Chapter~\ref{derha}. With a `good' derivation on $\smallk((t^\Gamma))$,
it is easy to show that the differential subfield of $\smallk((t^\Gamma))$ generated by a truncation 
closed subset is again truncation closed: see Section~\ref{extderco}. It is a more difficult to show that
adjoining the solutions in $\smallk((t^\Gamma))$ of first-order linear 
differential equations over a truncation closed differential subfield of 
$\smallk((t^\Gamma))$ to this subfield
preserves the property of being truncation closed: this is basically
the content of Theorem~\ref{thmA}. 

The next goal is to apply this to the differential
field $\mathbb{T}$ of (logarithmic-exponential) transseries, and 
related transseries 
fields. Now $\mathbb{T}$ is not a Hahn-field-with-good-derivation as above, 
but its definition does involve iterating a construction where at 
each step
one builds a Hahn-field-with-good-derivation whose coefficient field
is the previously constructed Hahn-field-with-good-derivation. 
In this way we have been able to prove our main result to the effect that
 under some mild assumption
on a truncation closed differential subfield of $\mathbb{T}$ that its
Liouville closure in $\mathbb{T}$ is also truncation closed,
where the Liouville closure of a differential subfield of $\mathbb{T}$ 
is the smallest differential subfield that is closed under exponentiation and 
taking antiderivatives. 

\medskip\noindent
We now turn to a more detailed description of the various chapters. 
Chapter 2 contains the preliminaries for the rest of the dissertation. We cover the basics about  ordered abelian groups, Hahn series, valued fields, and differential fields that are used later. In Chapter 3 we consider the operation of truncating Hahn series. We recall known results to the effect that
the property (for subsets of Hahn fields) of being truncation closed is preserved under several extension procedures. We also introduce the notion of an infinite part $\cU$ for a valued field, and we develop the notions of truncation with respect to $\cU$ and support of an element with respect to $\cU$. This can be thought of as an attempt to
define truncation in the first-order setting of valued fields. Chapter 4 contains undecidability results for valued fields with a monomial group and an infinite part: using some facts
about monadic second-order logic we show that Hahn fields equipped with a monomial group and truncation have an undecidable theory. We apply this to show that the exponential field of transseries with a predicate for its group of transmonomials is undecidable. We also prove a variant of the well-known fact that the differential field of Laurent series is undecidable. In Chapter 5 we introduce ``supported operators'' and prove some basic facts about these as a tool to be used in the
later chapters.  In Chapter 6  we consider derivations on Hahn fields for which the logarithmic derivative of every monomial  lies in the coefficient field. We show that for such a derivation, the differential field generated by a truncation closed subset  is again truncation closed, and that being truncation closed is preserved under adjoining solutions of certain types of differential equations of the form $Y'+bY+c=0$. In Chapter 7 we apply this to the  differential field $\Texp$ of exponential transseries. Although $\Texp$ is not a Hahn field,  it is an increasing union of (differential) Hahn fields where the notion of truncation still makes sense and is of interest. But for positive results in this setting about preserving the property of being truncation closed we need to introduce a sort of analogue of this property for sets of monomials, which we call  splitting. In Chapter 8 we consider the field $\T$ of logarithmic-exponential transseries as a union of copies of $\Texp$ and prove our main result. We conclude by briefly discussing some possibilities 
of expanding on  the results in this dissertation. 

\chapter{Preliminaries}

\noindent
In this chapter we summarize basic facts about valued fields, especially Hahn fields, and valued differential fields, as needed
in later chapters. We try to be self contained as to the various algebraic notions involved, and thus we include also some preliminaries on valuation theory and differential algebra. 
Almost all of this can be found in more detail and with proofs in \cite[Chapters 2, 3, 4]{ADH}. 

\medskip\noindent
We let $m,n$ range over $\mathbb{N}=\{0,1,2,\ldots\}$, the set of natural numbers.  Rings will have a multiplicative identity $1$. We identify $\mathbb{Z}$ in the usual (unique) way with a subring of any ring of characteristic zero. We let $K^\times$ denote the multiplicative group of a field $K$. For a field $K$ of characteristic zero, we identify $\Q$ with a subfield of $K$ in the usual way.

\section{Ordered Sets and Groups}

\subsection*{Ordered sets}
An \textbf{ordered set} $S$ is a set equipped with a distinguished {\em total\/} ordering $\leq$. 
If $S$ is an ordered set with ordering $\leq$ we let $\geq, <,$ and $>$ have the usual meaning. 
Subsets of $S$ will be considered as ordered sets with the ordering induced by $\leq$. 

\medskip \noindent
For an ordered set $S$ and a subset $B$ of $S$ we set $S^{>B}:= \{s\in S: s>b \text{ for all } b \in B\}$, and $S^{>s}:= S^{>\{s\}}$ for an element $s\in S$. Similarly we define $S^{<B}, S^{\leq B}, S^{\geq B}$. 
We set $S_\infty := S\cup \{\infty\}$ and extend the ordering of $S$ to an ordering on $S_\infty$ by $s< \infty$ for every $s\in S$.  

\medskip\noindent
Let $S$ be an ordered set. Then $S$ is  said to be \textbf{well-ordered} if every non-empty subset of $S$ has a minimum element with respect to $\leq$.
A subset $B$ of $S$ is said to be \textbf{convex in } $S$ if for all $a,b,c\in S$ with $a<b<c$ and $a,c\in B$ we have $b\in B$. The \textbf{convex hull} of $B\subseteq S$ is the smallest subset of $S$, under inclusion, that is convex and contains $B$.

\subsection*{Ordered abelian groups} Let $\Gamma$ be an additively written \textbf{ordered abelian group}, that is, an abelian group with an ordering such that for all $\alpha,\beta,\gamma\in \Gamma$ 
\[\alpha<\beta\ \Rightarrow\ \alpha+\gamma<\beta+\gamma\]

\medskip\noindent
If $\Gamma$ is an ordered abelian group we write $\Gamma^>$ instead of $\Gamma^{>0}$; similarly with $\Gamma^{<}, \Gamma^{\leq}, \Gamma^{\geq}$. Additionally we set $\Gamma^{\neq} = \Gamma \setminus \{0\}$ 
Given $\gamma \in \Gamma$ we set $|\gamma| = \begin{cases} \gamma \text{ if } \gamma \geq 0\\ -\gamma \text{ if } \gamma<0. \end{cases}$\\
Ordered abelian groups will be considered as $\Z$-modules in the usual way. 

\medskip\noindent
We consider the archimedean equivalence relation on an ordered abelian group $\Gamma$ where the archimedean class of an element $\gamma_0 \in \Gamma$ is given by 
\[ [\gamma_0] \ : =\ \{\gamma \in \Gamma \ :\ |\gamma_0| \leq n|\gamma| \text{ and } |\gamma|\leq n|\gamma_0| \text{ for some } n\ >\ 0\}.\]
We will consider the set of archimedean classes as an ordered set where 
\[[\gamma]\ \leq\ [\delta] \text{ iff } |\gamma|\ \leq\ n |\delta| \text{ for some } n  .\]
Thus $[0]=\{0\}$, and $[0] \leq [\gamma]$ for all $\gamma \in \Gamma$, with equality iff $\gamma = 0$; moreover 
\[[\gamma]\ <\ [\delta] \text{ iff } n|\gamma|\ <\  |\delta| \text{ for all } n  .\]
The \textbf{rank} of $\Gamma$ is the number of archimedean classes of $\Gamma^{\neq}$;  it equals the number of nontrivial convex subgroups of $\Gamma$ if the rank is finite. An ordered abelian group of rank 1 is also said to be \textbf{archimedean}. 

\begin{lemma}\label{fGfN1} If $\Delta, \Theta \subseteq \Gamma$ are well-ordered, then $\Delta +\Theta$ is well-ordered, and for every $\gamma\in \Gamma$ there are only finitely many $(\delta,\theta)\in \Delta\times\Theta$ with $\delta+\theta=\gamma$.
\end{lemma}

\begin{lemma}\label{fGstar1}
Let $G\subseteq \Gamma^>$ be a well ordered set. 
Then 
\[G^*\ :=\ \sum G\ :=\ \{g_1+\cdots +g_n\ :\ g_1,\ldots,g_n\in G\}\]
is well ordered and for every $\gamma\in G^*$ there are finitely many tuples $(n,g_1,\ldots,g_n)$, with $g_1,\ldots, g_n$ elements in $G$ such that $\sum_{i=1}^n g_i = \gamma$.
\end{lemma}

\subsection*{Valuations on abelian groups}\label{VAG}
A valuation on an (additively written) abelian group $\Gamma$ is a surjective map $v:\Gamma \rightarrow S_\infty$, where $S$ is a linearly ordered set, such that for $\alpha, \beta \in \Gamma$
\begin{itemize}
\item[(VG1)] $v(\alpha)= \infty \iff \alpha =0$,
\item[(VG2)]$v(\alpha) = v(-\alpha)$,
\item[(VG3)] $v(\alpha + \beta) \geq \min\{v(\alpha),v(\beta)\}$.
\end{itemize}

\noindent
A \textbf{valued abelian group} is a triple $(\Gamma, S, v)$ where $\Gamma$ is an ordered abelian group, $S$ is an ordered set and $v:\Gamma \rightarrow S_\infty$. It is easy to show that if $v(\alpha)<v(\beta)$ then  $v(\alpha + \beta) = v(\alpha)$; indeed
\[v(\alpha)\ =\ v(\alpha+\beta-\beta)\ \geq\ \min(v(\alpha+\beta),v(\beta))\ \geq\ v(\alpha).\]
We note that given any element $s \in S$, the subsets 
\[B(s):=\{\gamma\ :\ v(\gamma) > s \} \text{ and } \overline{B}(s):=\{\gamma:v(\gamma) \geq s\}\] of $\Gamma$ form subgroups with the group operation inherited from $\Gamma$. We denote their quotient by $G(s):= \overline{B}(s)/B(s)$.   

\subsection*{Hahn products}

Let $(G_s)_s\in S$ be a family of nontrivial ordered abelian groups indexed by a totally ordered set $S$. Consider the product group $\prod_{s\in S} G_s$ whose elements are the sequences $g=(g_s)_{s\in S}$ such that $g_s\in G_s$ for each $s\in S$. For such $g$, we define the support of $g$ as the set $\{s\in S: g_s \neq 0\}$.  

\medskip\noindent
We define the \textbf{Hahn product} of the family to be the subgroup $H[(G_s)_s]$ of $\prod_{s\in S} G_s$ with well ordered support. We can construe $G:= H[(G_s)_s]$ as an ordered valued abelian group, by equipping it with the \textbf{Hahn valuation} $v:G\rightarrow S_\infty$ given by $v(g) = \min(\supp(g))$ for $g\in G\neq$ and $v(0) = \infty$, and the order given by $0<g$ iff $0<g_{s_0}$ for $s_0 = v(g)$.

\section{Hahn Series}\label{HS}

\noindent
By a {\bf monomial set\/} we mean an ordered set $\fM$ whose
elements are thought of as monomials and whose ordering is
denoted by $\preceq$ (or $\preceq_{\fM}$ if we need to
indicate the dependence on $\fM$); the corresponding strict ordering, reverse ordering, and strict reverse ordering are then
denoted by $\prec$, $\succeq$, and $\succ$, respectively. A subset of the monomial set $\fM$ is said to be {\bf well-based\/} if it is well-ordered with respect to the  reverse ordering $\succeq$, equivalently, there is no infinite sequence $\fm_0 \prec \fm_1 \prec \fm_2 \prec \dots$ in the subset.

 Let $\fM$ be a monomial set, let $\fm$ and $\fn$ range over $\fM$, and let $C$ be an (additively written) abelian group. By a 
{\bf Hahn series with coefficients in $C$ and monomials in $\fM$\/}
 we mean a function $f: \fM \to C$ whose support  
 \[\supp f\ :=\ \{\fm\in \fM:\ f(\fm)\ne 0\}\]
 is well-based as a subset of $\fM$.
We often write such $f$ as a series $\sum_{\fm} f_{\fm}\fm$,
where $f_{\fm}:=f(\fm)$. 
We define the {\bf Hahn space\/} $C[[\fM]]$ to be the (additive) abelian group whose elements are
the Hahn series with coefficients in $C$ and monomials in $\fM$,
with addition as suggested by the series notation.

Let $(f_i)_{i\in I}\in C[[\fM]]^I$. We say that $(f_i)$ 
is \textbf{summable} (or $\sum_i f_i$ exists) if 
\begin{itemize}
	\item $\bigcup_i \supp(f_i)$ is well-based; 
    \item for each $\fm\in \fM$ there are only finitely many $i$ with $\fm\in \supp(f_i)$.
\end{itemize}
If $(f_i)_{i\in I}$ is summable, we define its sum $f=\sum_{i\in I} f_i\in  C[[\fM]]$ by $f_{\fm}=\sum_{i}f_{i,\fm}$. If $I$ is finite, then $(f_i)$ is summable and $\sum_i f_i$ is the usual finite sum.
If $(f_i)$ and $(g_i)\in C[[\fM]]^I$ are summable, then so is $(f_i+g_i)$, and $\sum_i f_i+g_i=\sum_i f_i + \sum_i g_i$. If $I$ and 
$J$ are disjoint sets and $(f_i)\in C[[\fM]]^I$ and $(f_j)\in C[[\fM]]^J$ are summable, then $(f_k)_{k\in I\cup J}$ is summable
and $\sum_{k\in I\cup J}f_k=\sum_{i\in I} f_i + \sum_{j\in J} f_j$. 
Finally, if the
family $(f_{i,j})_{(i,j)\in I\times J}$ of elements $f_{i,j}\in C[[\fM]]$ is summable, then 
$\sum_j f_{i,j}$ exists for every $i$, and $\sum_i(\sum_j f_{i,j})$ exists, and
likewise $\sum_j(\sum_i f_{i,j})$ exists, and
\[\sum_{i,j} f_{i,j}\ =\ \sum_i\big(\sum_j f_{i,j}\big)\ =\ \sum_j\big(\sum_i f_{i,j}\big).\]

\medskip\noindent
It is convenient to augment $\fM$ by an element $0\notin \fM$ and
extend $\preceq$ to an ordering of $\fM\cup \{0\}$, also denoted by $\preceq$, by declaring 
$0\preceq \fm$ (and thus $0\prec \fm$) for all $\fm$. 
For $f\in C[[\fM]]$ we define its \textbf{leading monomial} $\fd(f)\in \fM\cup\{0\}$ by
\[\fd(f):= \max \supp f\ \text{ if }f\ne 0, \qquad 
\fd(f)=0\ \text{ if }f=0.\]
This allows us to introduce a transitive and reflexive binary relation $\preceq$ on $C[[\fM]]$ by $f\preceq g$ if and only if $\fd(f)\preceq \fd(g)$; replacing in this equivalence $\preceq$ by $\prec$, $\succeq$, and $\succ$, respectively, we introduce likewise transitive
binary relations $\prec$, $\succeq$, and $\succ$ on $C[[\fM]]$. 
We also define the equivalence relation $\asymp$ on $C[[\fM]]$ by
$f\asymp g:\Leftrightarrow\ \fd(f)=\fd(g)$.  

\subsection*{Hahn rings} By a {\bf monomial group\/} we mean a multiplicative ordered abelian group whose elements are 
thought of as monomials. Let $\fM$ be a monomial group; we use the same notations as for monomial sets; in particular, the
ordering is denoted by $\preceq$, and we let $\fm$ and $\fn$
range over $\fM$.  
We recall the following basic fact about well-based subsets of $\fM$:

\begin{lemma}\label{fGfN} If $\fG, \fH\subseteq \fM$ are well-based, then $\fG\fH$ is well-based, and for every $\fm$ there are only finitely many $(\fg,\fh)\in \fG\times\fH$ with $\fg\fh=\fm$.
\end{lemma}

\noindent
Next, let $C$ be
a commutative ring with $1\ne 0$. Lemma~\ref{fGfN} allows us to define
for $f,g\in C[[\fM]]$ their product $fg\in C[[\fM]]$ by
$(fg)_{\fh}:=\sum_{\fm\fn=\fh}f_{\fm}g_{\fn}$, so
\[\supp fg\ \subseteq\ (\supp f)(\supp g).\] With this product operation $C[[\fM]]$ is a commutative ring with subrings $C[[\fM^{\preceq 1}]]$ and $C[[\fM^{\succeq 1}]]$. We have the ring embedding of $C$ into $C[[\fM]]$ sending $c\in C$ to $f\in C[[\fM]]$ with $f_1=c$ and $f_{\fm}=0$
for $\fm\ne 1$; we identify $C$ with a subring 
of $C[[\fM]]$ via this embedding. We also have the group embedding of 
$\fM$ into the
multiplicative group of units of $C[[\fM]]$ sending $\fm$ to
$f\in C[[\fM]]$ with $f_{\fm}=1$, and $f_{\fn}=0$ for 
$\fn\ne \fm$; we identify $\fM$ with a subgroup of this multiplicative group of units via this embedding, and identify the element $0$ of $\fM\cup \{0\}$ with the zero element of the ring $C[[\fM]]$. This has the effect that restricting the binary relation
$\preceq$ on $C[[\fM]]$ to $\fM\cup \{0\}$ gives back the originally given ordering $\preceq$ on $\fM\cup\{0\}$, and
likewise with $\prec$, $\succeq$, and $\succ$. Note that $C[[\fM]]^{\prec 1}=C[[\fM^{\prec 1}]]$.  

Here is another key fact about well-based subsets of $\fM$:

\begin{lemma}\label{fGstar} Let $\fG$ be a well-based subset of $\fM^{\prec 1}$. Then:
\begin{enumerate}
	\item[\rm{(i)}] $\fG^*:=\bigcup_n \fG^n$ is well-based;
	\item[\rm{(ii)}] for every $\fm$ there are only finitely many
tuples $(n,\fg_1,\dots, \fg_n)$ such that $\fg_1,\dots, \fg_n\in \fG$ and $\fm=\fg_1\cdots \fg_n$. 
\end{enumerate}
\end{lemma}

\noindent
For $\varepsilon\in C[[\fM]]^{\prec 1}$ the family $(\varepsilon^n)$ is summable and $(1-\varepsilon)\sum_n\varepsilon^n=1$, by Lemma~\ref{fGstar}. More generally, 
let $t=(t_1,\dots, t_n)$ be a tuple of distinct variables and let
\[F\ =\ F(t)\ =\ \sum_{\nu}c_{\nu}t^{\nu}\in C[[t]]\ :=\ C[[t_1,\dots, t_n]]\]
be a formal power series over $C$; here the sum ranges over all multiindices $\nu=(\nu_1,\dots, \nu_n)\in \N^n$, and
$c_{\nu}\in C$, $t^{\nu}:= t_1^{\nu_1}\cdots t_n^{\nu_n}$. For any tuple $\varepsilon=(\varepsilon_1,\dots, \varepsilon_n)$ of
elements of $C[[\fM]]^{\prec 1}$ the family $(c_{\nu}\epsilon^{\nu})$ is summable, where $\varepsilon^{\nu}:= \varepsilon_1^{\nu_1}\cdots \varepsilon_n^{\nu_n}$. Put
\[F(\varepsilon)\ :=\ \ \sum_{\nu}c_{\nu}\varepsilon^{\nu}\in C[[\fM]]^{\preceq 1}\ =\ C[[\fM^{\preceq 1}]].\]
Fixing $\varepsilon$ and varying $F$ we obtain a $C$-algebra morphism
\[F \mapsto F(\varepsilon)\ :\ C[[t]] \to C[[\fM]].\] 
Assume that $C$ contains $\mathbb{Q}$ as a subring.
Then we have the formal power series
\[ \exp(t)\ :=\ \sum_{i=0}^\infty t^i/i!\in \mathbb{Q}[[t]], \qquad \log (1+t)\ :=\ \sum_{j=1}^\infty (-1)^{j-1}t^j/j\in \mathbb{Q}[[t]]\]
with $\exp(t_1+t_2)=\exp(t_1)\exp(t_2)$ and $\log(1+t_1+t_2+t_1t_2) =\log(1+t_1)+\log(1+t_2)$
in  $\mathbb{Q}[[t_1,t_2]]\subseteq C[[t_1,t_2]]$. Also
$\log\big(\exp(t)\big) =t$ and 
$\exp\big(\log(1+t))= 1+t$ 
in $\mathbb{Q}[[t]]\subseteq C[[t]]$. Substituting elements of $C[[\fM]]^{\prec 1}$ in these identities yields that 
\[\delta\ \mapsto\ \exp(\delta)\ =\ \sum_{i=0}^\infty \delta^i/i!\ :\ C[[\fM]]^{\prec 1} \to 1+C[[\fM]]^{\prec 1},\]
is an isomorphism of the additive subgroup $C[[\fM]]^{\prec 1}$ of $C[[\fM]]$
onto the multiplicative subgroup $1+C[[\fM]]^{\prec 1}$ of $C[[\fM]]^\times$, with inverse
\[1+\varepsilon\ \mapsto\ \log (1+\varepsilon)\ =\ \sum_{j=1}^\infty (-1)^{j-1}\varepsilon^j/j\ :\ 1+C[[\fM]]^{\prec 1}\to C[[\fM]]^{\prec 1}.\]

\subsection*{Hahn fields} We now consider a coefficient field 
$\smallk$ instead of just a coefficient ring $C$ as before.
 For
nonzero $f\in \smallk[[\fM]]$ we have $f=a\fd(f)(1-\varepsilon)$ with $a\in \smallk^\times$ and $\varepsilon \prec 1$, so $f$ has a multiplicative inverse in $\smallk[[\fM]]$, namely $a^{-1}\fd(f)^{-1}\sum_n \varepsilon^n$. In particular,
$\smallk[[\fM]]$ is again a field with $\smallk$ as a subfield and
$\fM$ as a subgroup of its multiplicative group 
$\smallk[[\fM]]^\times$.  

Suppose $\smallk$ is an ordered field. Then we construe
$\smallk[[\fM]]$ as an ordered field extension of $\smallk$ by requiring that for nonzero $f\in \smallk[[\fM]]$ with $\fm=\fd(f)$ we have $f>0 \Leftrightarrow f_{\fm}>0$;
in this role as ordered field we call $\smallk[[\fM]]$ an \textbf{ordered Hahn field} (over $\smallk$).

\medskip\noindent
Often we prefer
additive notation. To explain this, let
$\Gamma$ be an additively written ordered abelian group (with zero element $0_\Gamma$ if we need to indicate the dependence on $\Gamma$); then $\le$ rather than $\preceq$
denotes the ordering of $\Gamma$.  When $\Gamma$ is clear from the context, we let $\alpha, \beta, \gamma$ range over $\Gamma$. 
We now have the monomial group 
\[\fM\ =\ t^\Gamma\]
where $t$ is just a symbol, and
$\gamma\mapsto t^\gamma: \Gamma\to t^\Gamma=\fM$ is an {\bf order-reversing\/} group
isomorphism of $\Gamma$ onto $\fM$. With this $\fM$ we denote $\smallk[[\fM]]$
also by $\smallk((t^\Gamma))$, and write 
$f=\sum_{\fm} f_{\fm}\fm\in \smallk[[\fM]]$ as
$f=\sum_\gamma f_{\gamma}t^\gamma$, with $f_\gamma:= f_{\fm}$ for $\fm=t^\gamma$.
In this situation we prefer to take $\supp f$ as a subset of $\Gamma$ 
rather than of $\fM=t^\Gamma$: $\supp f =\{\gamma:\ f_{\gamma}\ne 0\}$, and the
well-based requirement turns into the requirement that
$\supp f$ is a well-ordered subset of $\Gamma$. The order type of the support of $f$ is by definition the unique ordinal isomorphic to $\supp(f)$ with respect to the ordering induced by $\Gamma$, and will be denoted by $o(f)$. Of course, all this is only 
a matter of notation, and we shall freely apply results for
Hahn fields $\smallk((t^\Gamma))$ to Hahn fields $\smallk[[\fM]]$, since we 
can pretend that any monomial group $\fM$ is of the form 
$t^\Gamma$ for a $\Gamma$ as above.  Thus for $f$ in a Hahn field $\smallk[[\fM]]$, the ordinal $o(f)$ is isomorphic to $\supp(f)\subseteq \fM$ with respect
to the ordering induced by the {\em reverse\/} ordering on $\fM$.

\section{Valued Fields}
 
 \subsection*{Valuation rings and valued fields}
 A \textbf{valuation ring} is an integral domain in which the set of ideals is totally ordered by inclusion. Here are a few characterizations of valuation rings:
 
\begin{lemma} Let $\bigO$ be an integral domain. The following are equivalent:
\begin{enumerate}
\item[\rm{(i)}] The ring $\bigO$ is a valuation ring.
\item[\rm{(ii)}] For every element $x\neq 0$ in the field of fractions $K$ of $\bigO$, either $x$ or $x\inv$ is in $\bigO$. 
\item[\rm{(iii)}] For all $a,b \in \bigO$, either $b\in a\bigO$ or $a\in b\bigO$.
\end{enumerate}
\end{lemma} 

\medskip\noindent The set of non-units of a  valuation ring is an ideal. In fact, it is the unique maximal ideal of the valuation ring. Thus any valuation ring $\bigO$ is a local ring.

\medskip\noindent A valued field is a field $K$ equipped with a valuation ring $\bigO$ such that $K$ is the field of fractions of $\bigO$. 
The following are examples of valued fields:
\begin{itemize}
\item Any field $K$ with $K$ itself as the valuation ring. In this case we say that the valuation of $K$ is trivial.
\item The field of Laurent series in $x\inv$ over $\mathbb{C}$ with valuation ring 
\[\bigO\ :=\ \left\{f=\sum_{n\in \mathbb{N}}f_n x^{-n}\right\}\]
\item In fact, any Hahn field $\smallk((t^\Gamma))$ equipped with 
\[\bigO\ :=\ \{f\in \smallk((t^\Gamma)) :\ \supp(f) \subseteq \Gamma^{\geq}\}\].
\item Any Hausdorff field $K$ (a field of germs at $+\infty$ of continuous real functions) together with 
\[\bigO\ :=\ \left\{f\in K:\ f \text{ is eventually bounded}\right\}\]
\item The field $\Q$, equipped with the valuation ring
\[\bigO\ :=\ \left\{\frac{a}{b}:\ a,b\in \Z, \text{ and } b\notin p\Z\right\}\]
where $p$ is a fixed prime number. 
\end{itemize}

\subsection*{Residue fields} Given a local ring $R$ with maximal ideal $\smallo$ we define the \textbf{residue field of} $R$ to be the quotient $\smallk = \bigO/\smallo$. We call the quotient map $\bigO \rightarrow \smallk$ the residue map. For a valued field $K$, we define the \textbf{residue field of} $K$ to be the residue field of its valuation ring. The residue fields corresponding to the valued fields mentioned above are naturally isomorphic to $K$, $\C$, $\smallk$, $\R$, and $\Z/p\Z$ respectively. 

\medskip\noindent It is sometimes useful to distinguish valued fields by the characteristic of both itself and its residue field. We say that $K$ is of \textbf{equicharacteristic zero}, \textbf{equicharacteristic} $p$, or of \textbf{mixed characteristic} if $(\cha(K),\cha(\smallk))$ is $(0,0)$, $(p,p)$, or $(0,p)$, respectively, where $\smallk$ is the residue field of $K$ and $p$ is a prime number.   In this dissertation we are mainly concerned with the equicharacteristic zero case, especially when there is also a derivation in play.

\medskip\noindent We say that the residue field $\smallk$ of $K$ lifts to $K$ if there is a field embedding $\iota:\smallk\rightarrow K$. All the previously mentioned residue fields lift to their respective valued fields except for $\Z/p\Z$ to $\mathbb{Q}$. In fact there is no hope of lifting the residue field in the mixed characteristic case.

\subsection*{Dominance relations}
A \textbf{dominance relation} on a field $K$ is a binary relation $\preceq$ on $K$ such that for all $f,g,h\in K$:
\begin{itemize}
	\item[(DR1)] $1\not\preceq 0$,
    \item[(DR2)] $f\preceq f$,
    \item[(DR3)] $f\preceq g$,  $g\preceq h$ $\Rightarrow$ $f\preceq h$,
    \item[(DR4)] $f\preceq g$ or $g\preceq f$,
    \item[(DR5)] $f\preceq g$ $\Rightarrow $ $hf\preceq hg$,
    \item[(DR6)] $f\preceq h$, $g\preceq h$ $\Rightarrow$ $f+g\preceq h$.
\end{itemize}

\medskip\noindent
Let $\preceq$ be a dominance relation on the field $K$.  We define some further (asymptotic) binary  relations
$\succeq, \prec, \succ, \asymp, \sim$ on $K$ in terms of $\preceq$ as follows:
\begin{align*} f \succeq g &:\iff g\preceq f; \qquad f \prec g :\iff f\preceq g \text{ and }g\not\preceq f; \qquad f\succ g :\iff g\prec f;\\
 f\asymp g &:\iff f\preceq g \text { and }g\preceq f; \qquad f\sim g :\iff f-g\prec f.
 \end{align*}
We also say that $f$ \textbf{dominates} $g$ if $f\succeq g$, that $f$ \textbf{strictly dominates} $g$ if $f\succ g$, and that $f$ is \textbf{asymptotic} to $g$ if $f\asymp g$. The relation $\asymp$ is an equivalence relation on $K$; reflexivity and symmetry are straightforward from the definition and transitivity follows from DR3 and symmetry. 
Alternatively we note that by DR2 and DR3 we have that $\preceq$ is a preorder on $K$. If $f\sim g$, then $f,g\ne 0$. It is easy to check that
$\sim$ is an equivalence relation on $K^\times$. 

\subsection*{Valuations}
A \textbf{valuation} on a field $K$ is a map $v:K^\times \rightarrow \Gamma$ onto an ordered abelian group $\Gamma$ such that for all $x,y\in K^\times$,
\begin{itemize}
	\item[(VF1)] $v(x+y)\geq \min\{v(x),v(y)\}$ for $x+y\neq 0$, and
    \item[(VF2)] $v(xy) = v(x) + v(y)$.
\end{itemize}

\noindent
As for valued abelian groups it follows that if $v(x)< v(y)$, then $v(x+y) = v(x)$. We extend $+$ and $<$ to a binary operation $+$ on $\Gamma_\infty= \Gamma \cup \{\infty\}$ and a (total) ordering $<$ on $\Gamma_\infty$ so that for all $\gamma\in \Gamma$
\begin{itemize}
	\item $\infty > \gamma$ ,
    \item $\infty + \gamma = \gamma + \infty =  \infty$, and
    \item $\infty + \infty = \infty$,
\end{itemize} 
an extend $v$ to $v:K\rightarrow \Gamma_\infty$ by $v(0)=\infty$ all of $K$ by setting $v(0)= \infty$ so that $v$ is a valuation on the additive group of $K$ as defined on Section \ref{VAG}.

\begin{proposition}\label{propeqval} The following are equivalent for a subring $\bigO$ of a field $K$:
\begin{enumerate}
	\item[\rm{(i)}] $K$ equipped with $\bigO$ is a valued field; 
    \item[\rm{(ii)}] the equivalence $f\preceq g :\iff f\in \bigO g$ defines a dominance relation $\preceq$ on $K$;
    \item[\rm{(iii)}] there is a valuation $v:K\rightarrow \Gamma_{\infty}$ on $K$ such that $ \bigO=\{f\in K:\ v(f)\ge 0\}$.
\end{enumerate}
\end{proposition}
\begin{proof}
Assume (i) and define the binary relation $\preceq$ by:  $f\preceq g: \iff f\in \bigO g$. Then DR1-DR3, DR5, DR6 follow from $\bigO$ being a ring with $1\ne 0$, and DR4 holds since for $f\neq 0\neq g$, either $f/g \in \bigO$ or $g/f\in \bigO$, so $\preceq$ is a dominance relation.  
Assume (ii), and  let $v:K^\times \rightarrow \Gamma:=K^\times\hspace{-1 mm}/\hspace{-1 mm}\asymp$ be the map sending $f\in K^\times$ to its equivalence class
$f_{\asymp}$, and consider the quotient group $\Gamma$ of the multiplicative group $K^\times$ as an additive group. (This is just a notational convention.) 
Equip $\Gamma$ with the reverse order: $f_{\asymp}\leq g_{\asymp}$ iff $g\preceq f$. It is easy to check  then that $v$ is a valuation on the field $K$, with $\bigO=\{f\in K:\ v(f)\ge 0\}$.  Assuming $v$ is as in (iii) we have $v(f)\ge 0$ or $v(f^{-1})\ge 0$ for all $f\in K^\times$, so $\bigO$ is a valuation ring of $K$.  
\end{proof}

\noindent
Let $K$ be a valued field.  By Proposition~\ref{propeqval}
and its proof we have a valuation $v: K^\times\to \Gamma$ such that $\bigO=\{f\in K:\ v(f)\ge 0\}$, and such that if $v^*: K^\times\to \Gamma^*$ is also a valuation with 
$\bigO=\{f\in K:\ v^*(f)\ge 0\}$, then $v^*=\phi\circ v$ for a unique ordered group isomorphism $\phi: \Gamma\to \Gamma^*$. We equip $K$ with 
such a valuation $v$; by the fact just mentioned, all ways of doing this are equivalent. The dominance relation $\preceq$
on $K$ defined in (ii) of Proposition~\ref{propeqval} is then given in terms of $v$ by: $f\preceq g \iff v(f)\ge v(g)$; likewise: $f\prec g \iff v(f)> v(g)$, $f\asymp g \iff v(f)=v(g) $, and $f\sim g \iff  v(f-g)>v(f)$.
Given $S\subseteq K$ and $f\in K$ we set $S^{\preceq f}:=\{g\in K:\ g\preceq f\}$, and similarly with $S^{ \prec f}$, and the like. 

We make $K$ into a {\em topological field\/} by giving it the {\bf valuation topology}: a neighborhood basis of $0\in K$ is given by the sets $\{f\in K: v(f)>\gamma\}$ with $\gamma\in \Gamma$.  

\medskip\noindent We always consider the Hahn field $\smallk((t^\Gamma))$ to be equipped with the  valuation 
\[v\ :=\ v_\Gamma:\smallk((t^\Gamma))^\times\rightarrow \Gamma, \qquad v(f)\ :=\ \min \supp f.\] The dominance relation $\preceq$ associated to this valuation
then agrees with the relation $\preceq$ that we defined earlier on this Hahn field in terms of  leading monomials. Likewise for $\prec, \succ$, and $\asymp$.

A Hahn field $\smallk[[\fM]]$ is accordingly considered equipped with a valuation $v: \smallk[[\fM]]^\times \to \Gamma_{\fM}$
where the restriction of $v$ to the ordered multiplicative group $\fM\subseteq \smallk[[\fM]]^\times$ is an order-reversing group isomorphism onto the ordered (additive) group $\Gamma_{\fM}$. 

A \textbf{monomial group of} a valued field $K$ is a multiplicative subgroup $\fM$ of $K^\times$ that is bijectively mapped onto the value group
$\Gamma$ by $v$. Thus a Hahn field $\smallk[[\fM]]$ (with the usual valuation) has $\fM$ as a monomial group. 

\begin{example} The same valued field can have many monomial groups: For example, 
$K = \R[[x^\Z]]$ has the set $x^\Z$ as a monomial group, but  has also for every $r\in \R^\times$ the set $\{r^kx^k:\ k\in \Z\}$ as a monomial group. 
\end{example}

\subsection*{Valued field extensions} Let $K$ and $L$ be valued fields. We say that $L$ is \textbf{valued field extension of} $K$  (or $K$ is a \textbf{valued subfield of} $L$) if $L$ is a field extension such that $\bigO_L\cap K = \bigO$. (When two valued fields $K$ and $L$ are in play, it is understood that $\bigO$ is the valuation ring of $K$ and $\bigO_L$ is the valuation ring of $L$.) 
Let $L$ be a valued field extension of $K$. Then we have a natural field embedding $\smallk=\bigO/\smallo\to \smallk_L=\bigO_L/\smallo_L$ of the residue fields, and we identify
$\smallk$ with a subfield of $\smallk_L$ via this embedding. Likewise, we have a natural ordered group embedding $\Gamma \to \Gamma_L$ of the value groups, and we identify
$\Gamma$ with an ordered subgroup of $\Gamma_L$ via this embedding. This extension is said to be \textbf{immediate} if $\smallk=\smallk_L$ and $\Gamma=\Gamma_L$.
A valued field $K$ is \textbf{maximal} if it has no proper immediate valued field extension. A valued field is \textbf{algebraically maximal} if it has no proper immediate algebraic valued field extension. Hahn fields $\smallk[[\fM]$ are maximal, and every valued field has a maximal immediate extension.
\subsection*{Henselian fields}
A valued field $K$ is said to be \textbf{henselian} if for every polynomial $P\in \bigO[X]$ and every $a\in \bigO$ with $P(a)\prec 1$ and $(dP/dX) (a) \asymp 1$ there exists $b\in \bigO$ such that $b-a\prec 1$ and $P(b) = 0$. Hahn fields are henselian. Here are some useful characterizations of henselianity:

\begin{lemma}Let $K$ be a valued field. The following are equivalent:
\begin{enumerate}
	\item[\rm{(i)}] $K$ is henselian.
    \item[\rm{(ii)}] For every $P(X) = 1 + X + c_2X^2 +\cdots + c_nX^n \in K[X]$ with $c_2,\ldots, c_n \prec 1$, there is $x\in \bigO$ such that $P(x) = 0$.
    \item[\rm{(iii)}] For every $Q(Y) = Y^n + Y^{n-1} + c_2Y^{n-2} + \cdots + c_n\in K[Y]$ with $c_2, \ldots, c_n\prec 1$, there is $y\in \bigO$ such that $Q(y) = 0$.
    \item[\rm{(iv)}] For every $P(X) \in \bigO[X]$ and $a\in \bigO$ with $P(a) \prec (dP/dX)(a)^2$, there is $b\in \bigO$ such that $a-b\prec (dP/dX)(a)$ and $P(b)=0$.
\end{enumerate}
\end{lemma}

\begin{lemma} If a valued field is algebraically maximal, then it is henselian.
The converse holds in the equicharacteristic zero case.
\end{lemma}

\noindent
A \textbf{henselization} of a valued field $K$ is a valued field extension of $K$ that embeds uniquely over $K$ into every henselian valued field extension of $K$. 

\begin{lemma}
Every valued field $K$ has a henselization. Such a henselization is unique up to unique isomorphism over $K$, and is an immediate
algebraic valued field extension of $K$. 
\end{lemma}

\noindent
Let $E$ be a subfield of a Hahn field $\smallk[[\fM]]$. We consider $E$ as a valued subfield of $\smallk[[\fM]]$. Since 
$\smallk[[\fM]]$ is henselian, there is a unique embedding
$E^h\to \smallk[[\fM]]$ over $E$, and the image of this embedding is called the \textbf{henselization of $E$ inside $\smallk[[\fM]]$}. 

\section{Differential Rings}
\noindent
Let $R$ be a commutative ring.  A \textbf{derivation} on 
$R$ is an additive map  $\der: R \to R$
such that for all $a,b\in R$ we have  $\der(ab) = \der(a)b + a\der(b)$ (Leibniz rule).
A \textbf{differential ring} is a ring together with a derivation on it.

Let $R$ be a differential ring. Unless we specify otherwise, we let
$\der$ be the derivation of $R$, and for $a\in R$ we let $a'$ denote 
$\der(a)$. The 
\textbf{ ring of constants of $R$} is the subring 
\[C_R\ :=\ \{a\in R:\ a'=0\}\]
of $R$. A \textbf{differential ring extension of} $R$ is a differential ring $R^*$ that has $R$ as a subring and whose derivation extends the derivation on $R$; in that case we also call $R$ a \textbf{differential subring of} $R^*$. 
Let $R^*$ be a differential ring extension of $R$.  If $A$ is a subset of $R^*$ and $a'\in R[A]$ for all $a\in A$, then the subring $R[A]$ of $R^*$ is closed under the derivation of $R^*$, and is accordingly viewed as a differential subring of $R^*$. Given $a\in R^*$, we let $R\{a\}:= R[a,a',a'',\ldots]$ denote the smallest differential
subring of $R^*$ that contains $R$ and $a$, and we call $R\{a\}$  the \textbf{differential ring generated by} $a$ \textbf{over} $R$.

\medskip\noindent
The differential ring $R\{Y\}$ of differential
polynomials in the indeterminate $Y$ over $R$ is defined as follows: as a ring, this is just the
polynomial ring 
\[R[Y, Y', Y'',\dots]\ =\ R[Y^{(n)}:n=0,1,2,\dots]\]
in the distinct 
indeterminates $Y^{(n)}$ over $R$, and it is made into a differential ring extension of $R$ by requiring that $(Y^{(n)})'= Y^{(n+1)}$ for all $n$.  If $R$ is a domain, then $R\{Y\}$ is also a domain.

A \textbf{differential field} is a differential ring whose underlying ring is a field of characteristic zero; then the ring of constants is also a field.
Given a differential field denoted by $K$ we generally denote its field of constants by $C$, while if a differential field is denoted by another letter, say $L$, we let $C_L$ be its field of constants. If $R$ is a differential ring that is also an integral domain
of characteristic zero, then its ring of fractions is made into a differential field by setting 
\[\der\left(\frac{a}{b}\right)\ =\ \frac{\der(a)b - a\der(b)}{b^2}\ \text{ for } a,b\in R,\ b\ne 0.  \]
Let $K$ be a differential field. Then for $f\in K^\times$ we let
 $f^\dagger$ denote the logarithmic derivative $\der(f)/f$, so
$(fg)^\dagger=f^\dagger + g^\dagger$ for $f,g\in K^\times$.  
For a subset $S$ of $K^\times$ we
set $S^\dagger:=\{f^\dagger:\ f\in S\}$. Also $K^\dagger:=(K^\times)^\dagger$, an additive subgroup of $K$. 
A \textbf{differential field extension of $K$} is a differential ring extension $L$ of $K$ that is also a field. 
Let $L$ be a differential field extension of $K$ and $a\in L$.
We let $K\<a\>$ be the field of fractions of $K\{a\}$ inside $L$ and we call it the \textbf{differential field generated by $a$ over $K$}. We say that $a$ is \textbf{differentially-algebraic over $K$}, abbreviated $\d$-algebraic over $K$, if $P(a) = 0$ for some nonzero differential polynomial $P(Y)$ over $K$. If $a$ is not $\d$-algebraic over $K$ we say that $a$ is \textbf{differentially transcendental}, or $\d$-transcendental, over $K$. The field extension $L$ is \textbf{$\d$-algebraic over $K$} if every element in $L$ is $\d$-algebraic over $K$. Let $E\subseteq L$. We let $K\<E\>$ be the differential field generated by $E$ over $K$, that is, the smallest differential subfield of $L$ containing $K$ and $E$. We say that $E$ is \textbf{$\d$-algebraically-independent over $K$} if  $a$ is $\d$-transcendental over $K\<E\setminus\{a\}\>$
for all $a\in E$. The subset $E$ of $L$ is a \textbf{$\d$-transcendence basis of $L$ over $K$} if $E$ is $\d$-algebraically independent over $K$ and $L$ is $\d$-algebraic over $K\<E\>$. All $\d$-transcendence bases of $L$ over $K$ have the same cardinality, and we call this cardinality the \textbf{differential-transcendence degree}, or $\d$-transcendence degree, of $L$ over $K$.

\section{Valued Differential Fields}

\noindent
A \textbf{valued differential field} is a differential field $K$ equipped with a valuation ring $\bigO \supseteq \Q$ of $K$. In this section we fix a valued differential field $K$; as usual we denote its valuation ring by $\bigO$ and its residue field $\bigO/\smallo$ by $\smallk$. The next result is part of \cite[Lemma 4.4.2]{ADH}. 

\begin{lemma} Assume that $\der(\smallo) \subseteq \smallo$. Then $\der(\bigO)\subseteq \bigO$.
\end{lemma}
\begin{proof} 
Let $x\in \bigO$ and suppose that $x'\notin \bigO$. Then $1/x'\in \smallo$, so $x/x'\in \smallo$, $(1/x')'\in \smallo$, hence $(x/x')'\in \smallo$ . But $(x/x')' = x(1/x')' + 1\in 1+\smallo$, a contradiction.
\end{proof}
\noindent
If $\der(\smallo)\subseteq \smallo$, we equip $\smallk$ with the derivation  $a+\smallo \mapsto \der(a) + \smallo:\bigO/\smallo\rightarrow \bigO/\smallo$. The condition
 $\der(\smallo) \subseteq \smallo$ means that derivatives of ``small'' elements (elements of $\smallo$) are small, so we say that $\der$ is {\bf small} (or that $K$ has \textbf{small derivation}) if this condition is satisfied. 
If $K$ has small derivation, then we refer to $\smallk$ with the induced derivation as the \textbf{differential residue field} of $K$. 

\medskip\noindent
\subsection*{Continuous derivations} We now equip $K$ with the valuation topology. Even though it is not part of the definition of ``valued differential field'',  the
only case of interest to us is when the derivation is continuous. If the derivation on $K$ is continuous, then for every $P\in K\{Y\}$ the function $y\mapsto P(y): K\rightarrow K$ is continuous. 

If $\der$ is small, then $\der$ is continuous. In fact, continuous derivations are just  ``small derivations in disguise'' as shown by the next  result  (\cite[Lemma 4.4.7]{ADH}): 

\begin{lemma}
The following conditions on $K$ are equivalent:
\begin{enumerate}
\item[\rm{(i)}] $\der:K\rightarrow K$ is continuous;
\item[\rm{(ii)}] for some $a\in K^\times$ we have $\der\smallo \subseteq a\smallo$;
\item[\rm{(iii)}] for some $a\in K^\times$ the derivation $\derdelta:=  a^{-1}\der$ is small.
\end{enumerate}
\end{lemma}

\subsection*{Asymptotic fields and differential-valued fields} The material from this subsection can be found in more detail in \cite[Section 9.1]{ADH}. Let $K$ be a valued differential field. We say that $K$ is \textbf{asymptotic} if for all $f,g\in K^\times$ with $f,g\prec 1$,
\[f\prec g\  \Longleftrightarrow\  f'\prec g'.\]
The derivation of an asymptotic field is continuous (with respect to the valuation topology). We say that $K$ is a \textbf{differential-valued field}, or just \textbf{$\d$-valued}, if $K$ is asymptotic and $\bigO = C +\smallo$. If $K$ is $\d$-valued, then its constant field is naturally isomorphic to its residue field via the residue map. For the next lemma, see for example \cite[Proposition 9.1.3]{ADH}.

\begin{lemma}
The following are equivalent:
\begin{enumerate}
\item[\rm{(i)}] $K$ is $\d$-valued;
\item[\rm{(ii)}] $\bigO = C +\smallo$ and for all $f,g\in K^\times$, $f,g\prec 1\Rightarrow f^\dagger\succ g'$.
\end{enumerate}
\end{lemma}

\subsection*{H-fields}  For more details about $H$-fields, see \cite[Chapter 10]{ADH}. An $H$-field is by definition an ordered differential field $K$ such that:
\begin{enumerate}
	\item[\rm{(i)}] for all $f\in K$, if $f>C$, then $f'>0$;
    \item[\rm{(ii)}] $\bigO=C+\smallo$, where $\bigO = \{g\in K\ :\ |g|\leq c\text{ for some } c\in C\}$ and $\smallo$ is the maximal ideal of the convex subring $\bigO$ of $K$. 
\end{enumerate}

\noindent
An $H$-field $K$ is considered also as a valued differential field by taking the above $\bigO$ as the valuation ring. This makes $H$-fields into differential-valued fields. 
If $K$ is an $H$-field with $C\ne K$, then its valuation topology coincides with its order topology.
Hardy fields containing $\R$ as a subfield are $H$-fields. 
The $H$ in ``$H$-field'' is in honor of the pioneers Borel, Hahn, Hardy, and Hausdorff, three of which share the initial $H$.

An $H$-field $K$ is said to be \textbf{Liouville closed} if it is a real closed $H$-field and for all $a\in K$ there exist $y,z\in K^\times$ such that $y'=a$ and $z^\dagger = a$. 
An $H$-field $K$ is Liouville closed if and only if $K$ is real closed and every equation  $y'+ay=b$ with $a,b\in K$ has a nonzero solution in $K$.

\chapter{Truncation}
\noindent
In this Chapter we formally introduce the notion of truncation and state some well known results that showcase how the notion of a truncation closed subset of a Hahn field is robust in the sense that it is preserved under several extension procedures. In Section 2.2 we introduce the notion of an \textit{infinite part of a valued field} as a first-order counterpart to truncation. It allows us to extend the notion of truncation outside of the Hahn field setting. We also
consider some cases with archimedean value group where this first-order version of truncation necessarily agrees with truncation in an ambient Hahn field $C[[\fM]]$.

\section{Truncation in Hahn Fields}
\noindent
In this section $\fM$ is a monomial group.
For $f=\sum_{\fm}f_{\fm}\fm\in \smallk[[\fM]]$
and $\fn\in \fM$, the \textbf{truncation  $f|_{\fn}$ of $f$ at $\fn$} is defined by
\[ f|_{\fn}\ :=\ \sum_{\fm\succ \fn}f_{\fm}\fm, \quad\text{an element of } \smallk[[\fM].\]
Thus for $f,g\in \smallk[[\fM]]$ we have $(f+g)|_{\fn} = f|_{_\fn} +g|_{\fn}$. 
A subset $S$ of $\smallk[[\fM]]$ is said to be \textbf{truncation closed} if for all $f\in S$ and $\fn\in \fM$ we have $f|_{\fn} \in S$. For 
$f,g\in \smallk[[\fM]]$ we let $f\trunceq g$ mean that $f$ is a truncation of 
$g$, and let $f\truncof g$ mean that $f$ is a proper truncation of $g$,
that is, $f\trunceq g$ and $f\ne g$. 

\medskip\noindent
When, as in \cite{D}, Hahn fields are given as $\smallk((t^\Gamma))$, 
we adapt our notation accordingly:
the \textbf{truncation of $f =\ser{f}\in \smallk((t^\Gamma))$ at $\gamma_0\in \Gamma$} is then
\[f|_{\gamma_0}\ :=\ \sum_{\gamma<\gamma_0}f_\gamma t^\gamma.\]

\noindent
We say that $g$ is a \textbf{truncation of} $f$ if there is $\gamma \in \Gamma$ such that $g=f|_\gamma$ and write $g \trunceq f$. We say that $g$ is a \textbf{proper truncation of} $f$ if there is $\gamma \in \supp(f)$ such that $g= f|_\gamma$ and we write $g\truncof f$. Note that $g\truncof f$ if and only if $g \trunceq f$ and $g\neq f$.
We have the following.
\begin{enumerate}
	\item $\smallk((t^\Gamma))$ is truncation closed,
    \item $\emptyset$ is truncation closed,
    \item $\smallk$ is truncation closed,
    \item $\{f\in K :\ \supp(f) \text{ is finite}\}$ is truncation closed,
    \item $\smallk[t^\gamma]$ is truncation closed for any $\gamma$,
    \item $\smallk[t^{\gamma_0} +t^{\gamma}]$ is not truncation closed, for all distinct $\gamma_0, \gamma \in \Gamma^{\neq}$.
\end{enumerate}

We only justify (6), as (1)--(5) are immediate. Assume $\gamma_0, \gamma\in \Gamma^{\ne}$ and $\gamma_0<\gamma$.  To show that $\smallk[t^{\gamma_0}+t^\gamma]$ is not truncation closed, it is enough
to show that for all  $P(X)\in \smallk[X]$ we have  $P(t^{\gamma_0} + t^\gamma)\neq t^{\gamma_0}$. If $P\in \smallk$, then clearly $P(t^{\gamma_0} + t^{\gamma})=P\ne t^{\gamma_0}$. It is also clear that if $\deg(P)=1$, then $P(t^{\gamma_0}+t^\gamma)\ne t^{\gamma_0}$.  If $\deg (P)=n>0$ and $\gamma>0$, then the highest exponent of $t$ in $P(t^{\gamma_0}+t^\gamma)$ is $n\gamma\ne \gamma_0$ and thus $P(t^{\gamma_0}+t^\gamma)\neq t^{\gamma_0}$. If $\deg(P)=n>1$ and $\gamma<0$, then the lowest exponent of $t$ in $P(t^{\gamma_0}+t^\gamma)$ is $n\gamma_0\ne \gamma_0$ and thus
$P(t^{\gamma_0}+t^\gamma)\ne t^{\gamma_0}$.  

\medskip\noindent
Here is an easy but useful consequence of the definition:

\begin{lemma} If $V$ is a truncation closed $\smallk$-linear subspace of $\smallk[[\fM]]$, then $\supp(V)=V\cap\fM$. 
\end{lemma}

\noindent
We now list some items from \cite{D} that we are going to use:
 
\begin{proposition} \label{D} Let $A$ be a subset of $\smallk[[\fM]]$, $R$ a subring 
of $\smallk[[\fM]]$, and $E$ a $($valued$)$ subfield of $\smallk[[\fM]]$. Then:
\begin{itemize}
	\item[(i)] The ring as well as the field generated in $\smallk[[\fM]]$ by any truncation closed
subset of $\smallk[[\fM]]$ is truncation closed: \cite[Theorem 1.1]{D};
    \item[(ii)] If $R$ is truncation closed and all truncations of all $a\in A$ lie in $R[A]$, then the ring $R[A]$ is truncation closed;
    \item[(iii)] If $E$ is truncation closed and $E\supseteq \smallk$, 
then the henselization $E^{\h}$ of $E$ in 
$\smallk[[\fM]]$ is truncation closed: \cite[Theorem 1.2]{D};
    \item[(iv)] If $E$ is truncation closed, henselian, and $\operatorname{char}(\smallk)=0$, then any algebraic field extension of $E$ in $\smallk[[\fM]]$ is 
truncation closed: \cite[Theorem 5.1]{D}. 
\end{itemize}
\end{proposition}

\noindent
Item (ii) here is a variant of \cite[Corollary 2.6]{D}.
To prove (ii), assume $R$ is truncation closed and all truncations of
all $a\in A$ lie in $R[A]$. Let 
\[B\ :=\  \{f\in \smallk[[\fM]]:\ f\trunceq a \text{ for some }a\in A\}.\]
Then $B$ is truncation closed, and so is $R\cup B$, and thus the ring
$R[B]=R[A]$ is truncation closed, by (i) of the proposition above. 
An often used consequence of (i) and (ii) is that if $E$ is a truncation closed subfield of $\smallk[[\fM]]$ and $A\subseteq \fM$, then the subfield $E(A)$ of $\smallk[[\fM]]$ is also truncation closed.

Elaborating on (iii) and (iv) above, suppose $\smallk$ has characteristic $0$ and $E\supseteq \smallk$ is a truncation closed subfield of $\smallk[[\fM]]$. Then 
\[\fM_E\ :=\  E\cap \fM\ =\ \supp(E)\cap \fM\]
is a subgroup of $\fM$ with $E\subseteq \smallk[[\fM_E]]$. Since $\smallk[[\fM_E]]$ is henselian, we have $E^{\h}\subseteq \smallk[[\fM_E]]$. We call the subfield 
\[F\ :=\ \{f\in \smallk[[\fM]]:\ f \text{ is algebraic over }E\}\] 
of $\smallk[[\fM]]$  the {\em algebraic closure of $E$ in $\smallk[[\fM]]$} and we can characterize it as follows:

\begin{lemma}\label{actr} Let $\fM^{\div}_{E}:=\{\fg\in \fM:\ \fg^n\in \fM_E \text{ for some }n\ge 1\}$, the divisible closure of $\fM_E$ in $\fM$. Then $\fM^{\div}_{E}$ is a subgroup of
$\fM$, $E(\fM^{\div}_{E})$ and $F$ are truncation closed subfields of $\smallk[[\fM]]$, and 
\[F\ =\ E(\fM^{\div}_{E})^{\h}.\] 
\end{lemma}
\begin{proof} It follows from (i) and (ii) of Proposition~\ref{D} that $E(\fM^{\div}_E)$ is truncation closed. Also, 
$E(\fM^{\div}_E)\subseteq F$, hence $F$ is an algebraic extension of the henselization $E(\fM^{\div}_{E})^{\h}$, and thus $F$ is truncation closed by (iii) and (iv) of Proposition~\ref{D}. Therefore,
$\fM_F=\fM_E^{\div}$, and so $F$ is an immediate extension of $E(\fM^{\div}_E)^{\h}$. The latter is algebraically maximal, and
thus equals $F$.   
\end{proof}

\subsection*{Transcendental extensions}
The above deals with extension procedures of algebraic nature. 
As in \cite{D}, we also consider possibly transcendental adjunctions of
of the following kind.
Let $\mathcal{F}= (\mathcal{F}_n)$ be a family such that for each $n$: $\mathcal{F}_n\subseteq \smallk[[X_1,\ldots,X_n]]$, and for all $F\in \mathcal{F}_n$ we have $\partial F/\partial X_i \in \mathcal{F}_n$ for $i=1,\ldots,n$.
Let $K$ be a subfield of $\smallk[[\fM]]$.
We define the $\mathcal{F}$-\textit{extension}  $K(\mathcal{F},\prec 1)$
of $K$ to be
the smallest subfield of $\smallk[[\fM]]$ that contains $K$ and the set 
\[  \bigcup_n \{F(f_1,\ldots,f_n):\ F\in \mathcal{F}_n,\ 
f_1,\dots, f_n\in K^{\prec 1}\} .\]

\begin{lemma}\label{tL16}
Suppose  $\text{char}(\smallk)=0$ and $K$ is a truncation closed subfield of $\smallk[[\fM]]$ containing $\smallk$. Then $K(\mathcal{F},\prec 1)$ is truncation closed. 
\end{lemma}

\begin{proof}
Let $E$ be the largest truncation closed subfield of $K(\mathcal{F},\prec 1)$.
Towards a contradiction, assume that $E\ne K(\mathcal{F},\prec 1)$.
Let $n$ be minimal such that $F(f_1,\ldots,f_n)\notin E$ for some  $F\in \mathcal{F}_n$ and  $f_1,\ldots, f_n \in K^{\prec 1}$.
Take $f_1,\ldots,f_n\in K^{\prec 1}$ such that $F(f_1,\ldots,f_n)\notin E$ for some $F\in \mathcal{F}_n$ with lexicographically minimal $(o(f_1),\ldots, o(f_n))$.
Taking such $F$ (so $f = F(f_1,\ldots, f_n)\notin E$),
it suffices to show that then all proper truncations of $f$ lie in $E$.
If \[o(f_1)\ =\ \ldots\ =\ o(f_n) = 0\] then $f_1=\dots=f_n=0$, so 
$F(f_1,\ldots, f_n)\in \smallk$.
Next assume that $f_i\neq 0$ for some $i\in \{1,\ldots,n\}$ and
 $\phi$ is a proper truncation of $f$. Every monomial in $\supp f$ is
$\succeq \fm^k$ for some $\fm\in \bigcup_{i=1}^n \supp f_i$ and some 
$k\in \N$, so we can take $i\in \{1,\ldots,n\}$, $\fm\in \supp(f_i)$, and $N=N(\fm)\in \mathbb{N}$ such that $\supp(\phi)\succ \fm^k$ for all $k>N$.
Take $h,g \in K$ such that $f_i=h+g$, $ h =f_i|_{\fm}$ and $\fd(g) = \fm$. 
By Taylor expansion,
\[F(f_1,\ldots,f_n)\ =\ \sum_{k=0}^\infty \frac{\partial^k F}{\partial X_i^k}(f_1,\ldots,f_{i-1},h,f_{i+1},\ldots, f_n) \frac{g^k}{k!}.\]
Thus $\phi$ is a truncation of 
\[\sum_{k=0}^{N} \frac{\partial^k F}{\partial X_i^k}(f_1,\ldots,f_{i-1},h,f_{i+1},\ldots, f_n) \frac{g^k}{k!}.\]
Given that $(o(f_1),\ldots,o(f_n))$ is minimal and $o(h) < o(f_i)$, we get $\phi \in E$.  
\end{proof}

\noindent 
For example, if $\text{char}(\smallk)=0$, we could take
\[\cF_1=\{(1+X_1)^{-1},\ \exp X_1,\ \log(1+X_1)\},\qquad  \cF_n=\emptyset\  \text{ for }n>1\]
where $\exp X_1:=\sum_{i=0}^\infty X_1^i/i!\ $  and $\log(1+X_1):=\sum_{i=1}^{\infty}(-1)^{i+1}X_1^i/i$.

A subfield $E$ of $\smallk[[\fM]]$ is said to be $\cF$-closed if 
$f(\vec a)\in E$ for all
$f(X_1,\dots, X_n)\in \cF_n$ and $\vec a\in (E\cap\ \smallo )^{\times n}$, $n=1,2,\dots$. The $\cF$-closure of a subfield
$E$ of $\smallk[[\fM]]$ is the smallest $\cF$-closed subfield 
$\cF(E)$ of $\smallk[[\fM]]$ that contains $E$. The next result is \cite[Theorem 1.3]{D}. The more informative Lemma~\ref{tL16} above is not stated explicitly in \cite{D}.

\begin{corollary}\label{ta} If $\operatorname{char}(\smallk)=0$ and
$E\supseteq \smallk$ is a truncation closed subfield of $\smallk[[\fM]]$,
then its $\cF$-closure $\cF(E)$ is also truncation closed.
\end{corollary}

\noindent
Assuming $\operatorname{char}(\smallk)=0$ we now consider the meaning of the above for the case where $\cF_1=\{\exp X_1\}$ and $\cF_n=\emptyset$ for $n>1$. 
(This case is needed in Chapters 7 and 8.)
Let $E$ be a subfield of $k[[\fM]]$. We say that $E$ is {\bf closed under small exponentiation\/} if $\exp(E^{\prec 1})\subseteq E$.
We define the {\bf small exponential closure of $E$\/} to be the smallest subfield $F\supseteq E$ of $\smallk[[\fM]]$ that is closed under small exponentiation;
we denote this $F$ by
$E^{\smallexp}$. Thus by Corollary~\ref{ta}, if $E\supseteq \smallk$ is truncation closed, then $E^{\smallexp}\subseteq\smallk[[\fM_E]]$ is truncation closed.  

\subsection*{Additional facts on truncation} 
Besides $\Gamma$ we now consider a second ordered abelian 
group $\Delta$. Below we identify $\Gamma$ and $\Delta$ 
in the usual way with subgroups of
the lexicographically ordered sum $\Gamma\oplus \Delta$, so that 
$\Gamma+\Delta=\Gamma\oplus \Delta$, with $\gamma>\Delta$ for all $\gamma\in \Gamma^{>}$. This makes $\Delta$ a convex subgroup of $\Gamma+\Delta$. 
Let $\smallk_0$ be a field and $\smallk = \smallk_0((t^\Delta))$.
Then we have a field isomorphism 
\[\smallk((t^\Gamma))\longrightarrow \smallk_0((t^{\Gamma+ \Delta}))\]
that is the identity on $\smallk$, namely 
\[f=\sum_{\gamma}f_{\gamma}t^\gamma\mapsto 
\sum_{\gamma,\delta} f_{\gamma,\delta}t^{\gamma+\delta}\]
where $f_{\gamma}=\sum_\delta f_{\gamma,\delta}t^{\delta}$ for all $\gamma$. 
Below we identify $\smallk((t^\Gamma))$ with  $\smallk_0((t^{\Gamma+ \Delta}))$
via the above isomorphism. 
 For a set $S\subseteq \smallk((t^\Gamma))$
this leads to two notions of truncation: 
we say that $S$ is
$\smallk$-truncation closed if it is truncation closed with 
$\smallk$ viewed as the coefficient 
field and $\Gamma$ as the group of exponents (that is, viewing $S$ as a subset of the Hahn field $\smallk((t^\Gamma))$ over $\smallk$), and we say that
$S$ is $\smallk_0$-truncation closed if it is truncation closed  with 
$\smallk_0$ viewed as the coefficient 
field and $\Gamma+\Delta$ as the group of exponents (that is, viewing
$S$ as a subset of the Hahn field $\smallk_0((t^{\Gamma+\Delta}))$
over $\smallk_0$). 

\begin{lemma}\label{trtr1} If $S\subseteq\smallk((t^\Gamma))$ is $\smallk_0$-truncation closed, then $S$ is $\smallk$-truncation closed.
\end{lemma}
\begin{proof} Let $f\in \smallk((t^\Gamma))$ and $\gamma\in \Gamma$. If 
$ (\gamma+\Delta^{<})\cap \supp_{\smallk_0}f=\emptyset$, then the
$\smallk$-truncation of $f$ at $\gamma$ equals the $\smallk_0$-truncation 
of $f$ at 
$\gamma\in \Gamma +\Delta$.  

If $(\gamma+\Delta^{<})\cap \supp_{\smallk_0}f\ne
\emptyset$, then the $\smallk$-truncation of $f$ at $\gamma$ equals the $\smallk_0$-truncation of $f$ at the least element of $(\gamma+\Delta^{<})\cap \supp_{\smallk_0}f$.
\end{proof}

\begin{lemma}\label{trtr2} Let $\smallk_1$ be a truncation closed subfield of the Hahn field 
$\smallk_0((t^\Delta))= \smallk$ over $\smallk_0$, and let $V$ be a 
$\smallk_1$-linear subspace of $\smallk_1((t^\Gamma))\subseteq \smallk((t^\Gamma))$
such that $V\supseteq \smallk_1$ and $V$ is $\smallk$-truncation closed.
Then $V$ is $\smallk_0$-truncation closed.
\end{lemma} 
\begin{proof} If $\Gamma=\{0\}$, then $\smallk_1((t^\Gamma))=\smallk_1$, so
$V=\smallk_1$  is $\smallk_0$-truncation closed. In the rest of the proof
we assume $\Gamma\ne\{0\}$. Let $\beta=\gamma+\delta\in \Gamma+\Delta$ with
$\gamma\in \Gamma$ and $\delta\in \Delta$, and $f\in V$.
Let $g$ be the truncation of $f$ at $\beta$ in the Hahn field
$\smallk_0((t^{\Gamma+\Delta}))$ over $\smallk_0$ and $h$ the truncation of $f$ at $\gamma$ in the Hahn field $\smallk((t^\Gamma))$. Then $h\in V$
and $g=h+s$ with $s=\phi t^\gamma$ and $\phi\in \smallk_0((t^\Delta))$.
If $s=0$, then $g\in V$ trivially, so assume $s\ne 0$.
Then $f$ has a $\smallk$-truncation $h+\theta t^\gamma\in V$ with 
$0\ne\theta\in \smallk_1$ and so $\theta t^\gamma\in V$, $t^\gamma=\theta^{-1}(\theta t^\gamma)\in V$. Moreover,
$\phi$ is a $\smallk_0$-truncation of $\theta$, so $\phi\in \smallk_1$, hence
$s=\phi t^\gamma\in V$, and thus $g\in V$.
\end{proof}

\section{Infinite Parts}
\noindent
For the rest of this section $K$ is a valued field with monomial group $\fM$ and coefficient field $C$; the latter means that
$C$ is a subfield of $K$ that is contained in the valuation ring
of $K$ and maps (necessarily isomorphically) onto the residue field under the residue map. We let $v:K^\times\rightarrow \Gamma$ be the valuation of $K$, $\smallk$ its residue field, $\mathcal{O}$ the valuation ring of $K$, and $\smallo$ the maximal ideal of $\bigO$.
Note that $\bigO=C\oplus \smallo$.  

Suppose $\cU$ is a $C$-vector subspace of $K$ such that $K=\cU\oplus \bigO$, $\cU$ is closed under multiplication in $K$, and $\cU\supseteq \fM^{\succ 1}$.
We call such $\cU$ an \textit{infinite part of} $K$.
For $a\in K$ we define its $\cU$-\textit{part} to be 
the unique $u\in \cU$ such that $a= u+b$ for some $b \in \bigO$, and we denote the 
$\cU$-part of $a$ by $a|_{\cU}$.

\begin{example} The Hahn field $C[[\fM]]$ with coefficient field $C$ and
monomial group $\fM$, and the usual valuation given by its valuation ring 
\[C[[\fM^{\preceq 1}]]\ :=\ \{f\in C[[\fM]]:\ \supp f \subseteq \fM^{\preceq 1}\},\] has infinite part 
\[C[[\fM^{\succ 1}]]\ :=\ \{f\in C[[\fM]]:\ \supp f \subseteq \fM^{\succ 1}\}.\]
We call this the \textbf{canonical infinite part} of the Hahn field $C[[\fM]]$. 

More generally, suppose $K$ is a truncation closed subfield of $C[[\fM]]$ with $C\subseteq K$ and $\fM\subseteq K$. 
We consider $K$ as a valued subfield of this Hahn field,
and equip $K$ with its monomial group $\fM$ and its coefficient field $C$. Then $K$ has infinite part $K\cap C[[\fM^{\succ 1}]]$.
\end{example}

\subsection*{Truncation with respect to an infinite part}
For any monomial $\fm\in \fM$ we define the \textit{truncation of} $a$ \textit{at} $\fm$ \textit{with respect to} $\cU$, as $a|^\cU_{\fm}:=\fm\cdot(\fm\inv a)|_{\cU}$. 
Note that for $\fm = 1$ we get $a|^{\cU}_{1} = a|_{\cU}$.
Note also that for the canonical infinite part
$\cU$ of the Hahn field $C[[\fM]]$ we have $a|^{\cU}_{\fm}=a|_{\fm}$. 

The following rules are easy consequences of the definition above:

\begin{lemma} \label{Vtruncprop}
Let $a,b\in K$,  $c\in C$, and $\fm, \fn\in \fM$. Then:
\begin{enumerate}
	\item[\rm{(i)}] $(a+b)|_{\fm}^\cU =a|_{\fm}^\cU + b|_{\fm}^\cU$ and 
	$(ca)|_{\fm}^{\cU}= c\cdot a|_{\fm}^{\cU}$;
    \item[\rm{(ii)}] $a\preceq \fm$ $\iff$ $a|_{\fm}^\cU = 0$;
    \item[\rm{(iii)}] $a \succ \fm $ $\iff$ $a|_{\fm}^\cU \sim a$;
    \item[\rm{(iv)}] $a|_{\fm}^{\cU}\preceq a$, by (2) and (3);
    \item[\rm{(v)}] $a- a|_{\fm}^\cU\preceq \fm $;
    \item[\rm{(vi)}] $a|_{\fm}^\cU = b|_{\fn}^\cU=0 \Rightarrow (ab)|_{\fm\fn}^\cU=0$;
    \item[\rm{(vii)}] $a|_{\fm}^\cU\neq 0 \Rightarrow a|_{\fm}^\cU\succ \fm$;
    \item[\rm{(viii)}] $\fm\cdot a|_{\fn}^\cU = (\fm a)|_{\fm\fn}^\cU$;
    \item[\rm{(ix)}] $\fm|^{\cU}_{\fn}=\fm$ if $\fm\succ \fn$, and $\fm|^{\cU}_{\fn}=0$ if $\fm\preceq \fn$.  
\end{enumerate}
\end{lemma}

\noindent
The following lemmas give a commutativity property for truncation with respect to $\cU$.

\begin{lemma}\label{trunccom1}
Let $a\in K$ and $\fm\in \fM$. Then
\[  (a|_{\cU})|_{\fm}^\cU\ =\ (a|_{\fm}^{\cU})|_{\cU}\ =\ 
a|^{\cU}_{\max(1,\fm)}.\]
\end{lemma}

\begin{proof}
Let $a= u+ b$ with $u\in \cU$, $b\in \bigO$. Consider first the case $\fm\preceq 1$. Then

\begin{align*}
  (a|_{\cU})|_{\fm}^\cU\ &=\  u|_{\fM}^\cU \\
							&=\ \fm\cdot (\fm\inv u)|_{\cU} \\
                            &=\ \fm\cdot(\fm\inv u)\  =\ u.
\end{align*}
On the other hand, 
\begin{align*}
	   a|_{\fm}^\cU\ &=\   u|_{\fm}^\cU+ b|_{\fm}^\cU\\
    				  &=\ u+ b|_{\fm}^\cU.
\end{align*}
Now $b|_{\fm}^{\cU}\preceq b$ by (iv) of Lemma~\ref{Vtruncprop}, so
$b|_{\fm}^\cU\in \bigO$, and thus
 $(a|_{\cU})|_{\fm}^\cU = u = (a|^{\cU}_{\fm})|_{\cU}$. Since $u=a|^{\cU}_{1}$, this proves the lemma for $\fm\preceq1$.
Next, assume $\fm \succ 1$. Then by (i) and (ii) of Lemma~\ref{Vtruncprop},

\begin{align*}
a|_{\fm}^\cU\ &=\   \ u|_{\fm}^\cU+\ b|_{\fm}^\cU\\
					&=\ \ u|_{\fm}^\cU\\
                    &=\  (a|_{\cU})|_{\fm}^\cU.
\end{align*}
On the other hand, $ (a|_{\fm}^{\cU})|_{\cU} = \big(\fm\cdot (\fm\inv a)|_{\cU}\big)|_{\cU} = \fm\cdot (\fm\inv a)|_\cU=a|_{\fm}^{\cU}$, since $\cU$ is closed under products.
\end{proof}

\noindent
For the proof of Lemma \ref{trunccom}, note that for $\fm, \fn \in \fM$ and $a\in K$ we have $(\fm a)|_{\fn}^\cU = \fm\cdot  a|_{\fn/ \fm}^\cU$.

\begin{lemma}\label{trunccom}
Let $a\in K$ and $\fm\preceq \fn$ in $\fM$. Then 
\[  (a|_{\fm}^{\cU})|_{\fn}^{\cU}\ =\  (a|_{\fn}^{\cU})|_{\fm}^{\cU}\  =\   a|_{\fn}^\cU.\]
\end{lemma}
\begin{proof}
	To simplify notation we omit in this proof superscripts $\cU$ in expressions like  $a|^{\cU}_{\fm}$.
    Using Lemma \ref{trunccom1},
    \[\trunc{(\trunc{a}{\fm})}{\fn}\ =\ \trunc{\big(\fm\cdot(\trunc{\fm\inv a }{\cU})\big)}{\fn}\ =\ \fm\cdot\trunc{\big(\trunc{(\fm\inv a )}{\cU} \big)}{\fn/\fm}\]
    \[\ \ \ \ \ \ \ \  =\ \fm\cdot (\fm^{-1}a)|_{\fn/\fm}\ =\ (\fm\fm\inv a)|_{\fn}\ =\ a|_{\fn}.
    \]
    In a similar way we get $\trunc{(\trunc{a}{\fn})}{\fm} = \trunc{a}{\fn}$.
\end{proof}
 
\subsection*{$\cU$-support} For $a\in K$ we define the $\cU$-\textit{support of} $a$  as 
\[\supp_\cU(a)\ =\ \{\fm\in \fM: a - \trunc{a}{\fm}^\cU\asymp \fm  \}.\] 
When $K$ is a Hahn field $C[[\fM]]$ and $\cU$ is its canonical infinite part, this agrees with the
usual support $\supp a$.  Most of the following rules are easy to verify: 

\begin{lemma}\label{lemsuppU}
Let $a,b\in K$ and $\fm\in \fM$. Then:
\begin{enumerate}
	\item[\rm{(i)}] $\fm \cdot \supp_\cU(a) = \supp_\cU(\fm a)$, 
	\item[\rm{(ii)}] $\supp_\cU(a + b)\subseteq \supp_\cU(a)\cup\supp_\cU(b)$ and $\supp_{\cU}ca=\supp_{\cU}a$ for $c\in C^\times$, 
    \item[\rm{(iii)}] $\supp_\cU(a)\cap \supp_\cU(b) = \emptyset$ $\Rightarrow$ $\supp_\cU(a + b)= \supp_\cU(a)\cup\supp_\cU(b)$,
    \item[\rm{(iv)}] $a\asymp \fm \Rightarrow \fm \in \supp_\cU(a)$,
    \item[\rm{(v)}] $ \supp_\cU(a)\preceq a$,
    \item[\rm{(vi)}] $\supp_{\cU}( \fm)=\{\fm\}$, 
    \item[\rm{(vii)}] $\supp_{\cU}(a|^{\cU}_{\fm})=\supp_{\cU}(a)^{\succ \fm}$ and $\supp_\cU(a-\trunc{a}{\fm}^\cU)= \supp_\cU(a)^{\preceq \fm}$,
   \item[\rm{(viii)}] $\supp_\cU(a)\succ \fm\ \Leftrightarrow\ a=\trunc{a}{\fm}^\cU\ \Leftrightarrow\ a\in \fm\cU$.
\end{enumerate}
\end{lemma}
\begin{proof} We only do (vii), since (i)--(vi) are  easy to verify, and (viii) is an easy consequence of (vii), using also (viii) of Lemma~\ref{Vtruncprop} for the second equivalence. We have $a = \trunc{a}{\fm}^\cU + c\fm + b$ for unique $c\in C$ and $b\prec \fm$, by (v) in Lemma \ref{Vtruncprop}. Hence by (ii), (v), and (vi),
$\supp_{\cU}(a)^{\succ \fm}\subseteq  \supp_{\cU}(a|^{\cU}_{\fm})$.
Moreover, 
$\supp_{\cU}(a|^{\cU}_{\fm})\succ \fm$ by Lemma~\ref{trunccom},
and then applying (ii), (v), (vi) to $a|^{\cU}_{\fm}=-a-c\fm-b$ yields
$\supp_{\cU}(a|^{\cU}_{\fm})\subseteq \supp_{\cU}(a)$, 
which gives the first equality of (vii). As to the second equality,
(ii), (v), (vi) and the first equality  give 
\[\supp_{\cU}(a)^{\preceq \fm}\ \subseteq\ \supp_{\cU} (c\fm+b)\ =\ \supp_\cU(a-\trunc{a}{\fm}^\cU)\ \preceq\ \fm,\]
so by (ii) and the first equality,
$\supp_{\cU}(a-a|^{\cU}_{\fm})\subseteq \supp_{\cU}(a)^{\preceq \fm}$.
\end{proof}

\subsection*{Good Infinite Parts} As before, $\cU$ is an infinite part of $K$.  We say that $\cU$ is \textbf{good} if for all $a,b \in K$ we have $\supp_\cU(ab)\subseteq \supp_\cU(a)\supp_\cU(b)$. In view of the definability of truncation with respect to $\cU$ (in the valued field $K$ equipped with $\fM$, $C$, $\cU$) and the first-order nature of the $\cU$-support, being good is also of first-order nature. We do not know whether being good is a consequence of the definition of ``infinite part''. 
Note that the canonical infinite part of a Hahn field $C[[\fM]]$ is good, and so is the infinite part $K\cap C[[\fM^{\succ 1}]]$ of $K$ when $K\supseteq C(\fM)$ is a truncation closed subfield of $C[[\fM]]$.

\medskip\noindent
Assume $\cU$ is good and $K$ is  a truncation closed valued subfield of $C[[\fM]]$, with the same $C$ and $\fM$ as for $K$,  so $C(\fM)\subseteq K$.  Here ``truncation closed'' is of course with respect to truncation in the Hahn field $C[[\fM]]$. 
Thus for $a\in K$ and $\fm\in \fM$ we have in principle two truncations, namely $a|_\fm$ as an element of $C[[\fM]]$ and the truncation $a|^\cU_{\fm}$ with respect to 
$\cU$. Set
\[E\ :=\  \{a\in K:\ \trunc{a}{\fm} = \trunc{a}{\fm}^\cU \text{ for all } \fm\in \fM\}.\]
Thus $\supp(a)=\supp_{\cU}(a)$ for all $a\in E$.

\begin{lemma}
$E$ is a truncation closed subring of $K$ with $C[\fM]\subseteq E$. 
\end{lemma}
\begin{proof}
Clearly, $E$ is truncation closed in $C[[\fM]]$, and $E$ is a $C$-linear subspace of $K$ that contains $C$ and $\fM$.  Let $a,b\in E$; it remains to show that then $ab\in E$. We do so by induction on $(o(a),o(b))$ with respect to the lexicographic ordering. If either $o(a)=0$ or $o(b)=0$, then $a=0$ or $b=0$, in which case we are done. Let $o(a)>0$ and $o(b)>0$, and $\fd\in \fM$. We want to show $\trunc{ab}{\fd} = \trunc{ab}{\fd}^\cU$. Assume that  $\fm\in \supp(a)$ and $\fn \in \supp(b)$ are such that $\fm\fn \preceq \fd$. Then
\[a = a_1 + a_2 \text{ with } a_1 := \trunc{a}{\fm}\  \text{ and }\   b= b_1+b_2 \text{ with } b_1 = \trunc{b}{\fn},\]
and $ab= a_1b + a_2b_1 + a_2b_2$. Hence
\[\trunc{(ab)}{\fd}\ =\ \trunc{(a_1b)}{\fd} + \trunc{(a_2b_1)}{\fd}\  \text{ and }\  \trunc{(ab)}{\fd}^\cU\ =\ \trunc{(a_1b)}{\fd}^\cU + \trunc{(a_2b_1)}{\fd}^\cU.\]
By induction we have $ \trunc{(a_1b)}{\fd}=\trunc{(a_1b)}{\fd}^\cU$ and $\trunc{(a_2b_1)}{\fd} = \trunc{(a_2b_1)}{\fd}^\cU$, so $\trunc{ab}{\fd} = \trunc{ab}{\fd}^\cU$. 
If for all monomials $\fm \in \supp(a)$ and $\fn \in \supp(b)$ we have $\fm\fn\succ \fd$, then $\trunc{ab}{\fd} = ab = \trunc{ab}{\fd}^\cU$ where the last equality uses (8) of Lemma~\ref{lemsuppU}
and the assumption that $\cU$ is good.
\end{proof}

\noindent
This seems as far as we can go under first-order assumptions on 
$\cU$. In the next subsection we indicate some cases with archimedean $\fM$ where $a|^{\cU}_{\fm}=a|_{\fm}$ for all $a\in K$ and $\fm\in \fM$. 

\subsection*{The Archimedean Case}
In this subsection we assume that $\fM$ is archimedean and $K$ is a valued subfield of the Hahn field $C[[\fM]]$ with the same $C$ and $\fM$ as for $K$.  In the next lemma we do not  assume that $\cU$ is good, but derive it from the lemma. Let $\fm$ range over $\fM$. 

\begin{lemma}\label{archCM} Suppose $K=C(\fM)$. Then $\cU=\bigoplus_{\fm\succ 1} C\fm$.
\end{lemma}
\begin{proof} The key point is the easily verified observation that
\[F\ :=\ \{a\in C[[\fM]]:\ \supp(a)^{\succ \fm} \text{ is finite for all } \fm\}\]
is a subfield of $C[[\fM]]$ containing $K$.
\end{proof}

\noindent It is worth saying a bit more on this case. Since 
$C\cup \fM$ is truncation closed, $K=C(\fM)$ is truncation closed in $C[[\fM]]$. By Lemma~\ref{archCM} there is only one infinite part of $K$, in particular, $\cU=K\cap C[[\fM^{\succ 1}]]$ and so $\cU$ is good and
$a|^{\cU}_{\fm}=a|_{\fm}$ for all $a\in K$ and all $\fm$. We also
note that the field $F$ defined in the proof of that lemma is the closure of $C[\fM]$  (and of $C(\fM)$) in $C[[\fM]]$ with respect to the valuation topology. Thus
$F$ is henselian as a valued subfield of $C[[\fM]]$, by \cite[Corollary 3.3.5]{ADH}.  Note also that $F$ is truncation closed in $C[[\fM]]$ and has, just like $K$, only one infinite part, which is the same as
that of $K$, namely $\bigoplus_{\fm\succ 1} C\fm$. 

\medskip\noindent
In the remainder of the section we assume that $\cU$ is good. Let $E$ be the subring of $K$ defined in the previous subsection. We can say more about $E$ in the archimedean case that we are considering in this subsection:   

\begin{lemma}
$E$ is a subfield of $K$.
\end{lemma}
\begin{proof}
Let $f\in E$ with $f\neq 0$, so $f= c\fm(1-\epsilon)$ with $c\in C^\times$, $\fm \in \fM$ and $\epsilon \in \smallo$. Then $\epsilon\in E$ and $f\inv = c\inv \fm\inv (1-\epsilon)\inv$. Hence $f\inv \in E$ if and only if $(1-\epsilon)\inv\in E$. Let $\fd \in \fM$ be given. Since $\fM$ is archimedean, we have $n$ such that $\epsilon^n \preceq \fd$, so $(1-\epsilon)\inv = 1 + \epsilon + \cdots + \epsilon^nb$ with $b\preceq 1$. Then
\[\trunc{(1-\epsilon)\inv}{\fd}\ =\ \trunc{(1+\epsilon + \cdots + \epsilon^{n-1})}{\fd}\ =\ \trunc{(1+\epsilon + \cdots + \epsilon^{n-1})}{\fd}^\cU\ =\ \trunc{(1-\epsilon)\inv}{\fd}^\cU. \]
Thus $(1-\epsilon)\inv \in E$. 
\end{proof}

\begin{lemma}
Suppose $f\in C[[t_1,\ldots,t_n]]$ and $f(a) \in K$ for all $a\in  \smallo_K^n$. Then $f(a)\in E$ for all $a\in \smallo_E^n$.
\end{lemma}
\begin{proof}
Let $\epsilon= (\epsilon_1,\ldots,\epsilon_n)\in \smallo^n_E$. Let $\fd \in \fM$ be given. Take $m$ such that $\epsilon_i^m \preceq \fd$ for $i=1,\ldots,n$. We have $f= g + t_1^mh_1 + \cdots +t_n^m h_m$ where $g\in C[t_1,\dots, t_n]$ is of degree $<m$ in $t_i$ for every $i$, and $h_i \in C[[t_1,\ldots t_n]]$ for $i=1,\dots,n$.  Then $f(\epsilon) = g(\epsilon) + b$ with $b\preceq \fd$, so $\trunc{b}{\fd} = \trunc{b}{\fd}^\cU=0$, and thus $\trunc{f(\epsilon)}{\fd}= \trunc{f(\epsilon)}{\fd}^\cU$.
\end{proof}

\begin{corollary}
If $K$ is henselian, then $E$ is a henselian valued subfield of $K$.
\end{corollary}
\begin{proof}
Assume $K$ is henselian, and let $P(X) = 1+X + a_1X^2 + \cdots + a_n X^{n+1} \in E[X]$ with $a_i \in \smallo_E$ for $i=1,\ldots,n$.  Then $P$ has a unique zero $y\preceq 1$ in $K$. It suffices to show that $y\in E$. By the implicit function theorem we have $y = f(a_1,\ldots,a_n)$ for some power series $f\in C[[t_1,\ldots,t_n]]$. Using the previous Lemma, we conclude that $y\in E$.
\end{proof}

\begin{corollary}
Suppose $C$ has characteristic 0 and $K$ is henselian. Then $E$ is algebraically closed in $K$.
\end{corollary}
\begin{proof}
Since $E$ is henselian and of equicharacteristic zero, it is algebraically maximal by \cite[Corollary 3.3.21]{ADH}.  It remains to note that $C[[\fM]]$ is an immediate extension
of $E$ in view of $E\supseteq C(\fM)$. 
\end{proof}

\chapter{Undecidability Results for Hahn Fields with Truncation} \label{TU}

\noindent
In this chapter we consider some model theoretic properties of (valued) Hahn fields with truncation. We show (Corollary \ref{KGamvTUnd})  that such  structures
are very wild in the sense that they can even interpret the theory of $(\N;+,\times)$ via an interpretation of $(\N, \mathcal{P}(\N); +,\in)$. In particular, they
are undecidable; this answers a question posed informally by van den Dries some years ago. As an application we show in Section 3.3 that the exponential field $\T$ of transseries with its canonical monomial group $G^{\operatorname{LE}}$ is undecidable. 
In Section 3.4 we indicate for  Hahn fields with truncation a definable binary relation with ``bad'' properties such as the strict order property and the tree property of the second kind. In Section 3.5 we record an undecidability result for valued Hahn fields $C((t^{\Q}))$ with the derivation $d/dt$; this doesn't involve truncation.

In connection with Theorem~\ref{NinK} the author would like to thank Philipp Hieronymi and Erik Walsberg for bringing monadic second-order logic to his attention.

\section{The Set of Natural Numbers in Monadic Second-Order Logic}
\noindent
We start with the following well-known result.
\begin{theorem}\label{Nund}
The theory of $(\N;+,\times)$ is undecidable.
\end{theorem}

\subsection*{Monadic second-order logic} Given a structure $\cM=(M;\ldots)$, \textbf{monadic second-order logic of} $\cM$ extends first-order logic over $\cM$ by allowing quantification over subsets of $M$.
More precisely it amounts to considering the two-sorted structure $(M,\cP(M);\ldots,\in)$, where the usual interpretation is given to $\in \subseteq M\times \cP(M)$.
The following positive result and its proof appear in \cite{G}, and are in contrast with the negative Corollary \ref{MOSNplus}.
\begin{theorem}
The theory of $(\N,\cP(\N);\in)$ is decidable.
\end{theorem}

\begin{lemma}\label{timesinMOSNplus}
Multiplication on $\N$ is definable in $(\N,\mathcal{P}(\N);+, \in)$.
\end{lemma}
\begin{proof}
If the multiplication of consecutive numbers is defined, then general multiplication of two natural numbers can be defined in terms of addition:
\[n=mk \iff (m+k)(m+k+1) = m(m+1) + k(k+1) + n + n.\]

\noindent
If divisibility is defined, then multiplication of consecutive numbers is defined by
\[n=m(m+1)\iff \forall k(\in \N) (n|k \leftrightarrow [m|k \wedge (m+1)|k]).\]

\noindent
Divisibility can be defined using addition by
\[m|n \iff \forall S \big(0\in S \wedge \forall x (x\in S \rightarrow x+m\in S) \rightarrow n\in S\big)\,\]
where the variable $S$ ranges over $\cP(\N)$ and the variable $x$ ranges over $\N$.
Since addition is a primitive, multiplication is defined in $(\N,\mathcal{P}(\N); +, \in)$.
\end{proof}

\begin{corollary}\label{MOSNplus}
The theory of $(\N,\mathcal{P}(\N); +, \in)$ is undecidable.
\end{corollary}
\begin{proof}
This follows from Lemma \ref{timesinMOSNplus} and Theorem \ref{Nund}.
\end{proof}

\section{Hahn Fields with Truncation}\label{HFWT}
\noindent
Let $K=\smallk((t^\Gamma))$ be a Hahn field with nontrivial value group $\Gamma$. We consider $K$ below  as an $L_{\bigO, \fM, \cU}$-structure where 
\[L_{\bigO, \fM, \cU}\ :=\ \{0,1,+,\times,\bigO, \fM, \cU\},\] and the unary predicate symbols $\bigO, \fM,$ and $\cU$ are interpreted respectively as the valuation ring, the canonical monomial group $t^\Gamma$, and the canonical additive complement to $\bigO$. For $\gamma \in \Gamma$ and $\fm = t^\gamma$ we set $f|_\fm : = f|_\gamma$. Then we have the equivalence (for $f$, $u \in  K$)

\[f|_1 = u \iff u\in \cU\ \&\ \exists g\in \bigO (f=u+g),\]
showing that truncation at $1$ is definable in the $L_{\bigO, \fM, \cU}$-structure $K$. 
For $\fm \in t^\Gamma$ and $f\in K$ we have 
\[f|_\fm =  g \iff (\fm^{-1}f)|_0=\fm^{-1}g,\]
showing that the operation $(f,\fm)\mapsto f|_\fm : K\times t^\Gamma \rightarrow K$ is definable in the $L_{\bigO, \fM, \cU}$-structure $K$. 

For convenience of notation we introduce the asymptotic relations $\preceq,\prec ,$ and $\asymp$ on $K$ as follows. 
For $f,g\in K$, $f\preceq g$ if and only if there is $h\in \bigO$ such that $f=gh$, likewise $f\prec g$ if and only if $f\preceq g$ and $g\not\preceq f$, and  $f\asymp g$ if and only if $f\preceq g$ and $g\preceq f$. 
Let $R := \{(\fm,f)\in t^\Gamma\times K: \fm\in t^{\supp(f)}\}$. Then $R$ is definable in the $L_{\bigO, \fM, \cU}$-structure $K$ since for $a,b\in K$
\[(a,b)\in R\iff  a\in t^\Gamma \text{ and } b-b|_a \asymp a.\]

\begin{theorem}\label{NinK}
The $L_{\bigO, \fM, \cU}$-structure $K$ interprets $(\N,\mathcal{P}(\N); +, \in)$.
\end{theorem}
\begin{proof}
Let $\approx$ be the definable equivalence relation on $K$ such that $f\approx g$, for $f,g\in K$, if and only if $\supp(f) = \supp(g)$. 
Take $\fn\in t^\Gamma$ such that $\fn\prec 1$. Consider the element $f= \sum_n \fn^n\in K$, and the set $S= \{g\in K: \supp(g) \subseteq \supp(f)\}$. 
Let $E\subseteq t^{\supp(f)}\times (S/\hspace{-1mm}\approx)$  be given by 
\[(\fm,g/\hspace{-1mm}\approx)\in E :\iff \fm \in t^{\supp(g)},\] 
and note that $E$ is definable in the $L_{\bigO, \fM, \cU}$-structure $K$ since $R$ is. Define $\iota: \N \rightarrow t^{\supp(f)}$ by $\iota(n) = \fm^n$, and note that $\iota$ induces an isomorphism $(\N, \cP(\N);\in) \stackrel{\sim}{\longrightarrow} (t^{\supp(f)},S/\hspace{-1mm}\approx;E)$, such that $\iota(m+n) = \iota(m)\iota(n)$.
\end{proof}

\begin{corollary} \label{UndL}
The theory of the $L_{\bigO, \fM, \cU}$-structure $K$ is undecidable.
\end{corollary}
\begin{proof}
This follows easily from Theorem \ref{NinK} and Theorem \ref{MOSNplus}
\end{proof}

\noindent
In the next subsections we consider cases where we can replace $L_{\bigO, \fM, \cU}$ by another language.

\subsection*{The case of a real closed coefficient field.}
In this subsection we assume that the coefficient field $\smallk$ of $K$ is real closed. Then we have a field ordering $\leq$ on $K$ that extends the ordering of $\smallk$; it is given by: $f>0$ if and only if the leading coefficient of $f$ is $>0$. Note that then $f\in \bigO$ if and only if $|f|<u$ for all positive $u\in \cU$. Equipped with this ordering we now consider $K$ as an $L_{\le, \fM, \cU}$-structure where $L_{\le, \fM,\cU}:=\{0,1,+,\times,\le, \fM, \cU\}$.  We just observed that
$\bigO$ is $0$-definable in this $L_{\fM,\cU}$-structure. Thus:

\begin{corollary} The theory of $K$ as an $L_{\le, \fM,\cU}$-structure is undecidable.
\end{corollary}

\begin{corollary}  Suppose $\Gamma$ is divisible.  Then the theory of $K$ in the language $\{0,1,+,\times,\fM,\cU\}$ is undecidable.
\end{corollary}

\begin{proof}
The divisibility of  $\Gamma$ yields that $K$ is real closed, so the above ordering is the only field ordering on $K$ and is definable in the field $K$. 
\end{proof}


\subsection*{Defining the coefficient field $\smallk$} We now consider $K=\smallk((t^\Gamma))$ as an $L_{\bigO, \cU}$-structure, where 
\[L_{\bigO, \cU}\  :=\ \{0,1,+,\times, \bigO,\cU\}.\]
Note that for $f\in \bigO$ we have 
\[f\cU\subseteq \cU\  \iff\  f\in \smallk,\]
where we identify $\smallk$ with $\smallk t^0$. Thus we can define the coefficient field $\smallk$ in the $L_{\bigO, \cU}$-structure $K$. 

\medskip\noindent
\textbf{Question:} Is it possible to define the monomial group $t^\Gamma$ in the $L_{\bigO,\cU}$-structure $K$?

\subsection*{An approach without the monomial group} Alternatively we may work in the setting of the two-sorted structure $(K, \Gamma; v, T)$ where $K$ denotes the underlying field, $\Gamma$ is the ordered value group, $v$ is the valuation, and $T:K\times \Gamma \rightarrow K$ is such that $T(f,\gamma) = f|_\gamma$. Then we can define the binary relation $R\subseteq \Gamma\times K$ by 
\[(\gamma,f) \in R :\iff v(f-T(f,\gamma)) = \gamma .\]
We then obtain the following;
\begin{theorem}
The two-sorted structure $(K,\Gamma;v,T)$ interprets $(\N,\mathcal{P}(\N); +,\in)$.
\end{theorem}
The proof is similar to the proof of Theorem \ref{NinK}.
\begin{corollary}\label{KGamvTUnd}
The theory of the two-sorted structure $(K, \Gamma; v, T)$ is undecidable.
\end{corollary}

\section[Undecidability of Transeries with Exponential function and Monomials]{Undecidability of the Exponential Field of Transseries with a Predicate for its Multiplicative Group of Transmonomials}
\noindent
The field $\T$ of transseries is a real closed field extension of $\R$ with a canonical exponential operation \[\exp: \T\to  \T\] that extends the usual exponentiation
with base $e$ on $\R$. The elements of $\T$ are called \textbf{transseries} or  \textbf{logarithmic-exponential series}.  See \cite[Appendix A]{ADH} for a detailed
construction of $\T$. Below in Chapter 8 we repeat some key features
of this construction in connection with the canonical derivation of $\T$, but in this section the derivation plays no role, and we only  list here
some general facts about $\T$ that are relevant in connection with the undecidability results below. 

First, it is known from  \cite{DMM1} that $\T$ as an exponential field is an elementary extension of the real exponential field $\R_{\exp}$.  The field $\T$ also has a natural valuation on it given by the valuation ring 
\[\{f\in \T:\ |f|\le n \text{ for some }n\},\] 
which makes $\T$ a valued field with its subfield $\R$ as a field of representatives for the residue field.
The theory of $\T$ as a {\em valued exponential field\/} has also been analyzed, leading to model-completeness in a natural language, with various
 tameness properties such as weak o-minimality (and thus NIP) as a consequence; this is a special case of results in \cite{DII} and \cite{DL}.  

As a valued field, $\T$ has a distinguished monomial group $G^{\operatorname{LE}}$ whose elements are called \textbf{transmonomials} or \textbf{logarithmic-exponential monomials}. In fact, 
$\T$ is construed as a truncation closed subfield of the Hahn field  $\R[[G^{\operatorname{LE}}]]$, with $\R(G^{\operatorname{LE}})\subseteq \T$.  As a valued field with monomial group, the theory of $\T$ is also well-understood as a special case of the AKE-theorems, with AKE standing for Ax, Kochen, and Ersov. In particular, this complete theory is axiomatized by the
axioms for real closed fields with a nontrivial convex valuation and a monomial group of positive elements. 

\medskip\noindent
We now show that these positive results do not go through if we equip the valued field $\T$ simultaneously with its canonical exponentiation and  its distinguished monomial group $G^{\operatorname{LE}}$. In fact, the theory of this structure is undecidable in a rather strong way, and it turns out that we don't even need the valuation. The basic reason is that  $G^{\operatorname{LE}}=\exp \cU$, where $\cU=\{f\in \T:\ \supp f \succ 1\}$, in other words, $\cU$ is the natural infinite part of $\T$, that is, $\cU$ is the intersection of $\T$ with the canonical additive complement to the valuation ring of $\R[[G^{\operatorname{LE}}]]$.  In particular, $\cU$ is definable in the  exponential field $\T$ equipped with $G^{\operatorname{LE}}$. 

Using this, we can define truncation on $\T$ (with respect to $\cU$) in the same way we defined truncation for $\smallk((t^\Gamma))$. The same proof of Theorem \ref{NinK} goes through to show:
\begin{theorem} The valued exponential field $\T$ equipped with its monomial group $G^{\operatorname{LE}}$ interprets  the structure $(\N,\cP(\N);+,\in)$
\end{theorem}
\begin{corollary}
The theory of the valued exponential field  $\T$ equipped with its monomial group $G^{\operatorname{LE}}$ is undecidable.
\end{corollary}

\medskip\noindent
Note that the valuation ring of $\T$ is convex. Thus we can define the valuation ring from the infinite part $\cU$ and the ordering, since 
\[\bigO\  =\  \{f\in K:\  |f| \leq u\text{ for all positive } u\in \cU\}.\]
Furthermore, the ordering is definable from the field structure, so we obtain:

\begin{corollary}
The theory of the exponential field $\T$ with monomial group $G^{\operatorname{LE}}$ is undecidable.
\end{corollary}

\medskip\noindent

\section{Witnessing the Strict Order Property and a Tree Property}

\noindent
We have already shown how $(\N;+,\times)$ can be interpreted in the $L$-structure $K$ and thus we know that it has the strict order property and the tree property of the second kind among others. In this section we make explicit a binary relation that witnesses these properties inside $K$. 
\subsection*{The independence property}
Let $L$ be a language and $\cM= (M;\ldots)$ an $L$-structure. We say that an $L$-formula $\phi(x;y)$ \textbf{shatters} a set $A\subseteq M^x$ if for every subset $S$ of $A$ there is $b_S\in M^y$ such that for every $a\in A$ we have that $M\models \phi(a;b_S)$ if and only if $a\in S$. Let $T$ be an $L$-theory. We say that $\phi(x;y)$ has the \textbf{independence property with respect to }$T$, or \textbf{IP} for short, if there is a model $M$ of $T$, such that $\phi(x;y)$ shatters an infinite subset of $M^x$. 

For a partitioned formula $\phi(x;y)$ we let $\phi^{opp}(y;x) = \phi(x;y)$, that is, $\phi^{opp}$ is the same formula $\phi$ but where the role of the parameter variables and type variables is exchanged.

\begin{lemma}\label{IPopp}
A formula $\phi(x;y)$ has $\IP$ if $\phi^{opp}$ has $\IP$.
\end{lemma}
\begin{proof}
By compactness the formula $\phi(x;y)$ shatters some set $\{a_J: J\in \cP(\N)\}$. Let the shattering be witnessed by $\{b_I:I\subseteq \cP(\N)\}$. Let $B= \{b_{I_i}: i\in\N\}$ be such that $I_i = \{Y\subseteq \N: i\in Y\}$. Then we have
\[\models\phi(a_J, b_{I_i}) \iff i \in J,\]
and thus $\phi^{opp}$ shatters $B$. 
\end{proof}

\subsection*{The Strict Order Property}
We say that a formula $\phi(x;y)$ has the Strict Order Property, or $\SOP$ for short, if there are $b_i\in M^y $, for $i\in \N$, such that $\phi(M^x,b_i) \subset \phi(M^x,b_j)$ whenever $i<j$. 

\begin{proposition}
The formula $\varphi(x;y)$, defining the relation $R$ as in section \ref{HFWT}, has $\SOP$.
\end{proposition}

\begin{proof}
Let $\Theta= \{\theta_i: i\in \N\}$ be any subset of $\Gamma$ such that $\theta_i <\theta_j$ for $i<j$, and consider the set $\{f_n = \sum_{i=0}^n t^{\theta_i}: i \in \N\}$. Note that $\varphi(K,f_m)\subset \varphi(K,f_n)$ for $m<n$.  
\end{proof}

\subsection*{The tree property of the second kind}
We say that a formula $\phi(x;y)$ has the \textbf{tree property of the second kind}, or $\TP$ for short, if there are tuples $b_j^i\in M^y$, for $i,j\in \N$, such that for any $\sigma:\N\rightarrow \N$ the set $\{\phi(x;b^i_{\sigma(i)}): i\in \N\}$ is consistent and for any $i$ and $j\neq k$  we have $\{\phi(x;b^i_j),\phi(x;b^i_k)\}$ is inconsistent.

\begin{lemma}\label{TP2IP}
If $\phi(x;y)$ has $\TP$ then $\phi$ has $\IP$.
\end{lemma}
\begin{proof}
Let $\{\phi(x,b^i_j)\}_{i,j\in \N}$ witness $\TP$ for $\phi(x;y)$. Fix $j$. Without loss of generality we will assume that $j=0$. Consider the set $\{b^i_0\}$. Let $I\subseteq \N$. By $\TP$ there is $a_I \in M^x$ such that \[\cM\models  \phi(a_I;b^i_j) \iff (i\in I \text{ and }j=0, \text{ or } i\notin I \text{ and } j=1).\]
Thus by Lemma \ref{IPopp} $\phi(x;y)$ has IP.
\end{proof}

\begin{lemma}\label{TP2Char}
Let $A=\{a_i:i\in \N\}\subseteq M^x$ and $B=\{b_I:I\in \mathcal{P}(\N)\}\subseteq M^y$. Assume that there is $\phi(x;y)$ such that for any fixed $b_I\in B$
\[\models \phi(a;b_I)\iff \text{ there is }  i\in I \text{ such that }a=a_i .\]
Then $\phi$ has $\TP$. 
\end{lemma}

\begin{proof}
Let $\phi$, $A$, and $B$ be as in the hypothesis of the Lemma. Let $P=\{p_i \in \N\}$ be the set of primes where $p_i\neq p_j$ for $i\neq j$. We construct $A^i_j\subseteq \N$ recursively as follows:
\begin{itemize}
	\item $A^0_j:= \{p_j^{n_0}:n_0>0\}$
    \item $A^i_j:= \{p_{n_i}^m: n_i\in \N,\ m\in A^{i-1}_{j}\}.$
\end{itemize}
So for example $A^1_2 =\{p_{n_1}^{p_2^{n_0}}: n_1,n_0>0\}$. \\
\textbf{Claim 1} For $\alpha \in \N^{n}$ we have that $\bigcap_{i<n} A^i_{\alpha(i)}\neq \emptyset$.\\
It is not hard to check that 
\[p_{\alpha(0)}^{\text{\reflectbox{$\ddots$}} ^{p_{\alpha(n-1)}}}\in \bigcap_{i<n} A^i_{\alpha(i)}.\]
\textbf{Claim 2:} For fixed $i$, and $j\neq k$ we have $A^i_j\cap A^i_k = \emptyset$. \\
For simplicity in notation we prove the case where $i=1$. Let $m\in A^1_j\cap A^1_k$. Then $m = p_{m_1}^{p_j^{m_0}} = p_{n_1}^{p_k^{n_0}}$. Since $p_j^{m_0}$ and $p_k^{n_0}$ are nonzero, we have that $p_{m_1} = p_{n_1}$, and thus $p_j^{m_0} = p_k^{n_0}$. Similarly, since $m_0$ and $n_0$ are nonzero we conclude that $p_j=p_k$, and thus $j=k$. 

Now let $b^i_j = b_{A^i_j}$. By compactness, together with Claim 1, we get that the set $\{\phi(x;b^i_{\sigma(i)}): i\in \N\}$ is consistent.
By the hypothesis of the Lemma, together with claim 2, we get that for any $i$ and $j\neq k$  we have $\{\phi(x;b^i_j),\phi(x;b^i_k)\}$ is inconsistent.
\end{proof}
 If $\phi(x;y)$ and $A$ are as in the lemma, we say that $\phi(x;y)$ and $B$ \textbf{only shatter} $A$ \textbf{in} $M$. Note that in this case $A$ is in fact a definable set.

\begin{proposition}\label{TP2forK}
	The formula $\varphi(x;y)$, defining the relation $R$ as in section \ref{HFWT}, has $\TP$.
\end{proposition}
\begin{proof}
Let $\Theta$ be a well-ordered subset of $\Gamma$ and consider the sets \[t^\Theta=\{t^\theta:\theta \in \Theta\}, \text{ and } B= \left\{\sum_{\delta\in \Delta} t^\delta: \Delta \subseteq \Theta\right\}.\] 
It is clear then that $\varphi(x;y)$ and $B$ only shatter $t^\Theta$, and thus by Lemma \ref{TP2Char} the formula $\varphi(x;y)$ has $\TP$.
\end{proof}

\begin{corollary}
The formula $\varphi(x;y)$, defining the relation $R$ as in section \ref{HFWT}, has $\IP$.
\end{corollary}
\begin{proof}
The result follows directly from proposition \ref{TP2forK} and lemma \ref{TP2IP}.
\end{proof}

\section{Undecidability for some Differential Hahn Fields}
\noindent
In this section we consider the Hahn field $K=C((t^\Gamma))$ where $C$ is a field of characteristic zero, and $\Gamma$ is either $\Z$ or $\Q$. We equip $K$ with the derivation $t\frac{d}{dt}$ that is trivial on $C$ and sends $t^\gamma$ to $\gamma t^\gamma$. We prove below that the theory of the valued differential field $K$ is undecidable. This result is folklore for $\Gamma= \Z$ but we include it here since it is not easily found in the literature. 

\begin{lemma}
$\{c\in C:\ y^\dagger = cf^\dagger \text{ for some } y,f\in K^\times, f\not\asymp 1\}=\Q$
\end{lemma}
\begin{proof}
Let $y= at^r(1+\delta)$ and $f= bt^s(1+\epsilon)$ in $K^\times$ be such that $y^\dagger = cf^\dagger$ with $a,b,c\in C$, $r,s\in \Gamma$, $s\neq0$, and $\delta, \epsilon\in K^{\prec 1}$. Then $r +\frac{\delta'}{1+\delta} = c(s+ \frac{\epsilon'}{1+\epsilon})$, so $r= cs$ and thus $c=\frac{r}{s}\in \mathbb{Q}$.
For the other direction it is clear that if $c =\frac{p}{q}\in \Q$, then $y= t^p$ and $f= t^q$ satisfy $y^\dagger = cf^\dagger$.
\end{proof}

\begin{theorem}\label{undCtGam}
The theory of the valued differential field $K$ is undecidable. 
\end{theorem}
\begin{proof}
By the Lemma above the subring $\Q$ of $C$ is (existentially) definable in the valued differential field $C((t^\Gamma))$. Applying Julia Robinson's result from \cite{RJ} we obtain that the theory is undecidable.
\end{proof}
\noindent
Using the derivation $t\frac{d}{dt}$ is only done to simplify certain expressions: it is clear that Theorem \ref{undCtGam} goes through with $\frac{d}{dt}$ as the derivation instead of $t\frac{d}{dt}$.

\medskip\noindent
The valuation ring of the Hahn field $C((t^\Z))$ is definable in the field $C((t^\Gamma))$ by a theorem of Ax, (see for example \cite{FJ}). Therefore:
\begin{corollary} \label{undCtZ}
The theory of the differential field $C((t^\Z))$ is undecidable.
\end{corollary}

\chapter{Operators with Support}

\noindent
Throughout $C$ is an abelian group (additively written), $\fM$ is a monomial set, and $\fm$ and $\fn$ range over $\fM$. 
In the first section we define {\em strong\/} operators on the Hahn space 
$C[[\fM]]$ and prove Lemma \ref{sumExt} to the effect that certain maps from $C\fM$ to $C[[\fM]]$ can be extended to strong operators on $C[[\fM]]$. This is useful in constructing strong derivations, see for example the proof of Lemma \ref{extder}. 

In the second section $C$ is a commutative ring and $\fM$ a monomial group. This allows us to introduce the more restricted
notions of {\em supported\/} operators and {\em small\/} operators on the Hahn ring $C[[\fM]]$. Corollary~\ref{cnPn} says how to use small operators in order to construct supported operators; this will be needed in the proof of Lemma \ref{OperatorLemma}. 

\section{Strong Operators} 
\noindent
Consider the Hahn space $C[[\fM]]$. 
An {\bf operator on $C[[\fM]]$} is by definition a map  
\[P\ :\ C[[\fM]]\rightarrow C[[\fM]].\]
An operator $P$ on $C[[\fM]]$ is said to be 
\textbf{additive} if  $P(f+g) = P(f)+P(g)$ for all $f,g\in C[[\fM]]$. If $P$ and $Q$ are additive operators on $C[[\fM]]$, then so are $P+Q$ and $-P$ defined by $(P+Q)(f):=P(f)+Q(f)$ and 
$(-P)(f):= -P(f)$. The set of additive operators on $C[[\fM]]$ with the above addition and multiplication given by composition is a ring with the null operator $O$ as its zero element,
the identity operator $I$ on $C[[\fM]]$ as its identity element, and
$P+(-P)=O$ for $P$ in this ring.  
 We call an operator $P$ on $C[[\fM]]$ \textbf{strongly additive}, or \textbf{strong} for short, if for every summable family $(f_i)$ in $C[[\fM]]$  the family $(P(f_i))$ is summable and $\sum P(f_i) = P(\sum f_i)$. Strong operators on $C[[\fM]]$ are additive, and the null operator $O$ and the identity operator 
$I$ are 
strong. The following is a routine consequence of the definition of ``strong''. 

\begin{lemma}
If $P,Q$ are strong operators on $C[[\fM]]$, then so are $-P$, $P+Q$, $PQ$.
\end{lemma}

\noindent
Thus the set of strong operators on $C[[\fM]]$ is a subring of the ring 
of additive operators on $C[[\fM]]$.
Let $(P_i)_{i\in I}$ be a family of strong operators on $C[[\fM]]$. Then we say that {\bf $\sum_{i\in I} P_i$ exists } if 
 for all $f\in C[[\fM]]$ the family $(P_i(f))_{i\in I}$ 
 is summable.
If $\sum_{i\in I} P_i$ exists, then the operator $f\mapsto \sum_i P_i(f)$ 
on $C[[\fM]]$ is additive, and is denoted by $\sum_{i\in I} P_i$ (or by $\sum P_i$ if $I$ is clear from the context). We do not claim that if 
$\sum_{i\in I} P_i$ exists, then $\sum_{i\in I} P_i$ is strong.

\begin{lemma} Let $(P_i)_{i\in I}$ be a family of strong operators on $C[[\fM]]$. 
\begin{enumerate}
    \item[\rm{(i)}] If $I$ is finite, then $\sum_{i\in I} P_i$ exists, and
   $\sum P_i$
    is the usual finite sum of additive operators. 
    \item[\rm{(ii)}] If $I$ and $J$ are disjoint sets, and $(P_j)_{j\in J}$ is also a family of strong operators on $C[[\fM]]$, and both $\sum_{i\in I} P_i$, and $\sum_{j\in J} P_j$ exist, then $\sum_{k\in I\cup J} P_k$ exists and equals $\sum_i P_i + \sum_j P_j$. 
    \item[\rm{(iii)}] If $\sigma : I \rightarrow J$ is a bijection and $\sum_{i\in I} P_i$ exists, then $\sum_{j\in J} P_{\sigma\inv(j)}$ exists and equals  $\sum_{i\in I} P_i$.
\end{enumerate}
\end{lemma}

\begin{lemma}\label{strpr}
Let $(P_i)_{i\in I}$ be a family of strong operators on $C[[\fM]]$ such that $\sum_i P_i$ exists, and let $Q$ be a strong operator on $C[[\fM]]$. Then 
$\sum_i P_iQ$ and $\sum_i QP_i$ exist, and are equal to
 $(\sum_i P_i)Q$ and $Q(\sum_iP_i)$, respectively.
\end{lemma}

\noindent
The next lemma modifies a result by 
van der Hoeven \cite[Proposition 3.5]{H} and is sometimes useful in
showing that an operator is strong. 


\begin{lemma}\label{sumExt} Let $C$ be a commutative ring with $1\ne 0$. Let
 $\phi: C\fM\rightarrow C[[\fM]]$ be a map such that: \begin{enumerate} 
\item[\rm{(i)}] $\phi((a+b)\fm) = \phi(a\fm) + \phi(b\fm)$ for all $a,b\in C$ and $\fm\in \fM$;
\item[\rm{(ii)}] $\supp(\phi(c\fm))\subseteq \supp(\phi(\fm)) \cup \{\fm\}$ for all $c\in C$ and $\fm\in \fM$;
\item[\rm{(iii)}] for every well-based $\fG\subseteq \fM$, 
the family $(\phi(\fg))_{\fg\in \fG}$ is summable.
\end{enumerate}
Then $\phi$ extends uniquely to a strong operator on $C[[\fM]]$.
\end{lemma}

\begin{proof}
Let $f = \sum f_\fm \fm\in C[[\fM]]$. We first show that the family $(\phi(f_\fm \fm))_{\fm\in \supp(f)}$ is summable. Indeed, by (ii),
\[\bigcup_{\fm\in \supp(f)}\supp(\phi(f_m\fm))\ \subseteq\  \left(\bigcup_{\fm\in\supp(f)}\supp \phi(\fm)\right)\cup \supp(f).\]
Since the right-hand side is well-based by (iii), so is the left-hand side.

Let any $\fn$ be given. Then by (ii) again,
\begin{align*} \{\fm\ \in \supp(f):\ 
\fn \in \supp(\phi(f_\fm\fm))\} \ &\subseteq\ \{\fm\in \supp(f):\ 
\fn \in \supp(\phi(\fm)) \cup \{\fm\}\}\\
     &\subseteq\  \{\fm\in \supp(f):\ 
\fn \in \supp(\phi(\fm))\} \cup \{\fn\}.
\end{align*}
The right-hand side is finite by (iii), and so is the left-hand side. Thus 
$\phi$ extends to an operator $\psi$ on
$C[[\fM]]$ by $\psi(f) := \sum_{\fm\in \supp(f)} \phi(f_\fm\fm)= \sum_{\fm} \phi(f_\fm \fm)$.
Now let $(f_i)_{i\in I}$ be a summable family in $C[[\fM]]$, and set $\fG:=\bigcup_{i\in I}\supp(f_i)$. We claim that $\big(\phi(f_{i,\fg}\fg)\big)_{(i,\fg)\in I\times \fG}$ is summable. First, 
\[\bigcup_{(i,\fg)\in I\times \fG} \supp(\phi(f_{i,\fg}\fg))\ \subseteq\  \bigcup_{\fg\in \fG}\supp(\phi(\fg)) \cup \fG,\]
so the left-hand side is well-based.
Secondly, let $\fn$ be given. Then
\[\{\fg\in \fG:\fn\in \supp(\phi(f_{i,\fg}\fg))\}\subseteq  \{\fg\in\fG:\fn\in\supp(\phi(\fg))\}\cup \{\fn\},\]
so the left-hand side is finite.
For $i\in I$ and $\fg\in \fG$ we have
\[\fn\in \supp(\phi(f_{i,\fg}\fg))\ \Rightarrow\ f_{i,\fg}\ne 0\ \Rightarrow\ 
\fg\in \supp f_i\]
and that  
for a fixed $\fg\in \fG$ the set $\{i\in I: \fg\in\supp(f_i)\}$ is finite.
 Thus the set
\[\{(i,\fg)\in I\times\fG:\  \fn\in \supp(\phi(f_{i,\fg}\fg))\}\]
is finite, and
our claim that $\big(\phi(f_{i,\fg}\fg)\big)_{(i,\fg)\in I\times\fG}$ is summable
has been established. 
It follows that 
$(\psi(f_i))_{i\in I}$ is summable and
\[\sum_{i\in I}\psi(f_i)\ =\ \sum_{i\in I} \sum_{\fg\in\fG}\phi(f_{i,\fg}\fg)\ =\ \sum_{(i,\fg)\in I\times\fG}\phi(f_{i,\fg}\fg)\ =\ \sum_{\fg\in\fG}\sum_{i\in I}\phi(f_{i,\fg}\fg)\]
\[=\ \sum_{\fg\in\fG}\phi\left(\big(\sum_{i\in I}f_{i,\fg}\big)\fg\right)\ =\ \psi\left(\sum_{\fg\in\fG} \big(\sum_{i\in I}f_{i,\fg}\big)\fg\right)\ =\ \psi\left(\sum_{i\in I}f_i\right),\]
showing that $\psi$ is strong.
\end{proof}

\subsection*{Strongly additive maps}
For use in the last chapter we now consider a second monomial set $\fN$ and define a map
$P: C[[\fM]]\to C[[\fN]]$ to be {\bf strongly additive} if for every summable family $(f_i)$ in $C[[\fM]]$  the family $(P(f_i))$ is summable in $C[[\fN]]$ and $\sum P(f_i) = P(\sum f_i)$. Thus the trivial map sending every $f\in C[[\fM]]$ to $0\in C[[\fN]]$ is strongly additive. 

Suppose $P$ as above is strongly additive. Then $P$ is additive: $P(f+g)=P(f)+P(g)$ for $f,g\in C[[\fM]]$. Moreover, $-P$ is strongly additive, and if $Q: C[[\fM]]\to C[[\fN]]$ is also a strongly additive map, then $P+Q$ is strongly additive. Given also a monomial set $\fG$ and a strongly additive map $Q:C[[\fG]]\to C[[\fM]]$, the composed map
$P\circ Q: C[[\fG]] \to C[[\fN]]$ is strongly additive.

\section{Supported Operators and Small Operators}
\noindent
In this section $C$ and $\fM$ are equipped with product operations that 
make $C$ a commutative ring and $\fM$ a 
monomial group.
To say that an operator $P$ on $C[[\fM]]$ {\bf has support $\fG$\/} will mean that $P$ is additive, $\fG$ is well-based, and $\supp P(g) \subseteq \fG\supp(g)$ for all $g\in C[[\fM]]$.
An operator on $C[[\fM]]$ with support $\fG$ for some $\fG$
is said to be a {\bf supported operator\/}.

Let $f\in C[[\fM]]$. Then the operator
$g\mapsto fg$ on $C[[\fM]]$ has support $\supp f$. In particular, it is additive, and so
we have for each additive operator $P$ on $C[[\fM]]$ the additive operator $fP$ on $C[[\fM]]$ given by $(fP)(g):= f\cdot P(g)$.
In particular, the null operator $O$ has support $\emptyset$, and the identity operator $I$ has support $\{1\}$. Here is a key fact about supported operators:

\begin{lemma}\label{suppstrong} Every supported operator on $C[[\fM]]$ is strong.
\end{lemma}
\begin{proof} Suppose the operator $P$ on $C[[\fM]]$ has support $\fG$, and let $(f_i)\in C[[\fM]]^I$ be a summable family. Then
\[\bigcup_i \supp P(f_i)\ \subseteq\ \fG\cdot \bigcup_i \supp f_i\]
and the latter is a well-based subset of $\fM$. Let
$\fm \in \fM$ be given. Then there are only finitely many pairs $(\fg,\fh)\in \fG\times \bigcup_i \supp(f_i)$ with $\fm=\fg\fh$. This yields a finite $I(\fm)\subseteq I$ such that $i\in I(\fm)$
whenever $i\in I$, $(\fg,\fh)\in \fG\times \supp(f_i)$, and
$\fm=\fg\fh$. In particular, for every $i\in I$
with $\fm\in \supp P(f_i)$ we have $i\in I(\fm)$. 
Thus $\big(P(f_i)\big)$ is summable. Also $f = g+h$ with $f:=\sum_i f_i$, $g:=\sum_{i\in I(\fm)}f_i$ and $h:= \sum_{i\notin I(\fm)}f_i$, so
$P(f)=P(g)+P(h)$ with $\fm\notin \supp P(h)$. Hence
$$P(f)_{\fm}\ =\ P(g)_{\fm}\ =\ 
\sum_{i\in I(\fm)}P(f_i)_{\fm}\ =\ \sum_i P(f_i)_{\fm}\ =\ \big(\sum_i P(f_i)\big)_{\fm},$$
and thus $P(f)=\sum_i P(f_i)$. 
\end{proof}

\noindent
The following two lemmas are rather obvious:

\begin{lemma}\label{suppobv1} If $c\in C$ and the operators $P,Q$ on $C[[\fM]]$ have support
$\fG$ and $\fH$, respectively, then $cP$, $P+Q$, $PQ$ have support
$\fG$, $\fG\cup \fH$, $\fG\fH$, respectively.
\end{lemma}

\noindent
Thus the set of supported operators on $C[[\fM]]$ is a subring
of the ring of strong operators on $C[[\fM]]$. This subring also has a useful infinitary
property:

\begin{lemma}\label{suppobv2} Suppose $\fG\subseteq \fM$ is well-based, $(P_i)$ is a family of operators on $C[[\fM]]$, all having support 
$\fG$, and $\sum_i P_i$ exists. Then
$\sum_i P_i$ has support $\fG$. If in addition $I = \bigcup_{j\in J} I_j$ with pairwise disjoint $I_j$, then $\sum_{i\in I_j}P_i$ exists  
for every $j\in J$, and $\sum_{j\in J}\big(\sum_{i\in I_j}P_i\big)$ exists and equals 
$\sum_{i \in I }P_i$. 
\end{lemma}
 
\noindent
The next result requires a bit more effort:

\begin{lemma}\label{suppnonobv} Let $\fG, \fH\subseteq \fM$ be well-based, let $(P_i)_{i\in I}$ be a family of operators on $C[[\fM]]$, all having support $\fG$, and let $(Q_j)_{j\in J}$ be a family of operators on $C[[\fM]]$, all having support $\fH$. Assume also that $\sum_i P_i$ and $\sum_j Q_j$ exist.
Then $\sum_{i,j} P_iQ_j$ exists and equals $(\sum_iP_i)(\sum_j Q_j)$.
\end{lemma}
\begin{proof} Let $f\in C[[\fM]]$. Then $\supp P_iQ_j(f)\subseteq \fG\fH\supp f$, so $\bigcup_{i,j} \supp P_iQ_j(f)$ is well-based. Let $\fm\in \fM$ be given. By Lemma~\ref{fGfN} the set 
\[T\ :=\ \{(\fg, \fh, \fn)\in \fG\times \fH \times \supp f:\ \fg\fh\fn=\fm\}\]
is finite. As $\sum_jQ_j(f)$ exists, this yields a finite $J(\fm)\subseteq J$ such that $j\in J(\fm)$ whenever $j\in J$, $(\fg, \fh, \fn)\in T$, and 
$\fh\fn\in \supp Q_j(f)$. Now $P_iQ_j(f)=P_i(Q_j(f))$, so $j\in J(\fm)$ whenever $\fm\in \supp P_iQ_j(f)$. For each $j\in J(\fm)$ there are only
finitely many $i$ with $\fm\in \supp P_i(Q_j(f))$, so there are only
finitely many $(i,j)\in I\times J$ with $\fm\in \supp P_iQ_j(f)$. 
We have now shown that  $\sum_{i,j} P_iQ_j$ exists. That it equals $(\sum_iP_i)(\sum_j Q_j)$ is now an easy consequence of Lemmas~\ref{suppobv2} 
and  ~\ref{strpr}.
\end{proof}

\noindent 
An operator $P$ on $C[[\fM]]$ is said to be \textbf{small}
\footnote{Small operators were defined in \cite[p. 66]{DMM} without requiring additivity, but this invalidates the assertion there about the inverse of $I-P$. Fortunately, this assertion is only used later in that paper for additive $P$, for which it is correct.}, if $P$ has support
$\fG$ for some $\fG\prec 1$. Note that if
$c\in C$, and $P$ and $Q$ are small operators on $C[[\fM]]$, then $cP$, $P+Q$, and $PQ$ are small operators on $C[[\fM]]$. Small operators are particularly useful in constructing inverses of operators:

\begin{corollary}\label{cnPn} Let $P$ be a small operator. Then for any sequence $(c_n)$ in $C$ the sum $\sum_{n} c_nP^n$ exists and is supported; such a sum 
is small if $c_0=0$. Moreover, the operator
$I-P$ is bijective with inverse $\sum_n P^n$.  
\end{corollary}
\begin{proof}
Let $P$ have support $\fG\prec 1$. Then $P^n$ has support $\fG^n$.
Let $(c_n)$ be a sequence in $C$. Then for $g\in C[[\fM]]$, 
\[\bigcup_{n} \supp c_nP^n(f)\ \subseteq\ \big(\bigcup_{n} \fG^n\big)\cdot \supp f\ =\ \fG^*\cdot \supp f,\]
the latter is well-based, and for every $\fm\in \fM$ there are only finitely many $n$ with $\fm\in \fG^n\supp f$ by \ref{fGfN}
and \ref{fGstar}, so $\sum_nc_n P^n$ exists. Note that all $c_nP^n$ have support $\fG^*$. The rest is clear from Lemmas~\ref{suppstrong}, \ref{suppobv1}, and ~\ref{suppobv2}. 
\end{proof}

\chapter{Derivations on Hahn Fields}\label{derha}

\noindent
In this chapter $\smallk$ is a {\em differential\/} field with derivation $\der$ (possibly trivial), and $\Gamma$ is an ordered abelian group. Let $\alpha,\beta, \gamma$ range over $\Gamma$, and fix an additive map $c:\Gamma \rightarrow \smallk$. 

In the first section we consider how the map $c$ gives rise to a strong derivation on the Hahn field $\smallk((t^\Gamma))$ extending the given derivation $\der$ of $\smallk$. 

In the second section we
consider how to solve differential equations $y-ay'=f$ for $a,f\in \smallk((t^\Gamma))$ with $a\prec 1$. The main result in this chapter is Theorem~\ref{thmA}, which is about preserving truncation closedness under adjoining solutions of such equations to
truncation closed differential subfields of $\smallk((t^\Gamma))$. 
This will be important in the next chapter. 

In the third section we summarize a general construction from
\cite{DMM} that involves (partial) exponential maps on Hahn fields. 
This will also be needed in the next chapter.  

\section{Extending the Derivation on the Coefficient Field to the Hahn Field}\label{extderco}

\noindent
The map $c$ allows us to extend $\der$ to a derivation on the field $\smallk((t^\Gamma))$, also denoted by $\der$, by declaring for $f=\sum_{\gamma}f_\gamma t^\gamma\in \smallk((t^\Gamma))$ that
\[\der \left(f\right)\ :=\  \sum_{\gamma} \Big(\der (f_\gamma) +f_\gamma c(\gamma)\Big)t^\gamma.\]
Thus $(t^\gamma)'=c(\gamma)t^\gamma$. 
For $f\in \smallk((t^\Gamma))$ we have $\supp f'\subseteq \supp f$, so $f'\preceq f$. In particular, $\der$ is small, and the natural field isomorphism from 
$\smallk$ onto the (differential) residue field $\bigO/\smallo$ of $\smallk((t^\Gamma))$ is an isomorphism of differential fields. 
 What are the constants of $\smallk((t^\Gamma))$? In the  extreme case 
$c(\Gamma)=\{0\}$, the constant field of $\smallk((t^\Gamma))$ is
$C_{\smallk}((t^\Gamma))$. For us the opposite extreme is more relevant:

\begin{lemma}\label{cgood}
The following are equivalent:
\begin{enumerate}
	\item[\rm{(i)}] the constant field of $\smallk((t^\Gamma))$ is the same as the constant field of $\smallk$; 
    \item[\rm{(ii)}] $c$ is injective and $c(\Gamma) \cap \smallk^\dagger = \{0\}$;
    \item[\rm{(iii)}] for all $f\in \smallk((t^\Gamma))^\times$ with
$f\not\asymp 1$ we have $\der(f)\asymp f$, hence
$f^\dagger \asymp 1$.
\end{enumerate}
\end{lemma}

\noindent
Thus the valued differential field $\smallk((t^\Gamma))$ is asymptotic under the conditions of Lemma~\ref{cgood}.

\medskip\noindent
Sometimes it is more natural to consider a Hahn field $\smallk[[\fM]]$, where $\fM$ is not given in the form $t^\Gamma$, and then
a map $c: \fM\to \smallk$ is said to be \textbf{additive} if $c(\fm\fn)=c(\fm) + c(\fn)$
for all $\fm, \fn\in \fM$. Again, such $c$ allows us to extend $\der$ to a derivation $\der$ on the field $\smallk[[\fM]]$ by 
\[\der \left(f\right)\ :=\  \sum_{\fm} \Big(\der (f_\fm) +f_\fm c(\fm)\Big)\fm.\] 
This operator $\der$ on $\smallk[[\fM]]$ has support $\{1\}$, and so is strong. 

 \medskip\noindent
{\bf Examples.} For $\fM=x^{\mathbb{Z}}$ with $x\succ 1$ we have the usual derivation 
$\frac{d}{dx}$ with 
respect to $x$ on the field of Laurent series 
$\smallk[[x^{\mathbb{Z}}]]=\smallk((t^{\mathbb{Z}}))$ (with $t=x^{-1}$), but it is not of the form considered above.
The derivation $x\frac{d}{dx}$, however, does have the form above, with the
trivial derivation on $\smallk$ and $c(x^k)=k\cdot 1\in \smallk$ for 
$k\in \mathbb{Z}$. Likewise for $\fM=x^{\mathbb{Q}}$ ($x\succ 1$), 
and $\operatorname{char}(\smallk)=0$: then  $x\frac{d}{dx}$ has the above 
form, with the trivial derivation on $\smallk$ and 
$c(x^q)=q\cdot 1\in \smallk$ for 
$q\in \mathbb{Q}$.

\bigskip\noindent
We now return to the setting of  $\smallk((t^\Gamma))$, and observe that
$\der(f|_\gamma) = \der(f)|_\gamma$ for $f\in \smallk((t^\Gamma))$. Thus by Proposition \ref{D}:

\begin{corollary}\label{lem1}
If $R$ is a truncation closed subring of $\smallk((t^\Gamma))$ and $f\in R$ is such that $\der(g)\in R$ for every proper truncation $g$ of $f$, 
then all proper truncations of $\der(f)$ lie in $R$ and thus 
$R[\der(f)]$ is truncation closed.
\end{corollary}

\begin{lemma}\label{DifFieldGenTrunc}
Let $R$ be a truncation closed subring of $\smallk((t^\Gamma))$. 
Then the differential subring of $\smallk((t^\Gamma))$ generated by $R$ is 
truncation closed. 
\end{lemma}

\begin{proof}
Let $R_0 = R$ and $R_{n+1} = R_n[\der(f): f\in R_n]$.
Assume inductively that $R_n$ is a truncation closed subring of 
$\smallk((t^\Gamma))$ . 
For $f\in R_n$ and $\gamma\in \Gamma$ we have $\der(f)|_\gamma = \der(f|_\gamma)$, so all truncations of $\der(f)$ lie in $R_{n+1}$, and thus $R_{n+1}$ is truncation closed by Proposition~\ref{D}(ii). 
Since the differential ring generated by $R$ is $\bigcup_{n} R_n$, and a union of truncation closed subsets of $\smallk((t^\Gamma))$ is truncation closed, we conclude that the differential subring of $\smallk((t^\Gamma))$ generated by $R$ is truncation closed. 
\end{proof}

\noindent
Note that Lemma~\ref{DifFieldGenTrunc} goes through with ``subfield'' instead of
``subring''.


\section{Adjoining Solutions of $y-ay'=f$}\label{adj}
\noindent
We wish to preserve truncation closedness under
adjoining solutions of differential equations 
\[y-ay'\ =\ f \qquad(a,f \in \smallk((t^\Gamma)),\ a\prec 1).\]
This differential equation is expressed more suggestively as
$(I-a\der)(y)=f$, where $I$ is the identity operator on  
$\smallk((t^\Gamma))$ and $a\der$ is considered as a (strong) 
operator on $\smallk((t^\Gamma))$ in the usual way. Note that $a\der$ 
is even small as defined earlier, hence
$I-a\der$ is bijective, with inverse 
\[(I-a\der)\inv\ =\ \sum_n (a\der)^n\]
Thus the above differential equation has a unique solution 
$y=(I-a\der)\inv(f)$ in $\smallk((t^\Gamma))$. This is why we now turn our attention to the operator $(I-a\der)\inv$.

\medskip\noindent
For $n\geq 1$, $0\leq m\leq n$ we define $G^n_m(X)\in \mathbb{Z}\{X\}$ recursively as follows:
\begin{itemize}
\item $G^n_0 = 0$,
\item $G^n_n = X^n$,
\item $G^{n+1}_m = X(\der(G^n_m) +G^{n}_{m-1})$ for $1\leq m\leq n$. 
\end{itemize}
This recursion easily gives
\[(a\der)^n\ =\ \sum_{m=1}^n G^n_m(a)\der^m \qquad(n\ge 1),\]
hence
\begin{equation}
(I-a\der)\inv = I + \sum^{\infty}_{n=1} \sum_{m=1}^n G^n_m(a)\der^m.
\end{equation}

\noindent Since $G^n_m(X)$ is homogeneous of degree $n$, we have
$G^n_m(a)\in \smallk t^{n\alpha}$ for $a\in \smallk t^\alpha$.

\begin{lemma}\label{OpLemma1}
Suppose $R$ is a truncation closed differential subring of
$\smallk((t^\Gamma))$, $a\in R\cap\smallk t^\Gamma$, $a\prec 1$, $f\in R$, and  $(I-a\der)\inv(g)\in R$ 
for all $g\truncof f$. Then all proper truncations of $(I-a\der)\inv (f)$ lie in $R$.
\end{lemma}

\begin{proof} We have 
$(I-a\der)\inv(f) = f+  \sum^{\infty}_{n=1} \sum_{m=1}^n G^n_m(a)\der^m(f)$.
Also $a=a_{\alpha}t^\alpha$ with $\alpha>0$, hence $\supp  \left(G^n_m(a)\der^m(f)\right)\subseteq n\alpha+\supp f$ for 
$1\le m\le n$. Consider a proper truncation $(I-a\der)\inv(f)|_{\gamma}$
of $(I-a\der)\inv(f)$; we have to show that this truncation lies in $R$.
The truncation being proper gives $N\in \mathbb{N}^{\ge 1}$ and $\beta\in \supp f$ with 
$\gamma\le N\alpha +\beta$. 
Let $f_1:= f|_\beta$ and $f_2:=f-f_1$. Then $f_1, f_2\in R$ and
\[(I-a\der)\inv(f)\ =\ (I-a\der)\inv(f_1) + f_2 + \sum^{\infty}_{n=1} \sum_{m=1}^n G^n_m(a)\der^m(f_2).\]
Using $\supp f_2\ge \beta$ and truncating at $\gamma$ gives 
\[(I-a\der)\inv(f)|_{\gamma}\ =\ (I-a\der)\inv(f_1)|_\gamma + f_2|_\gamma + \left.\left(\sum^{N-1}_{n=1} \sum_{m=1}^n G^n_m(a)\der^m(f_2)\right)\right|_\gamma,\]
which lies in $R$, since $f_1\truncof f$ and thus $(I-a\der)\inv(f_1)\in R$. 
\end{proof}

\noindent
The key inductive step is provided by the next lemma. 

\begin{lemma}\label{OperatorLemma} Let $R$ be a truncation closed differential subring of $\smallk((t^\Gamma))$. Let $a,f\in R$ be such that $a\prec 1$ and for all $b,g\in R$,
\begin{enumerate}
\item[\rm{(i)}] $g\truncof f \Rightarrow (I-a\der)\inv (g)\in R$,
\item[\rm{(ii)}] $b\truncof a \Rightarrow (I-b\der)\inv(g)\in R$.
\end{enumerate}
Then all proper truncations of $(I-a\der)^{-1}(f)$ lie in $R$.
\end{lemma}

\begin{proof} Assume $(I-a\der)^{-1}(f)|_\gamma\truncof (I-a\der)^{-1}(f)$. Then we have
$\alpha\in \supp a$ and $\beta\in \supp f$, and $N\in \mathbb{N}^{\ge 1}$ such that
 $\gamma\leq N\alpha + \beta$. 
Put $a_1:=a|_{\alpha}$, $a_2:= a-a_1$, $f_1:=f|_\beta$, and $f_2:=f-f_1$. 
Set $P=a_1\der$, $Q=a_2\der$, so $a\der=P+Q$, $P,Q$ are small, and  
\begin{align*} (I-a\der)\inv(f)\ &=\ (I-a\der)\inv(f_1) + (I-a\der)\inv(f_2)\\
 &=\ (I-a\der)\inv(f_1) + \big(I-(P+Q)\big)\inv(f_2).\end{align*}
Now $Q(I-P)\inv=\sum_n QP^n$ is small, and we have the identity
\begin{align*} I-(P+Q)\ &=\ \big(I-Q(I-P)\inv\big)(I-P),\ \text{ so}\\ 
\big(I-(P+Q)\big)\inv\ &=\ (I-P)\inv \big(I-Q(I-P)\inv\big)\inv\\
&=\ \sum_{m=0}^\infty (I-P)\inv\big(Q(I-P)\inv\big)^m, \text{ hence}\\
(I-a\der)\inv(f)\ &=\ (I-a\der)\inv(f_1) + \sum_{m=0}^\infty (I-P)\inv \big(Q(I-P)\inv\big)^m(f_2). 
\end{align*}
For $m\ge N$ we have $\supp \big(Q(I-P)\inv\big)^m(f_2)\ge \gamma$, so truncating at $\gamma$ yields
\[(I-a\der)\inv(f)|_{\gamma}\ =\ (I-a\der)\inv(f_1)|_{\gamma} + \left.\left(\sum_{m=0}^{N-1}(I-P)\inv \big(Q(I-P)\inv\big)^m(f_2)\right)\right|_\gamma.\] 
Since $f_1\truncof f$ the first summand of the right hand side lies in $R$. Since $a_1\truncof a$, we have $(I-P)\inv(h)\in R$ for all $h\in R$, and thus, using $a_2, f_2\in R$, 
\[\sum_{m=0}^{N-1}(I-P)\inv (Q(I-P)\inv)^m (f_2)\in R.\]
Therefore $(I-a\der)^{-1}(f)|_\gamma\in R$. 
\end{proof}

\begin{theorem}\label{thmA} Let $E$ be a truncation closed differential subfield
of $\smallk((t^\Gamma))$. Let $\widehat{E}$ be the smallest differential subfield of 
$\smallk((t^\Gamma))$ that contains $E$ and is closed under $(I-a\der)\inv$ for all $a\in \widehat{E}^{\prec 1}$. Then $\widehat{E}$ is truncation closed.
\end{theorem}

\begin{proof}
Let $F$ be a maximal truncation closed differential subfield of $\widehat{E}$ 
containing $E$. (Such $F$ exists by Zorn's Lemma.) It suffices to show that 
$F= \widehat{E}$. Assume $F\ne \widehat{E}$. Then there exist  $a\in F^{\prec 1}$ and $f\in F$ with 
$(I-a\der)\inv(f) \notin F$. Take such $a$ and $f$ for which
$(\alpha, \beta)$ is minimal for lexicographically ordered pairs of ordinals,
where $\alpha$ is the order type of the support of $a$ and $\beta$ that of
 $f$. By minimality of $(\alpha,\beta)$ we can apply Lemma \ref{OperatorLemma} to $F,a$, and $f$ to get that $F\big((I-a\der)\inv(f)\big)$ is 
truncation closed, 
using also parts (i) and (ii) of Proposition~\ref{D}. Since $F\big((I-a\der)\inv(f)\big)$ is a 
differential subfield of $\widehat{E}$, this contradicts the maximality of $F$. 
\end{proof}

\subsection*{Complementing the previous results}
Suppose $\smallk_1$ is a differential 
subfield of $\smallk$ and $E$ is a truncation closed
subfield of 
$\smallk((t^\Gamma))$ such that $a\in \smallk_1$ and
$c(\gamma)\in \smallk_1$ whenever $at^\gamma\in E$, $a\in \smallk^\times$, $\gamma\in \Gamma$. Let $\Gamma_1$ be the subgroup of $\Gamma$ generated by the
$\gamma\in \Gamma$ with $at^\gamma\in E$ for some $a\in \smallk^\times$. 
In connection with Lemma~\ref{DifFieldGenTrunc} we note:

\medskip\noindent
{\em The differential subfield of $\smallk((t^\Gamma))$ generated by $E$ is 
contained in  $\smallk_1((t^{\Gamma_1}))$.}

\medskip\noindent
This is because $E\subseteq \smallk_1((t^{\Gamma_1}))$ and $\smallk_1((t^{\Gamma_1}))$ is a differential subfield of 
$\smallk((t^\Gamma))$. Likewise we can complement Theorem~\ref{thmA}:

\medskip\noindent
{\em If
$\der E\subseteq E$ and $\widehat{E}$ is the smallest 
differential subfield of $\smallk((t^\Gamma))$ that contains $E$ and is closed under
$(I-a\der)\inv$ for all $a\in \widehat{E}^{\prec 1}$,   
then $\widehat{E}\subseteq \smallk_1((t^{\Gamma_1}))$.}

\medskip\noindent
This is because $\smallk_1((t^{\Gamma_1}))$ is closed under
$(I-a\der)\inv$ for all $a\in\smallk_1((t^{\Gamma_1}))^{\prec 1}$. 

\section{Exponentiation} 

\noindent
Our aim is to apply the material above
to the differential field $\T_{\exp}$ of purely exponential transseries. 
The construction of $\T_{\exp}$ involves an
iterated formation of Hahn fields, where at each step we apply
 the following general procedure stemming from \cite{DMM} (but copied from \cite[Appendix A]{ADH}). 

\medskip\noindent
Define a {\bf pre-exponential ordered field}
to be a tuple~$(E, A, B, \exp)$ such that
\begin{enumerate}
\item $E$ is an ordered field;
\item $A$ and $B$ are additive subgroups of $E$ with $E=A\oplus B$ and $B$
convex in $E$;
\item $\exp\colon B\to E^{\times}$ is a strictly increasing group morphism (so $\exp(B)\subseteq E^>$).
\end{enumerate}
Let $(E,A,B,\exp)$ be a pre-exponential ordered field. We view $A$ as the part of $E$ where exponentiation is not yet defined, and accordingly we introduce
a ``bigger'' pre-exponential ordered field $(E^*,A^*,B^*,\exp^*)$ as follows:
Take a {\em multiplicative\/} copy $\exp^*(A)$ of the ordered additive
group $A$ with order-preser\-ving isomorphism 
$\exp^*\colon A\to\exp^*(A)$,
and put $E^*\ :=\ E[[\exp^*(A)]]$. Viewing $E^*$ as an ordered Hahn field over 
the ordered coefficient field $E$, we set
\[ A^*\ :=\ E[[\exp^*(A)^{\succ 1}]],\qquad B^*\ :=\ (E^*)^{\preceq 1}\ =\ E\oplus 
(E^*)^{\prec 1}\ =\ A \oplus B \oplus (E^*)^{\prec 1}.\] 
Note that $\exp^*(A)^{\succ 1}=\exp^*(A^{>})$. Next we extend
$\exp^*$ to $\exp^*\colon B^*\to (E^*)^{\times}$ by
\[\exp^*(a+b+\varepsilon)\ :=\ \exp^*(a)\cdot \exp(b)\cdot \sum_{n=0}^\infty
\frac{\varepsilon^n}{n!} \qquad(a\in A,\ b\in B,\ \varepsilon\in (E^*)^{\prec 1}).\]
Then $E\subseteq B^*=\operatorname{domain}(\exp^*)$, and $\exp^*$ extends $\exp$. Note that $E<(A^*)^>$ (but $\exp^*(E)$ is cofinal in $E^*$ if $A\ne \{0\}$). In particular, for $a\in A^{>}$, we have 
\[\exp^*(a)\in \exp^*(A^{>})\ \subseteq\ (A^*)^>,\ \text{  so $\exp^*(a)\ >\ E$.}\]

\medskip\noindent
Assume also that a derivation $\der$ on the field $E$ is given that respects
exponentiation, that is, $\der(\exp(b))=\der(b)\exp(b)$ for all 
$b\in B$. Then we extend $\der$ uniquely to a 
derivation on
the field $E^*$, also denoted by $\der$, by requiring 
that $\der$ is strong as an operator on the Hahn field $E^*$
over $E$, and that $\der(\exp^*(a))=\der(a)\exp^*(a)$
for all $a\in A$. This falls under the general construction at the beginning of this section with $\smallk=E$ and $\fM= \exp^*(A)$ by taking the additive function 
$c: \exp^*(A) \to E$ given by $c(\exp^*(a))=\der(a)$.  
This extended derivation on $E^*$ again respects
exponentiation in the sense that  
$\der(\exp^*(b))=\der(b)\exp^*(b)$ for all 
$b\in B^*$.

\subsection*{Exponential ordered fields} For the sake of completeness and because we use this notion in the next chapters
we define an {\bf exponential ordered field\/} to be an ordered field $E$ equipped with a strictly increasing group morphism 
$\exp: E \to E^{\times}$ from the additive group of $E$ into its multiplicative group. Note that then $\exp(E)\subseteq E^{>}$, and
that this is basically the case of a pre-exponential ordered field
$(E, A, B, \exp)$ with $A=\{0\}$ and $B=E$. 

\chapter{Increasing Unions of Differential Hahn Fields}

\noindent
As we mentioned in the Introduction, $\T_{\exp}$ is an increasing union of differential Hahn fields. Here we abstract from some details of its construction in order to
better focus on features that matter for truncation.
In the first section we present the general set-up of constructing
an increasing union $\smallk_{*}=\bigcup_n \smallk_n$ of
differential Hahn fields $\smallk_n=\smallk[[\fM^{(n)}]]$. In the second section we show how $\Texp$ fits into this setting.
In the third section we define {\em splitting\/} and show that  various
extension procedures preserve being truncation closed under a splitting assumption.
In the fourth section we extend the derivation of $\smallk_{*}$ to its completion $\smallk[[\fM]]$, where $\fM=\bigcup_n\fM^{(n)}$.
In the fifth section we show how to integrate in $\smallk_{*}$
if the latter is {\em transserial}. 
The main result of this chapter is Theorem~\ref{thmB} in the sixth section, to the effect that being truncation closed is preserved under an extension 
procedure that involves integration as well as adjoining solutions of certain subsidiary first-order linear differential equations. This will be crucial in the next (and last) chapter.  

\section{The Basic Set-Up}\label{basicsetup}

\noindent
Let $\smallk$ be a field of characteristic zero and $\fM$ a monomial group with distinguished subgroups $\fM_n$, $n=0,1,2,\dots$, such that, with
$\fM^{(n)}:=\fM_0 \cdots \fM_n\subseteq \fM$, we have:\begin{enumerate}
\item $\fM=\bigcup_n \fM^{(n)}$;
\item $\fM_m \prec \fm_n$ for all $m<n$ and $\fm_n\in \fM^{\succ 1}_n$.
\end{enumerate} 
Thus $\fM^{(n)}$ is a convex subgroup of $\fM$. We also have the direct product group $\fM_0 \times \cdots \times \fM_n$ with a group isomorphism 
\[\fM_0 \times \cdots \times \fM_n\to \fM^{(n)}, \quad (\fm_0,\dots, \fm_n) \mapsto \fm_0\cdots \fm_n\]
and for an element $(\fm_0,\dots, \fm_n)\in \fM_0\times \cdots \times \fM_n$ with $\fm_n\ne 1$ we have
\[\fm_0\cdots\fm_n\succ 1\ \Longleftrightarrow\ \fm_n\succ 1.\]
Setting $\smallk_n:= \smallk[[\fM^{(n)}]]$ we have 
$\smallk_0=\smallk[[\fM_0]]$ and field extensions
\[\smallk\ \subseteq\ \smallk_0\ \subseteq\ \smallk_1\ \subseteq\ \cdots\subseteq\ \smallk_n\ \subseteq\ \smallk_{n+1}\ \subseteq\ \cdots.\]
When convenient we identify $\smallk_{n}$ with the
Hahn field $\smallk_{n-1}[[\fM_n]]$ over $\smallk_{n-1}$ in the usual way
(where $\smallk_{-1}:=\smallk$ by convention). We set 
\[\smallk_{*}\ :=\ \bigcup_n \smallk_n, \quad\text{ a truncation closed subfield of the Hahn field }\smallk[[\fM]].\]
More generally, a set $S\subseteq \smallk_{*}$ is said to be truncation closed
if it is truncation closed as a subset of the Hahn field $\smallk[[\fM]]$. Recall from Chapter 3 that if $E\supseteq \smallk$ is a truncation closed subfield of $\smallk[[\fM]]$, then 
$\fM_E:=E\cap \fM=\supp E\cap \fM$ is a subgroup of $\fM$ with
$E\subseteq \smallk[[\fM_K]]$.

\subsection*{Closure under small exponentiation} The function $\exp: \smallk[[\fM]]^{\prec 1} \to \smallk[[\fM]]$ given by
$$\exp(\varepsilon)\ =\ \sum_{i=0}^{\infty} \varepsilon^i/i!$$ maps $\smallk_{*}^{\prec 1}$ bijectively onto $1+\smallk_{*}^{\prec 1}$. 
Thus if $E$ is a subfield of 
$\smallk_{*}$, then its small exponential closure $E^{\smallexp}$ is also a subfield of 
$\smallk_{*}$.  And if $E\supseteq \smallk$ is a truncation closed
subfield of $\smallk_{*}$, then $E^{\smallexp}\subseteq \smallk_{*}\cap \smallk[[\fM_E]]$ and $E^{\smallexp}$ is truncation closed. 

\subsection*{The valuation}
Take an order reversing group isomorphism $v: \fM \to \Gamma$ onto an
additive ordered abelian group $\Gamma$. Then $v$ extends to the valuation
$v:\smallk[[\fM]]^\times \to \Gamma$ given by $v(f)=v(\max \supp f)$. Note that for 
$a\in \smallk_n= \smallk_{n-1}[[\fM_n]]$ we have: 
\[a\prec_{\fM_n}1\ \Longleftrightarrow\ va>v(\fM^{(n-1)}) \quad(\text{with }\fM^{(-1)}:=\{1\}).\]
Using that the $\fM^{(n)}$ are convex in $\fM$
it follows that $\smallk_{*}$ is dense in the Hahn field
$\smallk[[\fM]]$ with respect to the valuation topology given by $v$. Note also that for $f\in \smallk[[\fM]]$,
$$ f\succ \fM^{(n-1)}\ \Longleftrightarrow\ f \succ \smallk_{n-1}.$$

\subsection*{Equipping $\smallk_{*}$ with a derivation}
We now assume that $\smallk$ is even a differential field, and that for
every $n$ there is given an additive map $c_n: \fM_{n} \to \smallk_{n-1}$.
Then we make $\smallk_n$ into a differential field by recursion on $n$: $\smallk_0=\smallk[[\fM_0]]$ is equipped with the derivation $\der_0$ given by the derivation of $\smallk$ and the additive map $c_0$, and
$\smallk_{n+1}=\smallk_{n}[[\fM_{n+1}]]$ is equipped with the derivation $\der_{n+1}$ given by the 
derivation $\der_n$ of $\smallk_{n}$
and the additive map $c_{n+1}: \fM_{n+1} \to \smallk_{n}$. Thus $\smallk_{n}$ is a differential field extension of $\smallk_{n-1}$. 
We make $\smallk_{*}$ into a differential field so that it contains every
$\smallk_n$ as a differential subfield. We let $\der$ denote the derivation of $\smallk_{*}$, so $\der$ extends each $\der_n$.  It follows easily that
$$(\fm_0\cdots \fm_n)^\dagger\ =\  c_0(\fm_0) +\cdots + c_n(\fm_n), \qquad (\fm_0\in \fM_0, \dots, \fm_n\in \fM_n).$$ 
This suggest we define the additive function $c: \fM \to \smallk_*$ by 
$$c(\fm_0\cdots \fm_n)\ :=\ c_0(\fm_0) +\cdots + c_n(\fm_n), \qquad (\fm_0\in \fM_0, \dots, \fm_n\in \fM_n),$$
so $c$ extends each $c_n$, and $\fm^\dagger=c(\fm)$ for all $\fm\in \fM$.

\begin{lemma}\label{small} If $f\in \smallk_{*}$ and $f\prec 1$, then $\der(f)\prec 1$. $($In other words: $\smallk_{*}$ has small derivation.$)$
\end{lemma}
\begin{proof} Let $f\in \smallk_n$ and $f\prec 1$. We show by induction on $n$ that $\der(f)\prec 1$. For $n=0$ this follows from $\der(f)\preceq f$. Let $n\ge 1$, so $f=\sum_{\fm\in \fM_n}f_{\fm}\fm$
with all $f_{\fm}\in \smallk_{n-1}$. Then 
$$f\ =\ f_1 + \sum_{\fm\in \fM_n^{\prec 1}}f_{\fm}\fm,\qquad f_1\prec 1,$$
and thus $\der(f_1)\prec 1$ by the inductive assumption. Also 
$$\der\big(\sum_{\fm\in \fM_n^{\prec 1}}f_{\fm}\fm\big)\ =\ \sum_{\fm\in \fM_n^{\prec 1}}\big(\der(f_{\fm})+c_n(\fm)f_{\fm}\big)\fm\ \prec\ 1,$$
which gives $\der(f)\prec 1$. 
\end{proof}

\begin{lemma}\label{daggerprec} Let $f\in \smallk_n^\times$. Then $f^\dagger\preceq g$ for some $g\in \smallk_{n-1}$.
\end{lemma}
\begin{proof} We have $f=f_{\fm}\fm(1+\varepsilon)$ with
$f_{\fm}\in \smallk_{n-1}^\times$, $\fm\in \fM_n$, and $\varepsilon\in \smallk_n^{\prec 1}$. It remains to use that 
$$f^\dagger\ =\ f_{\fm}^\dagger + c_n(\fm) + \frac{\der(\varepsilon)}{1+\varepsilon},$$
with $\der(\varepsilon)\prec 1$ by the previous lemma. 
\end{proof} 

\noindent
The next result is about the constant field of $\smallk_{*}$ and is an easy consequence of Lemma~\ref{cgood}.

\begin{corollary}\label{samecons} The following conditions are equivalent: \begin{enumerate}
\item[\rm{(i)}] for every $n$ the map $c_n: \fM_n\to \smallk_{n-1}$ is injective and $c_n(\fM_n)\cap \smallk_{n-1}^\dagger=\{0\}$;
\item[\rm{(ii)}] $\smallk_{*}$ has the same constant field as $\smallk$.
\end{enumerate}
\end{corollary}

\noindent
We obtained $\der_n$ as a strong operator on $\smallk_n$ when the latter is viewed as a Hahn field over
$\smallk_{n-1}$, but an easy induction on $n$ shows that $\der_n$ is actually a 
strong operator on the Hahn field $\smallk_n=\smallk[[\fM^{(n)}]]$ over $\smallk$. This fact will be used tacitly at various places in this chapter. Example:
$$ \der\big(\exp(\varepsilon)\big)\ =\ \der(\varepsilon)\cdot\exp(\varepsilon)\qquad \text{for }\varepsilon\in \smallk_{*}^{\prec 1}.$$
It follows that if $E$ is a differential subfield of $\smallk_{*}$
(that is, $\der E \subseteq E$), then so is $E^{\smallexp}$. 

One natural question we could ask at this point is the following:
Is the differential field generated by a truncation closed subset of $\smallk_*$ truncation closed? The answer is no, even in the case where $\smallk_{*}=\T_{\exp}$, as we shall see at the end of the next section.


\section{The Main Example: $\T_{\exp}$} \label{METEXP}

\noindent
We now indicate how
$\T_{\exp}$ is a union $\smallk_{*}$ as above. Here 
$\smallk=\R$ and $\fM_0= x^{\R}$, so 
$$\smallk_0\ =\ \R[[x^{\R}]]\ =\ \R((t^\R))\quad \text{ with }\ t\ :=\ x^{-1}.$$ We equip $\smallk_0$ with the strongly additive $\R$-linear derivation
$\der_0=x\cdot\frac{d}{dx}$, that is, for $f=\sum_r f_rx^r\in \smallk[[x^{\R}]]$ we have $\der_0 (f)=\sum_r rc_rx^r$.
Thus $\der_0$ is the derivation on the Hahn field $\R[[x^\R]]$
over $\R$ that corresponds to the trivial derivation on $\R$ and the function $c_0: x^{\R} \to \R$ given by $c_0(x^r)=r$. Next, we make
$\smallk_0$ into a pre-exponential ordered field
$(\smallk_0, A_0, B_0, \exp_0)$: \begin{enumerate}
\item for the ordering of $\smallk_0$ we take its ordering as an
ordered Hahn field over $\R$; in fact, we have no other choice, since $\smallk_0$ is real closed;
\item $A_0:= \R[[\fM_0^{\succ 1}]]$, and $B_0:=\R[[\fM_0^{\preceq 1}]]$; 
\item $\exp_0(r+\epsilon)\:=\ e^r\sum_{i=0}^\infty \epsilon^i/i!$.
\end{enumerate}
The derivation $\der_0$ respects exponentiation.
This is stage $0$ of our construction. At stage $n$ we have
a pre-exponential ordered field $(\smallk_n, A_n, B_n, \exp_n)$
with a derivation $\der_n$ on $\smallk_n$ that respects exponentiation. Moreover, $\smallk_n=\smallk_{n-1}[[\fM_n]]$
(with $\smallk_{-1}=\R$), 
$A_n=\smallk_{n-1}[[\fM_n^{\succ 1}]]$ and $B_n=\smallk_{n-1}[[\fM_n^{\preceq 1}]]$. Applying to $(\smallk_n, A_n, B_n, \exp_n)$
the construction that extends $(E,A,B,\exp)$ in the previous section to $(E^*, A^*, B^*, \exp^*)$ yields a pre-exponential ordered field
$(\smallk_{n+1}, A_{n+1}, B_{n+1}, \exp_{n+1})$.
Thus $$\smallk_{n+1}\ =\ \smallk_n[[\fM_{n+1}]], \qquad
\fM_{n+1}\ =\ \exp_{n+1}(A_{n}).$$ As in the previous section we now extend $\der_n$ uniquely to a derivation 
$\der_{n+1}$ on $\smallk_{n+1}$ by requiring that $\der_{n+1}$ is strong as an operator on the Hahn field $\smallk_{n+1}$ over 
$\smallk_n$ with 
$\der_{n+1}(\exp_{n+1}(a))=\der_n(a)\exp_{n+1}(a)$
for all $a\in A_n$. This extended derivation on $\smallk_{n+1}$
respects exponentiation and is the derivation on the Hahn field $\smallk_n[[\fM_{n+1}]]$
over $\smallk_n$ that corresponds to the derivation $\der_n$ on 
$\smallk_n$ and the additive function 
$$c_{n+1}\ :\ \fM_{n+1}\to \smallk_n, \qquad c_{n+1}\big(\exp_{n+1}(a)\big)\ :=\ \der_n(a)\quad (a\in A_n).$$
In this way we obtain field extensions 
\[\smallk\ \subseteq\ \smallk_0\ \subseteq\ \smallk_1\ \subseteq\ \cdots\subseteq\ \smallk_n\ \subseteq\ \smallk_{n+1}\ \subseteq\ \cdots.\]
We set $\T_{\exp}=\smallk_*:=\bigcup_n \smallk_n$, denote the derivation on 
$\T_{\exp}$ that extends all $\der_n$ by $\der$. The function
$\T_{\exp} \to \T_{\exp}$ that extends all $\exp_n$ is denoted by $\exp$, and we also write $\ex^f$ instead of $\exp(f)$ for $f\in \Texp$. 
With the appropriate monomial groups 
$\fM^{(n)}:=\fM_0 \cdots \fM_n$ satisfying 
\[\fM_0\ =\ \fM^{(0)}\ \subseteq\ \fM^{(1)}\ \subseteq\ \fM^{(2)}\ \subseteq \cdots \qquad \text{(as ordered abelian groups)}\]
and the monomial group $\fM=\bigcup_n \fM^{(n)}$ the above falls under the 
general set-up specified in Section~\ref{basicsetup}.  Note that the groups $\fM_n$ and $\fM^{(n)}$ are divisible. Hence $\fM$ is divisible, and so
$\R[[\fM]]$ and $\Texp=\smallk_{*}$ are real closed. The next lemma is basically
\cite[Lemma 2.2]{DMM}:

\begin{lemma}\label{AM}  The $A_n$ and $\fM_n$ are related as follows:  \begin{enumerate}
\item[\rm{(i)}] $A_n=\{f\in \smallk_n:\ \supp f \succ \fM^{(n-1)}\}$;
\item[\rm{(ii)}] $A_0+ \cdots + A_n=\{f\in \smallk_n:\ \supp f \succ 1\}$;
\item[\rm{(iii)}] $\fM^{(n)}=x^{\R}\cdot\exp(A_0+ \cdots + A_{n-1})$. 
\end{enumerate}
\end{lemma}

\noindent
We refer to \cite[Section 3]{DMM} for the fact that the derivation $\der$ on $\T_{\exp}$ has constant field $\smallk=\R$.
In that paper $\T_{\exp}$ is denoted by $\smallk((t))^{\text{E}}$; the derivation $\frac{d}{dx}$ used there is to be thought of as taking the derivative with respect to the ``independent variable'' $x=t^{-1}$, and we have $\der=x\cdot \frac{d}{dx}$ for our derivation $\der$. Taking this into account, the following is implicit in \cite[Lemma 3.5]{DMM}:

\begin{lemma}\label{transexp} For $\fm\in \fM_{n+1}\setminus \{1\}$ we have
$c(\fm)\in A_n^{\ne}$, so $c(\fm) \succ \fM^{(n-1)}$. 
\end{lemma} 
\begin{proof} Let $\fm\in \fM_{n+1}\setminus \{1\}$. Then 
$\fm=\exp(a)$ with
$a\in A_n^{\ne}$, and then $c(\fm)=\der(a)$.
If $n=0$, then $a=\sum_{r>0}a_rx^r$ with $r$ ranging over $\R$ and real coefficients $a_r$, and so $c(\fm)=\sum_{r>0}ra_rx^r\in A_0^{\ne}$. If $n>0$, then $c(\fm)=\der(a)\in A_n^{\ne}$ by \cite[Lemma 3.5(iii)]{DMM}. 
\end{proof}

\noindent   
We also recall some of the notations for certain subsets of $\Texp$ and its valued group that were used in \cite{ADH}, and relate it to the notation in this section: $E_n$ there corresponds to $\smallk_n$ here, $G_n$ there to ${\fM^{(n)}}$ here, 
the set $G^{\text{E}}=\bigcup_n G_n$ of exponential transmonomials there to $\fM=\bigcup_n {\fM^{(n)}}$ here, and $A_n$ and $B_n$ there have the same meaning as here. 
In the next chapter we focus on $\Texp$ and its extension $\T$ and there we do adopt these notations $G_n$ and $G^{\E}$.

\subsection*{Closure under exponentiation} Let $E$ be a subfield
of $\Texp$. We say that $E$ is {\bf closed under exponentiation\/} if $\exp(E)\subseteq E$. We define the {\bf closure of $E$ under exponentiation\/} 
to be the smallest subfield $F\supseteq E$ of $\Texp$ that is closed under exponentiation;
we denote this $F$ by $E^{\exp}$. Note that if $\der E\subseteq E$, then $\der E^{\exp}\subseteq E^{\exp}$. 

\begin{lemma}\label{exptrunc} If $E\supseteq \R$ is truncation closed, then so is $E^{\exp}$. 
\end{lemma} 
\begin{proof} Let $K\supseteq E\supseteq \R$ be a truncation closed subfield of $E^{\exp}$, and suppose
$K\ne E^{\exp}$. It is enough to show that then there exists a truncation closed subfield $L$ of $E^{\exp}$ that properly contains $K$. If $K$ is not closed under small exponentiation, then $L:=K^{\smallexp}$ works. 

Assume that $K$ is closed under small exponentiation. Let $a\in K$ with $\ex^a\notin K$ be such that $o(a)$ is minimal. 
We have $a=a|_{1} + r + \varepsilon$ with $r\in\R$ and $\varepsilon\in K^{\prec 1}$, so $\ex^a=\ex^{a|_1} \cdot \ex^r\cdot \exp(\varepsilon)$
with $\ex^r\in \R\subseteq K$ and $\exp(\varepsilon)\in K$.  Thus by the minimality of $o(f)$ we have $a=a|_{1}$, that is, $\supp a \subseteq \fM^{\succ 1}$. 
Then $\ex^a\in \fM$ by Lemma~\ref{AM}, and so $L:=K(\ex^a)$ is truncation closed.  
\end{proof}

\bigskip\noindent
As promised, we now give an example of a truncation closed subfield of $\T_{\exp}$ such that the differential subfield of $\T_{\exp}$ generated by it is not truncation closed: 

\subsection*{An example inside $\Texp$ where being truncation closed is not preserved.} Take two strictly decreasing sequences 
\[\alpha_0 > \alpha_ 1 >\alpha_2 > \cdots, \text{ and }\beta_0 > \beta_1 > 
\beta_2\cdots \] of real numbers such that $\alpha_n > \beta_0$ and $\beta_n > 0 $ for all $n$,
and the elements $g:=\sum_n \alpha_nx^{\alpha_n}$ and $h:=\sum_n \beta_n x^{\beta_n}$  of $\smallk_0=\R[[x^\R]]$ are $\text{d}$-algebraically independent. 
We now have the monomial 
\[ f\ =\ \exp\left(\sum_n x^{\alpha_n} + \sum_n x^{\beta_n}\right)\in \fM_1,\]
so the subfield $\R(f)$ of $\smallk_1\subseteq \T_{\exp}$ is truncation closed by (i) and (ii) of Proposition~\ref{D}, but $\R\langle f\rangle$, the differential field generated by $\R(f)$ in $\smallk_1$, is not truncation closed. 
(Note that $\R\langle f \rangle$ is also the differential ring
generated by $\R(f)$ in $\smallk_1$.) To see this, note that $g+h = f^{\dagger} \in \R\langle f \rangle$, so if $\R\langle f \rangle$ were truncation closed, then both $g$ and $h$ would be in $\R\langle f \rangle$, making the differential transcendence degree of $\R\langle f \rangle$ greater than $1$, a contradiction.

\section{Splitting}\label{neat}

\noindent
{\em In  this section $E\supseteq \smallk$ is a truncation closed subfield of $\smallk_{*}$}. (We are in the general setting of Section~\ref{basicsetup}.)
The example at the end of the previous section shows that for the differential subfield of $\smallk_{*}$ generated by $E$ to be truncation closed, we are forced to make extra assumptions on $E$.

Let $S\subseteq \fM$. We say that $S$ {\bf splits} if for all ${\fm_0\cdots \fm_n}\in S$
with $\fm_i\in \fM_i$ for $i=0,\dots,n$ we have
${\fm_i}\in S$ for $i=0,\dots,n$. Thus if $S\ne \emptyset$ splits, then $1\in S$. Splitting is a monomial analogue of {\em truncation closed}.  The ``extra assumptions'' involve splitting: 

\begin{lemma}\label{neatdif} Suppose $\fM_E:=E\cap \fM$ splits, and
$c(\fM_E)\subseteq E$. Then the differential
subring $R$ of $\smallk_{*}$ generated by $E$ is truncation closed,
and thus the differential subfield $\Q\<E\>$ of $\smallk_{*}$ generated by $E$ is truncation closed as well. 
\end{lemma}
\begin{proof} Recall that $\fM_E$ is a subgroup of $\fM$ and
$E\subseteq \smallk[[\fM_E]]$, and that $\smallk_{*}\cap \smallk[[\fM_E]]$ is a
differential subfield of $\smallk_{*}$. Thus $R\subseteq \smallk_{*}\cap \smallk[[\fM_E]]$. Let $A\supseteq E$ be any truncation closed subring of $R$ such that $A\ne R$. It is enough to show that then there is a truncation closed subring of $R$ that properly contains $A$. Take $n$ minimal such that for some $f\in A\cap\smallk_n$ we have $\der(f)\notin A$, and take such $f$ with $o(f)$
minimal. We claim that every proper truncation of $\der(f)$ lies in $A$. (As a consequence, $A[\der(f)]$ is a truncation closed subring of $R$ that properly contains $A$.) The case $n=0$ of the claim is clear. Let $n\ge 1$. Then
$f=\sum_{\fm\in \fM_n}f_{\fm}\fm$ with all $f_{\fm}\in \smallk_{n-1}$, and so 
\[g\ :=\ \der(f)\ =\ \sum_{\fm\in \fM_n}\big(\der(f_{\fm})+c_n(\fm)
f_{\fm}\big)\fm\ =\ \sum_{\fm\in \fM_n}g_{\fm}\fm,\quad \text{where }g_{\fm}\ :=\ \der(f_{\fm})+c_n(\fm)f_{\fm}.\]
Such a proper truncation equals
$g|_{\fp}$ where
$\fp\in \supp(g)$, so $\fp=\fq\fn$, $\fq\in \fM^{(n-1)}$, 
$\fn\in \fM_{n}$. In the rest of the proof $\fm$ ranges over $\fM_{n}$. Then
\[g|_{\fp}\ =\ 
\sum_{\fm\succ \fn}g_{\fm}\fm + (g_{\fn}|_{\fq})\fn.\]
We have $\der(\sum_{\fm\succ \fn}f_{\fm}\fm)=\sum_{\fm\succ \fn}g_{\fm}\fm$, and $\sum_{\fm\succ \fn}f_{\fm}\fm$ is a proper truncation of
$f$ and thus lies in $A$, with $o(\sum_{\fm\succ \fn}f_{\fm}\fm)< o(f)$, so 
$\sum_{\fm\succ \fn}g_{\fm}\fm\in A$. 
Now $\fp\in \fM_E$, so $\fn\in \fM_E$ since $\fM_E$ splits. As $A$ is truncation closed, we have $f_{\fn}\fn\in A$, and so $f_{\fn}\in A\cap\smallk_{n-1}$. The minimality of $n$ gives $\der(f_{\fn})\in A$,
so $g_{\fn}=\der(f_{\fn})+c(\fn)f_{\fn}\in A$, hence $(g_{\fn}|_{\fq})\fn\in A$, and thus $g|_{\fp}\in A$, as promised. 
\end{proof} 

\noindent
The proof of Lemma~\ref{neatdif} shows that if $\fM_E$ splits
and $c(\fM_E)\subseteq E$, then 
$\fM_E=\fM_R=\fM_{\Q\<E\>}$, so the assumption that $\fM_E$ splits and $c(\fM_E)\subseteq E$ is inherited by $R$ and $\Q\<E\>$ from $E$. The same is true
for the (truncation closed) henselization $E^{\text{h}}$ of the field $E$ in 
$\smallk[[\fM]]$, since 
$E^{\text{h}}\subseteq \smallk_{*}\cap \smallk[[\fM_E]]$, and thus
$\fM_{E^{\text{h}}}=\fM_E$. Note also that if
$\der E\subseteq E$, then $c(\fM_E)\subseteq E$.  

It will be useful to consider some more extension procedures and see
when splitting and the like are preserved.
Let $\fG$ be a subset of $\fM$. Then $E[\fG]$ and
$E(\fG)$ are truncation closed by (i) and (ii) of Proposition~\ref{D}. Let $\fN$ be the subgroup of $\fM$ generated by $\fG$ over $\fM_E$. Then $E(\fG)\subseteq \smallk[[\fN]]$, and
thus $\fM_{E(\fG)}=\fN$. Note that if $\fM_E$ and $\fG$ split, then so does
$\fN$, and if $c(\fM_E), c(\fG)\subseteq E(\fG)$, then $c(\fN)\subseteq E(\fG)$. Here is a consequence: 

\begin{lemma}\label{algtrunc} Assume $\fM_E$ splits. Let $F$ be the algebraic closure of 
$E$ in $\smallk[[\fM]]$. Then $F$ is truncation closed, $F\subseteq\smallk_{*}$ and $\fM_F$ splits. If $c(\fM_E)\subseteq E$, then
$c(\fM_F)\subseteq E$. If $\der E\subseteq E$, then $\der F\subseteq F$. 
\end{lemma}
\begin{proof} Let $\fG:=\{\fg\in \fM:\ \fg^n\in \fM_E \text{ for some }n\ge 1\}$. Then $\fG$ equals $\fN$ as defined in the remarks above, and
$\fG$ splits. By Lemma~\ref{actr} we have $F=E(\fG)^{\text{h}}$, so
$\fM_F=\fG$, and thus previous remarks apply to give the desired result.  
\end{proof} 

\noindent
With $\smallk_{*}=\Texp$ as in the previous section, $\smallk_{*}$ is real closed, and then $F$ as in the lemma above is the real closure in
$\smallk_{*}$ of its ordered subfield $E$.

\subsection*{Closure under exponentiation}  In this subsection $\smallk_{*}=\Texp$ as in Section~\ref{METEXP}. 

\begin{lemma}\label{expsplit} If $\fM_E$ splits, then 
$\fM_{E^{\exp}}$ splits.
\end{lemma} 
\begin{proof} Let $K\supseteq E$ be a truncation closed subfield of $E^{\exp}$ such that $\fM_K$ splits, and suppose
$K\ne E^{\exp}$. It is enough to show that then there exists a truncation closed subfield $L$ of $E^{\exp}$ that properly contains $K$ such that $\fM_L$ splits. If $K$ is not closed under small exponentiation, then $L:=K^{\smallexp}$ (with $\fM_L=\fM_K$) works. 

Assume that $K$ is closed under small exponentiation. Let $a\in K$ with $\ex^a\notin K$ be such that $o(a)$ is minimal. 
We have $a=a|_{1} + r + \varepsilon$ with $r\in\R$ and $\varepsilon\in K^{\prec 1}$, so $\ex^a=\ex^{a|_1} \cdot \ex^r\cdot \exp(\varepsilon)$
with $\ex^r\in \R\subseteq K$ and $\exp(\varepsilon)\in K$.  Thus by the minimality of $o(f)$ we have $a=a|_{1}$, that is, $\supp a \subseteq \fM^{\succ 1}$. 
Then $\ex^a\in \fM$ by Lemma~\ref{AM}, and so $L:=K(\ex^a)$ is truncation closed. It remains to show that $\fM_L$ splits.
With $\fm:=\ex^a$, it is enough to show that $\{1,\fm\}$ splits by earlier remarks.
Take $n$ with $a\in \smallk_n$. We can assume that $n$ here is minimal in the sense that there is no $b\in K\cap\smallk_{n-1}$ with $o(b) \le o(a)$ and
$\ex^b\notin K$.  Then 
$a=a_n + a_{n-1} + \cdots + a_0$ with $a_i\in A_i$ for $i=0,\dots,n$ and $a_n\ne 0$. Each partial sum $a_n + a_{n-1} +\cdots + a_i$ is a truncation of $a$ and so
lies in $K$, hence all $a_i\in K$.  If $a_i=0$ for all $i<n$, then $\fm\in \fM_{n+1}$, and so
$\{1,\fm\}$ splits and we are done. Suppose $a_i\ne 0$ for some $i<n$. Then $o(a_n)< o(a)$, so $\ex^{a_n}\in K$ and $b:=a_{n-1}+\cdots + a_0\in K\cap\smallk_{n-1}$
has the property that $o(b)\le o(a)$ and $\ex^b\notin K$, contradicting the minimality of $n$.
\end{proof}

\subsection*{More about splitting} The following generalities will also be useful in the next chapter. Let $R\supseteq \smallk$ be a truncation closed subring of $\smallk[[\fM]]$ such that $\fM_R$ splits. 
 
\begin{lemma}\label{split1} Let $V$ be a truncation closed $\smallk$-linear subspace of $\smallk[[\fM]]$ such that $\fM_V$ splits. Then for the $($truncation closed$)$ subring
$R[V]$ of $\smallk[[\fM]]$ we have that $\fM_{R[V]}$ splits.
\end{lemma}
\begin{proof} We have $\fM_{R[V]}=\supp(R[V])$. Now use that every element of $\supp (R[V])$ is of the form
$\fm\fn_1\cdots \fn_k$ with $\fm\in \supp(R)=\fM_R$ and $\fn_1,\dots, \fn_k\in \supp(V)=\fM_V$, $k\in \N$. 
\end{proof}

\begin{lemma}\label{split2} Let $F$ be the fraction field of $R$ in $\smallk[[\fM]]$. Then $\fM_F$ splits. 
\end{lemma} 
\begin{proof} Since $F$ is truncation closed, we have $\fM_F=\supp (F) $. Now use that every element of $\supp(F)$ is of the form
$\fm\fn^{-1}$ with $\fm, \fn\in \supp(R)=\fM_R$.  
\end{proof}

\section{Extending the Derivation on $\smallk_*$ to $\smallk[[\fM]]$}

\noindent
Recall from Section~\ref{basicsetup} that
$\der_n$ is a strong operator on the Hahn field $\smallk_n=\smallk[[\fM^{(n)}]]$ over $\smallk$. 
Recall also that $\der$ denotes the common extension of the $\der_n$ to a derivation of $\smallk_*$. 
Next we show how $\der$ extends naturally to $\smallk[[\fM]]$. The results of this section are not used later, but are included for their independent interest. 
They might also become useful when trying to extend the results in this chapter and the next to the completions of $\Texp$ and $\T$. 

\begin{lemma}\label{extder}
The derivation $\der$ of $\smallk_*$ extends uniquely to a strong operator 
on the Hahn field $\smallk[[\fM]]$ over $\smallk$. This operator is a derivation on $\smallk[[\fM]]$. 
\end{lemma}
\begin{proof}
We first show that the restriction of $\der$ to $\smallk\fM$ satisfies the conditions of Lemma \ref{sumExt}. Let $\fm$ range over $\fM$. It is clear that for $a,b\in \smallk$,
\[\der((a+b)\fm)\ =\ \der(a\fm) + \der(b\fm),\qquad
\supp \der(a\fm)\  \subseteq\ \big(\supp\der(\fm)\big)\cup\{\fm\}. \]
Let $\fG\subseteq \fM$ be well-based; we need to show that the family 
$(\der(\fg))_{\fg\in \fG}$ is summable. 
Towards a contradiction, assume $\bigcup_{\fg\in \fG}\supp \der(\fg)$ is not well-based. Then we have a sequence 
\[\fm_0\fg_0\ \prec\ \fm_1\fg_1\ \prec\ \fm_2\fg_2\ \prec\  \cdots\]
such that $\fg_i \in \fG$ and $\fm_i \in \supp c(\fg_i)$ for all $i$. Passing to a subsequence we arrange that
$(\fg_i)$ is constant or strictly decreasing. The first case gives 
$\fg\in \fG$ with $\fg_i=\fg$ for all $i$, and would give a strictly increasing sequence $(\fm_i)$
in $\supp c(\fg)$, which is impossible. So $(\fg_i)$ is strictly decreasing. Passing once again to a subsequence we
arrange that either for some $n$ all $\fg_i\in \fM^{(n)}$, or there are $n_0 < n_1 < n_2 < \cdots$ in $\N^{\ge 1}$ such that $\fg_i\in \fM^{(n_i)}\setminus \fM^{(n_i-1)}$ for all $i$. Consider first the case that $n$ is such that
$\fg_i\in \fM^{(n)}$ for all $i$,
and thus $\fm_i\in \fM^{(n)}$ for all
$i$. But $\big(\der_n(\fg_i)\big)_{i\in \N}$ is summable, with 
$\sum_i \der_n(\fg_i)=\sum_i c(\fg_i)\fg_i$, and the fact that the latter 
sum exists in $\smallk[[\fM^{(n)}]]$ 
contradicts $\fm_0\fg_0\prec \fm_1\fg_1\prec \cdots$. Next,
suppose $n_0 < n_1 < n_2 < \cdots$ in $\N^{\ge 1}$ are such
that $\fg_i\in \fM^{(n_i)}\setminus \fM^{(n_i-1)}$ for all $i$. Then $\fg_i=\fh_i\fn_i$ with $\fh_i\in \fM^{(n_i-1)}$
and $\fn_i\in \fM_{n_i}\setminus \{1\}$ for all $i$. It
follows that $(\fn_i)$ is strictly decreasing, and so
$(\fm_i\fg_i)$ would be too, a contradiction.

Now, let $\fm$ be given and suppose towards a contradiction that 
$\fm\in \supp \der \fg$ for infinitely many $\fg\in \fG$. Then 
$\fm\in \supp \der \fg_i$ for a sequence 
$\fg_0 \succ \fg_1\succ \fg_2 \succ \cdots$ in $\fG$. Passing to a subsequence we
arrange that either for some $n$ all $\fg_i\in \fM^{(n)}$, or there are $n_0 < n_1 < n_2 < \cdots$ in $\N^{\ge 1}$ such that $\fg_i\in \fM^{(n_i)}\setminus \fM^{(n_i-1)}$ for all $i$. If $n$ is such
that all $\fg_i\in \fM^{(n)}$, then $\fm\in \fM^{(n)}$,
contradicting that $\sum_i \der_n(\fg_i)$ exists in $\smallk[[\fM^{(n)}]]$. Next,
suppose $n_0 < n_1 < n_2 < \cdots$ in $\N^{\ge 1}$ are such
that $\fg_i\in \fM^{(n_i)}\setminus \fM^{(n_i-1)}$ for all $i$. Then $\fg_i=\fh_i\fn_i$ with $\fh_i\in \fM^{(n_i-1)}$
and $\fn_i\in \fM_{n_i}\setminus \{1\}$ for all $i$. Then
$(\fn_i)$ is strictly decreasing, and 
$\supp\der(\fg_i)\subseteq \fM^{(n_i-1)}\fn_i$ for all $i$. But $\fm\in \supp \der\fg_i$ for all $i$ contradicts that for $i\ne j$ the sets $\fM^{(n_i-1)}\fn_i$ and $\fM^{(n_j-1)}\fn_j$ are disjoint.  

Thus $(\der(\fg))_{\fg\in \fG}$ is indeed summable.  
We can now appeal to Lemma~\ref{sumExt} to conclude that the restriction of $\der$ to $\smallk\fM$ has a unique extension to a strong operator on 
$\smallk[[\fM]]$. It is clear that this extension extends the derivation 
$\der$ on $\smallk_*$ and is a derivation on $\smallk[[\fM]]$. 
\end{proof}

\noindent
We consider $\smallk[[\fM]]$ below as equipped with
the derivation defined in Lemma~\ref{extder} and denote it also by $\der$. By Lemma~\ref{small} we have $\der(\fm)\prec 1$ for all
$\fm\in \fM^{\prec 1}$, so this derivation on 
$\smallk[[\fM]]$ is small. The next result extends Corollary~\ref{samecons} from 
$\smallk{_*}$ to $\smallk[[\fM]]$:
\begin{lemma} The following conditions are equivalent: \begin{enumerate}
\item[\rm{(i)}] for every $n$ the map $c_n: \fM_n\to \smallk_{n-1}$ is injective and $c_n(\fM_n)\cap \smallk_{n-1}^\dagger=\{0\}$;
\item[\rm{(ii)}] $\smallk[[\fM]]$ has the same constant field as 
$\smallk$.
\end{enumerate}
\end{lemma}
\begin{proof} The direction (ii)$\Rightarrow$(i) is immediate from
Corollary~\ref{samecons}. For (i)$\Rightarrow$(ii), assume (i), and let $f\in \smallk[[\fM]]\setminus \smallk$; we show that then 
$\der(f)\notin \smallk$ (and thus $\der(f)\ne 0$). Subtracting 
an element of $\smallk$ from $f$ we arrange that $1\notin \supp(f)$.
Let $\fd:=\fd(f)\in \fM^{\ne 1}$ be the leading monomial of $f$, and take
$n$ minimal with $\fd\in \fM^{(n)}$. Then $\fd=\fp\fn$ with
$\fp\in \fM^{(n-1)}$ and $\fn\in \fM_n,\ \fn\ne 1$. We claim that
then $\fd(\der(f))=\fq\fn$ for some $\fq\in \fM^{(n-1)}$ (and thus
$\der(f)\notin \smallk$, as promised). 
We have $f=g+h$ where 
\[g\in \smallk_n,\quad \supp(g)=\supp(f)\cap \fM^{(n)},\quad h\in \smallk[[\fM]],\quad \supp(h)\cap\fM^{(n)}=\emptyset.\]  
Now $g=\sum_{\fm\in \fM_n}g_{\fm}\fm$ with all $g_{\fm}\in \smallk_{n-1}$, and so $g=g_{\fn}\fn + \sum_{\fm\in \fM_n^{\prec \fn}}g_{\fm}\fm$, and so
\[\der(g)\ =\ \big(\der(g_{\fn})+c_n(\fn)g_{\fn}\big)\fn + 
\sum_{\fm\in \fM_n^{\prec \fn}} \big(\der(g_{\fm})+c_n(\fm)g_{\fm}\big)\fm,\]
and thus $\fd\big(\der(g)\big)=\fq\fn$, where 
$\fq\in \fM^{(n-1)}$. Next, let $\fm\in \supp(h)$. Then
$\fm=\fm_0\cdots\fm_{n+k}$ with $k\ge 1$ and all $\fm_i\in \fM_i$,
and $\fm_{n+k}\ne 1$. Then $\fm_{n+k}\prec 1$, since
$\fm_{n+k}\succ 1$ would give $\fm\succ \fd$, contradicting that
$\fd$ is the leading monomial of $f$. It follows that
for the term $h_{\fm}\fm$ of $h$ (with $h_{\fm}\in \smallk$),
we have $\der(h_{\fm}\fm)=\big(\der(h_{\fm})+c(\fm)h_{\fm}\big)\fm\prec \fM^{(n)}$. Hence $\fd(\der(f))=\fq\fn$.
\end{proof}

\section{Integration in $\smallk_*$}

\noindent
In this section we assume that the initial derivation on $\smallk$ is trivial. (This is the case for $\T_{\exp}$ where $\smallk=\R$.) Then $1$ has no antiderivative in $\smallk[[\fM]]$. 
More generally:

\begin{lemma} Let $g\in \smallk[[\fM]]$. Then $1\notin \supp \der (g)$.
\end{lemma}
\begin{proof} With $g=\sum g_{\fm}\fm$ (where all $g_{\fm}\in \smallk$) we have
$\der(g)=\sum g_{\fm}c(\fm)\fm$. For $\fm\in \fM_0$ we have $c(\fm)\in \smallk$, with $c(1)=0$, and thus $1\notin \supp g_{\fm}c(\fm)\fm$.  Next, consider a monomial $\fm\in \fM\setminus \fM_0$. Then we have $n\ge 1$ with $\fm=\fm_0\cdots \fm_n$, $\fm_i\in \fM_i$ for $i=0,\dots,n$ and $\fm_n\ne 1$, 
hence $g_{\fm}c(\fm)\fm\in \smallk_{n-1}\fm_n$, in particular,
$1\notin \supp g_{\fm}c(\fm)\fm$.  
\end{proof}

\noindent
Thus if $f\in \smallk[[\fM]]$ and $1\in \supp f$, then $f\notin \der\smallk[[\fM]]$. Under a reasonable assumption on $\smallk_{*}$ (being transserial), $1\in \supp f$ is the only obstruction for $f\in \smallk_*$ to have an antiderivative in $\smallk_*$, by the next lemma. The term ``transserial'' is suggested by the fact that by Lemma~\ref{transexp} the field $\Texp$
of purely exponential transseries has the property we are now defining: we call $\smallk_{*}$ \textbf{transserial} if $c_0:\fM_0\to \smallk$ is injective and $c(\fm)\succ \fM^{(i-1)}$ for all $i\in \N$ and 
$\fm\in \fM_{i+1}\setminus\{1\}$. Note that if $\smallk_{*}$ is transserial, then $c: \fM \to \smallk_{*}$ is injective.

\begin{lemma}\label{cint} Assume $\smallk_{*}$ is transserial. Let $f\in \smallk_n$ be such that $1\notin \supp f$. Then $f\in \der\smallk_n$.
\end{lemma}

\begin{proof}
We proceed by induction on $n$. 
For $n=0$ we have $f = \sum_{\fm\in \fM_0^{\ne 1}} f_\fm \fm$ (with all $f_{\fm}\in \smallk$) and then $f=\der(g)$ for $g = \sum_{\fm\in \fM_0^{\ne 1}} \frac{f_\fm}{c(\fm)}\fm$. Assume the lemma holds for a certain $n$, and suppose
$f\in \smallk_{n+1}$ and $1\notin \supp f$. We have
$f = \sum_{\fm\in \fM_{n+1}}f_\fm \fm$ with all $f_\fm \in \smallk_{n}$, and $1\notin \supp f_{1}$. The inductive assumption gives $f_1=\der(g_1)$ with $g_1\in \smallk_{n}$. Next, let
$\fm\in \fM_{n+1}$, $\fm\ne 1$; we wish to find $g_{\fm}\in \smallk_n$ such that $\der(g_{\fm}) + g_{\fm}c(\fm)=f_{\fm}$,
that is, $\big(\der_n + c(\fm)I_n\big)(g_{\fm})=f_{\fm}$
(where $I_n$ is the identity operator on $\smallk_n$), equivalently,
$$\big(I_n+ c(\fm)^{-1}\der_n\big)(g_{\fm})\ =\ c(\fm)^{-1}f_{\fm}.$$ 
Using $c(\fm)\succ \fM^{(n-1)}$ we obtain from Section~\ref{adj} that the additive operator
$I_n+c(\fm)^{-1}\der_n$ on the Hahn field $\smallk_n=\smallk_{n-1}[[\fM_n]]$ over $\smallk_{n-1}$ is bijective, so we do have a unique solution $g_{\fm}$ as desired, and moreover $g_{\fm}=0$ whenever $f_{\fm}=0$. 
This yields the element $g:= g_1 + \sum_{\fm\in \fM_{n+1}^{\neq 1}} g_{\fm}\fm$ of $\smallk_{n+1}$ satisfying $f=\der(g)$.
\end{proof}

\begin{corollary} If $\smallk_{*}$ is transserial, then $f\in \der\smallk_{*}$ for all $f\in \smallk_*$  with $1\notin\supp(f)$. 
\end{corollary}

\noindent
In the proof of Lemma~\ref{cint} we saw how inverting operators
$I_n-a\der_n$ for $a\in \smallk_n$ with $a\prec \fM^{(n-1)}$
plays a role in integrating elements $f\in \smallk_n$.
This role will become more prominent in connection with truncation
in the next section.

\begin{lemma}\label{transconstant} Suppose $\smallk_{*}$ is transserial. Then $C_{\smallk_{*}}=\smallk$.
\end{lemma}
\begin{proof} We verify condition (i) of Corollary~\ref{samecons}.
Thus we have to show that for every $n$ the map $c_n: \fM_n\to \smallk_{n-1}$
is injective and $c_n(\fM_n)\cap \smallk_{n-1}^\dagger=\{0\}$. Since $\smallk^\dagger=\{0\}$, this holds for $n=0$. Let $n\ge 1$. We already noted that transseriality gives injectivity of $c_n$.
Let $\fm\in \fM_n\setminus\{1\}$. Then $c_n(\fm)\succ \fM^{(n-2)}$
by transseriality,
while for every $a\in\smallk_{n-1}^\dagger$ we have $a\preceq b$ for some $b\in\fM^{(n-2)}$ by Lemma~\ref{daggerprec}. 
This yields $c_n(\fM_n)\cap \smallk_{n-1}^\dagger=\{0\}$. 
\end{proof} 

\noindent
Thus for transserial $\smallk_{*}$ and given $f\in \smallk_{*}$, any two elements $g\in \smallk_{*}$ with $f=\der(g)$ differ by an element of $\smallk$. This is what we use tacitly later. 

\begin{lemma} Suppose $\smallk_{*}$ is transserial. Then $\smallk_{*}$ is differential-valued.
\end{lemma}
\begin{proof}  We show by induction on $n$ that $\smallk_n$ is $\d$-valued.
The case $n=0$ is taken care of by Lemmas~\ref{cgood} and ~\ref{transconstant}.  Assume $\smallk_n$ is $\d$-valued. To show that  $\smallk_{n+1}$ is $\d$-valued, let
$f,g\in \smallk_{n+1}$, $0\ne f,g\prec 1$; 
it suffices to show that then $f^\dagger \succ \der(g)$. 
We have $f=f_{\fm}\fm(1+\varepsilon)$ with $f_{\fm}\in \smallk_n^\times$, $\fm\in \fM_{n+1}$, and $\varepsilon\prec \smallk_n$, so by Lemma~\ref{cgood} applied to
$\smallk_n$ in the role of $\smallk$,
\[f^\dagger\ =\  f_{\fm}^\dagger + c_{n+1}(\fm) + \frac{\der(\epsilon)}{1+\varepsilon}, \qquad \frac{\der(\epsilon)}{1+\varepsilon}\prec \smallk_n.\]
If $\fm\ne 1$, then $f^\dagger\sim c_{n+1}(\fm)\succ \smallk_{n-1}$
by Lemma~\ref{daggerprec}, giving $f^\dagger \succ 1\succ \der(g)$.  So assume $\fm=1$. Then  $f\sim f_{\fm}\prec 1$, so $f^\dagger\sim f_{\fm}^\dagger\in \smallk_n^\times$. We have $g= h+\delta$ with $h\in \smallk_n^{\prec 1}$
and $\delta\prec \smallk_n$, so $\der(g)=\der(h) + \der(\delta)$. Now $f^\dagger\sim f_{\fm}^\dagger \succ \der(h)$ by the inductive
assumption. Also $\der(\delta)\prec\smallk_n$ by Lemma~\ref{cgood}
applied to $\smallk_n$ in the role of $\smallk$, so $f^\dagger \succ \der(\delta)$, and thus $f^\dagger \succ \der(g)$. 
\end{proof}

\section{Truncation and Integration in $\smallk_{*}$}

\noindent
We continue to assume in this section that the derivation is trivial on $\smallk$. 

\subsection*{Using operators $(I-a\der)\inv$ on $\smallk_*$}

\noindent
\begin{lemma} Let $F$ be a truncation closed
differential subfield of $\smallk_{*}$ that contains
$\smallk$ and $\fM$. Let $F_{\infty}$ be the smallest differential subfield
of $\smallk_{*}$ that contains $F$ such that for every $n$,
$F_n:=F_{\infty}\cap \smallk_n$ is closed under
$(I_n-a\der_n)\inv$ for all $a\in F_n$ with $a\prec\fM^{(n-1)}$, where by convention 
$\fM^{(-1)}:=\{1\}$. 
Then $F_{\infty}$ is truncation closed.
\end{lemma}
\begin{proof} First note that $c(\fM_n)\subseteq F\cap \smallk_{n-1}$: 
this is because for $\fm\in \fM_n$ we have $\fm\in F$, so
$c(\fm)=\fm^\dagger\in F$, hence $c(\fm)\in F\cap \smallk_{n-1}$.
We define differential subfields
$K_n\subseteq \smallk_n$ by recursion on $n$ as follows;
$K_n$ is the smallest differential subfield of $\smallk_n$ such that
\begin{itemize}
 \item $K_{n-1}\subseteq K_n$,
 \item $F\cap \smallk_n\subseteq K_n$, and
 \item $K_n$ is closed under $(I_n-a\der_n)\inv$ for all $a\in K_n$ with $a\prec\fM^{(n-1)}$,
\end{itemize}
where by convention $K_{-1}=\smallk_{-1}=\smallk$. 
We show by induction on $n$: 
\begin{enumerate}
 \item $K_{n-1}({\fM_n}) \subseteq K_n \subseteq K_{n-1}[[{\fM_n}]]$,  and 
 \item $K_n$ is truncation closed. 
\end{enumerate}
The first inclusion holds because $K_{n-1}\subseteq K_n$ and ${\fM_n}\subseteq F\cap \smallk_n\subseteq K_n$.
Also $c(\fM_n)\subseteq K_{n-1}$, so  $K_{n-1}[[{\fM_n}]]$ is a differential subfield of $\smallk_n$ closed under $(I_n-a\der_n)\inv$ for all $a\in K_{n-1}[[{\fM_n}]]$ with $a\prec\fM^{(n-1)}$.
Moreover, $F\cap \smallk_n$ is truncation closed, so $F\cap\smallk_n$ is $\smallk_{n-1}$-truncation closed by Lemma~\ref{trtr1}, hence 
\[F\cap \smallk_n\ \subseteq\ (F\cap \smallk_{n-1})[[{\fM_n}]]\ \subseteq\  K_{n-1}[[{\fM_n}]],\]
and thus $K_n\subseteq K_{n-1}[[{\fM_n}]]$ and $F\cap \smallk_n$
is $K_{n-1}$-truncation closed. 
The subfield $E_n$ of $K_{n-1}[[{\fM_n}]]$ generated by $K_{n-1}$ and $F\cap \smallk_n$ is a differential subfield of  $K_{n-1}[[{\fM_n}]]$, and $E_n$ is $K_{n-1}$-truncation closed
by Proposition~\ref{D}(i). Applying Theorem~\ref{thmA} to $E_n$ in the role of $E$ we conclude that $K_n$ is $K_{n-1}$-truncation closed, and so $K_n$ is $\smallk_{n-1}$-truncation closed.  
In view of $K_{n-1}\subseteq K_n$ and Lemma~\ref{trtr2} (applied to
$\smallk, K_{n-1}$, $\smallk_{n-1}$, $K_n$ in the roles of $\smallk_0, \smallk_1$, $\smallk$, $V$ in that lemma)
 it follows that $K_n$ is truncation closed. It is also clear by induction on $n$ that $K_n\subseteq F_n$.
Hence $K_{\infty}:= \bigcup_n K_n$ is a differential subfield of $F_{\infty}$ that contains $F$.
Moreover, $K_{\infty}\cap \smallk_n=K_n$ is closed under $(I_n-a\der_n)\inv$ for all $a\in K_n$ with $a\prec\fM^{(n-1)}$, so $K_{\infty}=F_{\infty}$ by the minimality of $F_{\infty}$.
Since $K_{\infty}$ is truncation closed, so is $F_{\infty}$.
\end{proof} 

\noindent
The assumption that ${\fM}\subseteq F$ is too strong for our purpose, but
we replace it in the next lemma by a splitting assumption. Recall that splitting was introduced
in Section~\ref{neat}. 


\begin{lemma}\label{opClosure} Let $F$ be a truncation closed
differential subfield of $\smallk_{*}$ that contains
$\smallk$ and such that ${\fM}_F$ splits. 
Let $F_{\infty}$ be the smallest differential subfield
of $\smallk_{*}$ containing $F$ such that for every $n$,
$F_n:=F_{\infty}\cap \smallk_n$ is closed under
$(I_n-a\der_n)\inv$ for all $a\in F_n$ with $a\prec\fM^{(n-1)}$. 
Then $F_{\infty}$ is truncation closed and 
${\fM}_{F_{\infty}}={\fM}_F$.
\end{lemma}
\begin{proof} Recall that $\fM_{F,n}:= (\supp F) \cap \fM_n=F\cap \fM_n$, a subgroup of
$\fM_n$, and $$\fM_F\ :=\ F\cap \fM\ =\ (\supp F)\cap \fM,$$ a subgroup of $\fM$. Setting
$\fM_F^{(n)}:=\fM_{F,0}\cdots \fM_{F,n}$, conditions (1) and (2) in Section~\ref{basicsetup} hold for $\fM_F$, $(\fM_{F,n})$, $(\fM_F^{(n)})$ in place of
$\fM$, $(\fM_n)$, $(\fM^{(n)})$; note also that $\fM^{(n)}\cap F=\fM_F^{(n)}$.  
We set 
\[\smallk_{F,n}\ :=\ \smallk[[{\fM_F^{(n)}}]],\]
s subfield of $\smallk_n= \smallk[[{\fM^{(n)}}]]$, with $F\cap \smallk_n\subseteq \smallk_{F,n}$. We also set $\smallk_{F,-1}:=\smallk$. 
Using that $F$ is a differential subfield of $\smallk_{*}$ we get $c_n(\fM_{F,n})\subseteq F\cap \smallk_{n-1}\subseteq \smallk_{F,n-1}$, and so $\smallk_{F,n}$ is a differential subfield of $\smallk_n$. 
The increasing chain 
\[\smallk\ \subseteq\ \smallk_{F,0}\ \subseteq\ \smallk_{F,1}\ \subseteq\ \cdots\subseteq\ \smallk_{F,n}\ \subseteq\ \smallk_{F,n+1}\ \subseteq \cdots\]
yields a differential subfield $\smallk_{F,*}:=\bigcup_n \smallk_{F,n}$ of $\smallk_{*}$ with $F\subseteq \smallk_{F,*}$. 
It remains to apply the previous lemma with
$\smallk_{F,*}$ instead of $\smallk_{*}$.  
\end{proof} 


\noindent
Let $F$ be as in Lemma~\ref{opClosure}. With the notations
in the proof of this lemma we record for later use: \begin{enumerate}
\item for $a\in \smallk_{F,n}$ we have: $a\prec\fM^{(n-1)} \Leftrightarrow a\prec\fM_F^{(n-1)}$;
\item if $a\prec\fM_F^{(n-1)}$, then $\smallk_{F,n}$ is closed under $(I_n-a\der_n)\inv$; 
\item if $\smallk_{*}$ is transserial, then so is $\smallk_{F,*}$, and $f\in \der\smallk_{F,n}$ for all $f\in \smallk_{F,n}$ with $1\notin \supp(f)$.
\end{enumerate}
Items (1) and (2) were tacitly used in the proof of Lemma~\ref{opClosure},
and (3) is a consequence (really a special case) of Lemma~\ref{cint}.

\begin{theorem}\label{thmB} Assume $\smallk_{*}$ is transserial. Let $F$ be a truncation closed differential subfield of $\smallk_{*}$ that contains $\smallk$ and such that ${\fM}_F$ splits. 
Let $F_{\infty}$ be the smallest differential subfield
of $\smallk_{*}$ that contains $F$ such that for every $n$,
$F_n:=F_{\infty}\cap \smallk_n$ is closed under
$(I_n-a\der_n)\inv$ for all $a\in F_n$ with $a\prec\fM^{(n-1)}$, and $f\in \der F_{\infty}$ for all $f\in F_\infty$ with $1\notin\supp(f)$. 

Then $F_{\infty}$ is truncation closed and ${\fM}_{F_{\infty}}={\fM}_F$.
\end{theorem}

\begin{proof} By the remarks preceding this lemma, $\smallk_{F,*}$ has the closure properties required for $F_{\infty}$, so $F_{\infty}\subseteq \smallk_{F,*}$ by the minimality of $F_{\infty}$, and thus $\fM_{F_{\infty}}=\fM_F$.

Let any truncation closed differential subfield $K\supseteq F$
of $F_{\infty}$ be given, and suppose $K\ne F_{\infty}$. It
suffices to show that then there is a
truncation closed differential subfield $L$ of $F_{\infty}$
that strictly contains $K$. If $K_n:=K\cap\smallk_n$ is not closed under $(I_n-a\der_n)\inv$ for some $n$ and $a\in K_n$ with $a\prec\fM^{(n-1)}$, then Lemma~\ref{opClosure} applied to $K$ in the role of $F$ yields such an extension $L$. So we can assume
that for every 
$n$, $K_n:=K\cap\smallk_n$ is closed under $(I_n-a\der_n)\inv$ for all $a\in K_n$ with $a\prec\fM^{(n-1)}$. 
Since $K\ne F_{\infty}$, we have a nonempty set 
\[ S\ :=\ \{f\in K:\  f\notin \der K,\ 1\notin \supp(f)\}.\]
It is enough to find $f\in S$ and $g\in \smallk_{*}$ such that $f=\der(g)$ (so $g\in F_{\infty}$) and every proper truncation of $g$ lies in $K$ (so $K(g)$ is truncation closed by (i) and (ii) of Proposition~\ref{D}). 

Take $f\in S$ such that $o(f)=\min o(S)$. 
Take $n$ minimal such that $f\in K_n$. 
Then 
\[f\ =\ \sum_{\fm\in \fM_{F,n}}f_{\fm}{\fm}\qquad\text{(all $f_{\fm}\in K_{n-1}$, with $K_{-1}:=\smallk$)}.\]  
Consider first the case $n=0$. Then $f=\der(g)$ with
$g:=\sum_{\fm\in \supp f}\frac{f_\fm}{c(\fm)}\fm$. Every proper truncation of $g$ equals $\der(\tilde{f})$ 
for some proper truncation $\tilde{f}$ of $f$, and thus lies in $K$. 

Next, let $n\ge 1$. Then $1\notin \supp f_{1}$. Since $f_1\in K_{n-1}$, the minimality of $n$ gives $f_1=\der(g_1)$ with $g_1\in K_{n-1}$, using also Lemmas~\ref{cint}
and ~\ref{transconstant}. Now, let
$\fm\in \fM_{F,n}\setminus \{1\}$. As we saw in the proof of Lemma~\ref{cint} (with $n+1$ in the role of the present $n$) we have $g_{\fm}\in \smallk_{n-1}$ such that $\der(g_{\fm}) + g_{\fm}c(\fm)=f_{\fm}$,
that is, 
\[\big(I_{n-1}+ c(\fm)^{-1}\der_{n-1}\big)(g_{\fm})\  =\ c(\fm)^{-1}f_{\fm}.\]
Now $c(\fm)\in  K_{n-1}$, so $K_{n-1}$ is closed under $\big(I_{n-1}+ c(\fm)^{-1}\der_{n-1}\big)\inv$, 
and thus  $g_{\fm}\in K_{n-1}$.  Moreover, if $f_{\fm}=0$, then $g_{\fm}=0$. 
This yields the element 
\[g\ :=\  g_1 + \sum_{\fm\in \fM_{F,n}^{\neq 1}} g_{\fm}\fm\ =\ \sum_{\fm\in \fM_{F,n}}g_{\fm}\fm\]
of $\smallk_{n}$ satisfying $f=\der(g)$. It only remains to show that all proper truncations of $g$
are in $K$. Such a proper truncation equals
$g|_{\fp}$ where
$\fp\in \supp(g)$, so $\fp=\fn\fq$, $\fn\in \fM_F^{(n-1)}$, 
$\fq\in \fM_{F,n}$. In the rest of the proof $\fm$ ranges over $\fM_{F,n}$. Then
\[g|_{\fp}\ =\ 
\sum_{\fm\succ \fq}g_{\fm}\fm + (g_{\fq}|_{\fn})\fq.\]
We have $\der(\sum_{\fm\succ \fq}g_{\fm}\fm)=\sum_{\fm\succ \fq}f_{\fm}\fm$, and here the right hand side is a proper truncation of
$f$ and thus lies in $K$, with $o(\sum_{\fm\succ \fq}f_{\fm}\fm)< o(f)$, so $\sum_{\fm\succ \fq}f_{\fm}\fm\notin S$ and thus
$\sum_{\fm\succ \fq}g_{\fm}\fm\in K$. Since $g_{\fq}\in K$ and $\fq\in F\subseteq K$ we also have $(g_{\fq}|_{\fn})\fq\in K$,
and thus $g|_{\fp}\in K$, as promised.
\end{proof}

\noindent
Let $\smallk_*$ and $F$ be as in the theorem and let $\fm\in \fM_F$. Then for the derivation $\derdelta:= \fm\der$ on $\smallk_{*}$ we have the
following: $f\in \derdelta F_{\infty}$ for all $f\in F_{\infty}$ with $\fm\notin \supp(f)$.  Theorem~\ref{thmB}, including the remark we just made, will be essential
in proving the main result of this dissertation in the next chapter. 





\chapter{Truncation in $\T$}

\section{Introduction}\label{intro}

\noindent
In this chapter we recollect from \cite{ADH} how to construct the field $\T$ of logarithmic-exponential transseries from $\Texp$, and extend the notion of splitting from sets of exponential transmonomials to sets of logarithmic-exponential transmonomials in order to prove our main result. 

As to $\T$, this is basically $\Texp$ extended with logarithms. For a detailed description, see
Section~\ref{consT} below. For now, we just mention the following basic facts about $\T$ in order to
 state the main result of this dissertation in a way that can be easily grasped. \begin{enumerate}
\item $\T$ is a truncation closed subfield of a Hahn field $\R[[G^{\LE}]]$, with $\R(G^{\LE})\subseteq \T$. 
Here $G^{\LE}$ is a monomial group whose elements are the \textbf{logarithmic-exponential transmonomials}. 
It contains the group $G^{\E}$ of exponential transmonomials as a subgroup, and the subfield $\Texp$ of the Hahn field $\R[[G^{\E}]]\subseteq \R[[G^{\LE}]]$, defined as in Section~\ref{METEXP},
 is a subfield of $\T$. 
\item We regard $\T$ as an {\em ordered\/} subfield of the ordered Hahn field $\R[[G^{\LE}]]$. The map $f\mapsto \ex^f$ on $\Texp$ extends naturally to an isomorphism, called \textbf{exponentiation}, of the ordered additive group of $\T$ onto its multiplicative group $\T^{>}$ of positive elements; we denote it also by $f\mapsto \ex^f$.
\item There is given a natural automorphism $f\mapsto f{\uparrow}$ of the ordered field $\T$ over $\R$, to be thought of as $f(x) \mapsto f(\ex^x)$. Its $n$th iterate is denoted by $f\mapsto f{\uparrow}^n$. 
Its inverse is $f\mapsto f{\downarrow}$, to be thought of as $f(x) \mapsto f(\log x)$, and the $n$th iterate of this inverse is denoted by $f\mapsto f{\downarrow}_n$.
\item With $G^{\E,n}:= G^{\E}{\downarrow}_n$ we  have 
an increasing sequence $G^{\E}\subseteq G^{\E,1}\subseteq G^{\E,2}\subseteq \cdots \subseteq G^{\LE}$, and
\[G^{\LE}\ =\ \bigcup_n G^{\E,n}.\]
In particular, $\uparrow$ and $\downarrow$ restrict to automorphisms of the (ordered) monomial group $G^{\LE}$. 
\item The derivation $\der$ on $\Texp$ extends naturally to a derivation on $\T$, also denoted by $\der$, such that $\der(\ex^f)=\der(f)\ex^f$ for all $f\in \T$. As an ordered field equipped with this extended derivation, $\T$ is a Liouville closed $H$-field with constant field $\R$. 
\end{enumerate}

\medskip\noindent
Let $K$ be a subfield of $\T$. In view of (2) we can define $K$ to be {\bf closed under exponentiation\/} if $\ex^K\subseteq K$, and we define the {\bf closure of $K$ under exponentiation\/} to be the smallest subfield
$L\supseteq K$ of $\T$ that is closed under exponentiation; we denote this $L$ by $K^{\exp}$. (For subfields of $\Texp$ this agrees with the previous chapter.) Note that if $\der(K)\subseteq K$, then $\der(K^{\exp})\subseteq K^{\exp}$. Call a set $S\subseteq \T$ {\bf truncation closed\/} if it is truncation closed as a subset of $\R[[G^{\LE}]]$. We can now state one of our results; it is proved in the next section and does not involve the derivation:

\begin{proposition}\label{trexp} If the subfield $K\supseteq \R$ of $\T$ is truncation closed, then so is $K^{\exp}$.
\end{proposition} 

\medskip\noindent
In the next section we also show that the property for sets $S\subseteq G^{\E}$ to split can be extended uniquely to sets $S\subseteq G^{\LE}$ in such a way that the following conditions are satisfied: \begin{enumerate}
\item[{(Sp1)}] for $S\subseteq G^{\E}$ to split as defined earlier agrees with $S$ to split as a subset of $G^{\LE}$; 
\item[{(Sp2)}] for $S\subseteq G^{\LE}$, $S$ splits iff $S\cap G^{\E,n}$ splits for all $n$;
\item[{(Sp3)}] for $S\subseteq G^{\LE}$, $S$ splits iff $S{\uparrow}$ splits. 
\end{enumerate}

\noindent
In the next section we also prove the following variant of Lemma~\ref{expsplit}: 

\begin{proposition}\label{trexpsplit} Suppose $K\supseteq \R$ is a truncation closed subfield of $\T$ such that $K\cap G^{\LE}$ splits.  Then $K^{\exp}\cap G^{\LE}$ splits. 
\end{proposition} 

\noindent
Any differential subfield of $\T$ containing $\R$ is an $H$-field, where the subfield is given the ordering induced by $\T$. Let $K\supseteq \R$ be a differential subfield of $\T$. We define $K^{\Li}$ to be the smallest real closed subfield of $\T$ containing $K$ that is closed under exponentiation and integration, where the latter means that for all $f\in K$ there exists $g\in K$ with $f=\der(g)$. Then $K^{\Li}$ is also a differential subfield of $\T$, and is Liouville closed as an $H$-field. In fact, $K^{\Li}$ is the smallest differential subfield of $\T$ that contains $K$ and is Liouville closed; we call $K^{\Li}$ the {\bf Liouville closure of $K$ in $\T$}. 
 
 We can now state our main result, proved in the last section: 

\begin{theorem}\label{thC} \label{THC} Let $K\supseteq \R$ be a truncation closed differential subfield of $\T$ such that $K\cap G^{\LE}$ splits.
Then $K^{\Li}$ is truncation closed and $K^{\Li}\cap G^{\LE}$ splits. 
\end{theorem} 

\noindent
For example, $K:=\R(x)$ is a truncation closed differential subfield of $\T$ with $K\cap G^{\LE}=x^{\Z}$. Thus
the Liouville closure $\R(x)^{\Li}$ of $\R(x)$ in $\T$ is truncation closed, and $\R(x)^{\Li}\cap G^{\LE}$ splits. 

\subsection*{Changing the derivation}
The derivation $\der$ that we use here is not the derivation $\frac{d}{dx}$ on $\T$ that is used systematically in \cite{ADH}, but it is closely related: $\der=x\cdot\frac{d}{dx}$ and we explain here why it doesn't matter which of these derivations we use. 
First, $\T$ with the derivation $\der$ is isomorphic via $\uparrow$ to $\T$ with the derivation $\frac{d}{dx}$, since for $f\in \T$ we have a chain rule for $f{\uparrow}$ 
thought of as $f(\ex^x)$:
 \[ \frac{df{\uparrow}}{dx}\ =\ \ex^x\cdot\big(\frac{df}{dx}\big){\uparrow}\ =\ \big(x\frac{df}{dx}\big){\uparrow}\ =\ \der(f){\uparrow}.\]
Next, $\uparrow$ preserves truncation: for $f,g\in \T$,
$$f\truncof g\ \Longleftrightarrow\ f{\uparrow}\truncof g{\uparrow}. $$
See the next section for these facts. Let $(\T,\frac{d}{dx})$ refer to
the $H$-field $\T$ with derivation $\frac{d}{dx}$ instead of
$\der$. Thus $(\T, \frac{d}{dx})$ is also Liouville closed. Let $K$ be a subfield of $\T$. If $x\in K$, then $K$ is a differential subfield of $\T$ iff $K$ is a differential subfield of $(\T, \frac{d}{dx})$. Note that if $K$ is a Liouville closed subfield of
$\T$, then we do have $x\in K$, and so $K$ is also a Liouville
closed subfield of $(\T, \frac{d}{dx})$; likewise with the roles of $\T$ and $(\T, \frac{d}{dx})$ interchanged. For differential subfields of $(\T,\frac{d}{dx})$ we define its Liouville closure in
$(\T,\frac{d}{dx})$ in the same way as for $\T$ (with $\frac{d}{dx}$ instead of $\der$). It should be clear now that Theorem \ref{thC} goes through with $(\T,\frac{d}{dx})$ instead of
$\T$, and with $K^{\Li}$ replaced by the smallest Liouville closed
subfield of $(\T,\frac{d}{dx})$ containing $K$.

\section{Constructing $\T$ and Splitting}\label{consT} 

\noindent
The introduction to this chapter described basic facts about $\T$, but not how $\T$ is obtained from $\Texp$. In this section we take care of that. Most of it
is taken from \cite[Appendix A]{ADH} to which we refer for proofs omitted here.  In this section we also extend the notion of splitting to subsets of $G^{\LE}$
and obtain Propositions~\ref{trexp} and ~\ref{trexpsplit} on the exponential closure of subfields of $\T$. 

We start by using the alternative notation $\R[[x^\R]]^{\E}$ for $\Texp$ to indicate $\R[[x^{\R}]]$ as
the first step in its construction and to suggest the presence of the formal variable $x$.
The idea is to 
replace $x$ here by $\log x$, $\log \log x$, and so on. Formally, we introduce distinct symbols $\ell_0, \ell_1,\ell_2,\ldots$, with $\ell_0$ standing for $x$, and 
$\ell_1$, $\ell_2, \dots$ to be interpreted later as $\log x$, $ \log \log x,\cdots$ so that
 for each $f(x) \in  \R[[x^\R]]^{\E}$ we have a corresponding $f(\ell_n)\in \R[[\ell_n^\R]]^{\E}$. More precisely, 
we take for each $n$  a copy $G^{\E,n}$ of the ordered abelian group $G^{\E}$ with an isomorphism
\[\fm\mapsto \fm\downer_n\ :\  G^{\E}\rightarrow G^{\E,n}\]
of ordered abelian groups, where $x^r\downer_n = \ell_n^r \in G^{\E,n}$ for $r\in \R$; by convention, $x:= x^1$ and $\ell_n:= \ell_n^1$.  For $n=0$ we take $G^{\E,0}=G^{\E}$ and let ${\downarrow}_0$ be the identity map.
 The above map $\fm \mapsto \fm{\downarrow}_n$ extends uniquely to a strongly additive $\R$-linear bijection
\[f\mapsto f\downer_n\ :\  \R[[G^{\E}]]\rightarrow \R[[G^{\E,n}]],\]
between Hahn spaces over $\R$. This map $f\mapsto f\downer_n$ is an isomorphism of ordered Hahn fields over $\R$. 
We denote the image of $\R[[x^\R]]^{\E}\subseteq \R[[G^{\E}]]$ under this map by $\R[[\ell_n^\R]]^{\E}$ and make the latter into an {\em exponential\/} ordered field so that 
\[f\mapsto f\downer_n\ :\  \R[[x^\R]]^{\E}\rightarrow \R[[\ell_n^\R]]^{\E}\]
is an isomorphism of exponential ordered fields. We denote the exponentiation on $\R[[\ell_n^\R]]^{\E}$ by $\exp$ as well. 
For $n=0$ we have of course $f\downer_0 = f$ for $f\in \R[[G^{\E}]]$.
Recall that $G^{\E}$ is the union of an increasing sequence $(G_m)_m$ of convex subgroups. Thus we have an increasing sequence $(G_m\downer_n)_m$ of convex subgroups  of $G^{\E,n}$ with $\bigcup_m G_m\downer_n = G^{\E,n}$ and  $\R[[\ell_n^\R]]^{\E}= \bigcup_m \R[[G_m\downer_n]]$.  

So far these are just notational conventions, but now a fact that goes beyond mere notation: for each $n$ there is a unique $\R$-linear embedding
$$ \R[[\ell_n^\R]]^{\E}\to \R[[\ell_{n+1}^\R]]^{\E}$$
of exponential ordered fields with the following properties: \begin{itemize}
\item it sends $\ell_n^r$ to $\exp(r\ell_{n+1})$ for all $r\in \R$;
\item for every $m$ it maps $G_m\downer_n$ into $G_{m+1}\downer_{n+1}$ and $\R[[G_m\downer_n]]$ into $\R[[G_{m+1}\downer_{n+1}]]$;
\item for every $m$ its restriction to a map $\R[[G_m\downer_n]]\to \R[[G_{m+1}\downer_{n+1}]]$ between Hahn spaces over $\R$ is strongly additive.
\end{itemize}
Identifying $\R[[\ell_n^\R]]^{\E}$ with its image in  $\R[[\ell_{n+1}^\R]]^{\E}$ under this map we have $\ell_n^r=\exp(r\ell_{n+1})$ for $r\in \R$ and $G_m\downer_n\subseteq G_{m+1}\downer_{n+1}$ for all $n$, and so we obtain a chain 
\[G^{\E}\ =\ G^{\E,0}\ \subseteq\ G^{\E,1}\ \subseteq\ G^{\E,2}\ \subseteq\  \cdots\]
of ordered groups and a corresponding chain
\[\Texp\ =\ \R[[\ell_0^\R]]^{\E}\ \subseteq\ \R[[\ell_1^\R]]^{\E}\ \subseteq\  \R[[\ell_2^\R]]^{\E}\ \subseteq\ \cdots,\]
of exponential ordered fields. 
 We set 
\[ G^{\LE}\ := \bigcup_n G^{\E,n},\qquad \T\  := \ \bigcup_n \R[[\ell_n^\R]]^{\E},\]
making $G^{\LE}$ into an ordered group extension of all $G^{\E,n}$, and $\T$ into an exponential ordered field extension of all $\R[[\ell_n^\R]]^{\E}$.
It also makes $\T$ a truncation closed subfield of $\R[[G^{\LE}]]$.
We refer to the elements of $\T$ as \textbf{transseries}, or $\LE$-series, and to the elements of $G^{\LE}$ as \textbf{transmonomials}, or $\LE$-monomials. 
Note that by virtue of this construction $\T$ and $\R[[G^{\LE}]]$ are real closed.

\subsection*{Exponentiation in $\T$}
We obtained $\T$ as an ordered exponential field, and denote its exponential map by $\exp$. For $f\in \T$ we also write $\ex^f$ in place of $\exp(f)$. We have $\exp(\T)=\T^{>}$. (For $\Texp$ we do not have $\exp(\Texp)=\Texp^{>}$.) The inverse of $\exp$ is denoted by $\log : \T^{>}\to \T$, so $\log(\ell_n)=\ell_{n+1}$
for all $n$.  For $f= c\fm(1-\epsilon)$ for $c\in \R^{>}$, $\fm\in G^{\LE}$, $\epsilon \in \T^{\prec 1}$,  
$$\log(f)\ =\ \log(\fm) + \log(c) - \sum_{i=1}^{\infty} \frac{\epsilon^i}{i}\qquad  \text{(with the usual real value of $\log c$)}.$$
A key fact is that $\log(G^{\LE})=\{f\in \T:\ \supp f\succ 1\}$. 
(This was used in Section 4.3.) At this point we have justified items (1) and (2) about $\T$ in the introduction to this chapter.   

\subsection*{Proof of Proposition~\ref{trexp}} This goes just as for $\Texp$: First, for a subfield $E$ of $\T$ its small exponential closure  (in $\R[[G^{\LE}]]$)
is contained in $\T$, and is truncation closed if $E\supseteq \R$ is truncation closed. Next, we can basically copy the proof of Lemma~\ref{exptrunc}, using  the fact
that $\ex^a\in G^{\LE}$ if $a\in \T$ satisfies $\supp a \succ 1$.

\subsection*{The upward shift operator} As to items (3) and (4) in the introduction to this chapter, we just refer to \cite[Appendix A]{ADH}
for the construction of the automorphism $f\mapsto f{\uparrow}$ 
(called the {\bf upward shift}) of the ordered field $\T$
such that (3) and (4) hold. The restriction of the $n$th iterate of its inverse $f\mapsto f{\downarrow}$ to $\R[[x^{\R}]]$ agrees with the map $f\mapsto f{\downarrow}_n: \R[[x^{\R}]]\to \R[[\ell_n^{\R}]]$ in the above construction of $\T$, so
the various ways we used the notation ${\downarrow}_n$ are in agreement. Thus we have an ordered field isomorphism $f\mapsto f{\downarrow}: \R[[G^{\E,n}]]\to \R[[G^{\E,n+1}]]$ induced by the isomorphism $$\fm\mapsto \fm{\downarrow}\ :\  G^{\E,n} \to G^{\E,n+1}$$
of ordered abelian groups. Hence for $f,g\in \R[[G^{\E,n}]]$ we have: $f\truncof g \Leftrightarrow f{\downarrow}\truncof g{\downarrow}$. Since this equivalence holds for any $n$, it holds for all $f,g\in \T$, which justifies a fact stated at the end of the introduction to this chapter.  

It is also important that the upward shift is not just an isomorphism of $\T$ as an ordered field, but also of the {\em exponential} ordered field $\T$, that is, $\exp(f{\uparrow})=\exp(f){\uparrow}$ for $f\in \T$.

\subsection*{Splitting}
Here we define what it means for a set $S\subseteq G^{\LE}$ to
split so as to satisfy conditions (Sp1), (Sp2), (Sp3) from the
introduction to this chapter. First we show in the next lemma that for subsets of $G^{\E}$ to split is invariant under certain operations. The
proof uses a key fact about the automorphism $\uparrow$, namely that $G_n{\uparrow}\subseteq G_{n+1}$ for all $n$.

\begin{lemma}\label{sp1} Let $S\subseteq G^{\E}$. Then \begin{enumerate}
\item[\rm{(i)}]  $S$ splits $\ \Longleftrightarrow\ $ $S{\uparrow}$ splits;
\item[\rm{(ii)}] $S$ splits $\ \Longrightarrow\ $  $S{\downarrow}\cap G^{\E}$ splits.
\end{enumerate}
\end{lemma}
\begin{proof} We use here notation from the last chapter for $\smallk_{*}=\Texp$, so $G_n= \fM^{(n)}=\fM_0\cdots \fM_n$.  For $\fm=\fm_0\cdots\fm_n$, with $\fm_i\in \fM_i$ for $i=0,\dots,n$, we have $\fm{\uparrow}=1\cdot\fm_0{\uparrow} \cdots \fm_n{\uparrow}$
with $1\in \fM_0$ and $\fm_i{\uparrow}\in \fM_{i+1}$ for $i=0,\dots,n$. Applying this to $\fm\in S$ yields (i).  As to (ii), assume $S$  splits, and let
$\fm\in S{\downarrow}\cap G^{\E}$ be as above. Then $\fm{\uparrow}\in S$, and so by the above, $\fm_i{\uparrow}\in S$ for $i=0,\dots,n$, 
hence $\fm_i\in S{\downarrow}\cap G^{\E}$ for those $i$. Thus $S{\downarrow}\cap G^{\E}$ splits. 
\end{proof} 

\noindent
Let $S\subseteq G^{\E,n}$. Then $S=R{\downarrow}_n$ for a unique
$R\subseteq G^{\E}$, namely $R=S{\uparrow}^n$. Let us say that
$S$ \textbf{splits in $G^{\E,n}$} if $R$ as a subset of $G^{\E}$ splits. Now we also have $S\subseteq G^{\E,n+1}$, and:

\begin{lemma}\label{sp2} $S$ splits in $G^{\E,n}$ iff $S$ splits in $G^{\E,n+1}$.
\end{lemma}
\begin{proof} With $R$ as above, we have $R{\uparrow}=S{\uparrow}^{n+1}$. Apply Lemma~\ref{sp1}(i)  to $R$ in the role of $S$. 
\end{proof}

\noindent
Next we consider sets $S\subseteq G^{\LE}$, and we say that such a set $S$ {\bf splits} if $S\cap G^{\E,n}$ splits in $G^{\E,n}$ for all $n$. In order to show that
this notion has the properties (Sp1), (Sp2), (Sp3) stated in the introduction to this chapter, we need one more fact:

\begin{lemma}\label{sp3} Suppose $S\subseteq G^{\E,n+1}$ splits in $G^{\E,n+1}$. Then $S\cap G^{\E,n}$ splits in $G^{\E,n}$. 
\end{lemma} 
\begin{proof} Let $R:= S{\uparrow}^{n+1}$. Then $R\subseteq G^{\E}$  splits, so $R{\downarrow}\cap G^{\E}$ splits by Lemma~\ref{sp1}(ii). Now use that $R{\downarrow}\cap G^{\E}=\big(S\cap G^{\E,n}\big){\uparrow}^n$. 
\end{proof}

\noindent
It is now routine to show that (Sp1), (Sp2), (Sp3) hold: for (Sp2), first observe as a consequence of Lemmas~\ref{sp2} and ~\ref{sp3} that for $S\subseteq G^{\E,n}$, $S$  splits iff $S$ splits in $G^{\E,n}$.  It is also easy to check that there is no other way to extend the notion of splitting for
subsets of $G^{\E}$ to subsets of $G^{\LE}$ so that (Sp1), (Sp2), (Sp3) hold. 

\medskip\noindent
We also extend some facts from Section~\ref{neat} to the present setting:

\begin{lemma}\label{sp4} Let $R\supseteq \R$ be a truncation closed subring of $\R[[G^{\LE}]]$ such that $R\cap G^{\LE}$ splits. Let $V$ be a truncation closed $\R$-linear subspace of $\R[[G^{\LE}]]$ such that $V\cap G^{\LE}$ splits. Then the $($truncation closed$)$ subring
$R[V]$ of $\R[[G^{\LE}]]$ has the property that $R[V] \cap G^{\LE}$ splits. The $($truncation closed$)$ fraction field $F$ of $R$ in $\R[[G^{\LE}]]$ has the property that $F\cap G^{\LE}$ splits. 
\end{lemma} 

\noindent
The proofs consist of routine reductions to Lemmas~\ref{split1} and ~\ref{split2}.  

\begin{lemma}\label{sp5} Let $E\supseteq \R$ be a truncation closed ordered subfield of the $($real closed$)$ Hahn field $\R[[G^{\LE}]]$ such that $E\cap G^{\LE}$ splits.  
Let $F$ be the real closure of $E$ in $\R[[G^{\LE}]]$. Then $F\cap G^{\LE}$ splits. 
\end{lemma} 
\begin{proof}  Use that by Lemma~\ref{actr} we have $F\cap G^{\LE}=\{\fg\in G^{\LE}:\ \fg^n\in E\cap G^{\LE}\}$. 
\end{proof}

\subsection*{Proof of Proposition~\ref{trexpsplit}} Let $K\supseteq \R$ be a truncation closed subfield of $\T$ such that $K\cap G^{\LE}$ splits. Our job is to show that $K^{\exp}\cap G^{\LE}$
splits. We set $K_n:= K\cap \R[[\ell_n^{\R}]]^{\E}$, so
$K_n$ is a truncation closed subfield of $\R[[G^{\E,n}]]$ such that
$K_n\cap G^{\E,n}$ splits. The exponential closure $K_n^{\exp}$ of 
$K_n$ in $\T$ is also its exponential closure in 
$\R[[\ell_n^{\R}]]^{\E}$, so $K_n^{\exp}\cap G^{\E,n}$ splits
in $G^{\E,n}$ by Lemma~\ref{expsplit}. Taking the union over all $n$ now gives the desired result. 

\subsection*{The derivation of $\T$}  We take this from \cite[Appendix A]{ADH}, but alert the reader that the derivation on $\T$ we denote here by
$\der$ equals the derivation $x\cdot \frac{d}{dx}$  where $\frac{d}{dx}$ is the derivation on $\T$ constructed in that appendix (and which, unfortunately for us, is denoted 
there by $\der$). Keeping that in mind we easily verify item (5) of the introduction to this chapter. Moreover, if $m\ge n$, then 
$$\der \R[[G_m{\downarrow}_n]]\ \subseteq\ 
\R[[G_m{\downarrow}_n]],$$ and the restriction of $\der$ to $\R[[G_m{\downarrow}_n]]$ is a strong operator on that Hahn field over $\R$. Fixing $n$ and taking the
union over all $m\ge n$ shows that $\R[[\ell_n^{\R}]]^{\E}$ is a differential subfield of $\T$. 

We define the elements $e_n\in G_n\subseteq \Texp$ by
$e_0 := x$ and $e_{n+1}:= \exp(e_n)$. An easy induction on $n$ gives $e_n=x{\uparrow}^n$. We also set 
$\fe_n:= e_0e_1\cdots e_{n-1}$ (so $\fe_0=1$ and $\fe_1=x$).

The interaction of $\der$ and $\uparrow$ is described in the  commutative diagram

\[\xymatrixcolsep{8pc}\xymatrix{
\T \ar[d]_{\upper^n} \ar[r]^{\der} & \T \ar[d]^{\ \upper^n} \\
\T \ar[r]^{\frac{1}{\fe_n}\der} & \T
}\]
which by restriction yields the commutative diagram
\[\xymatrixcolsep{8pc}\xymatrix{
\R[[\ell_n^{\R}]]^{\E} \ar[d]_{\upper^n} \ar[r]^{\der} & \R[[\ell_n^{\R}]]^{\E} \ar[d]^{\ \upper^n} \\
\Texp \ar[r]^{\frac{1}{\fe_n}\der} & \Texp
}\]
With this last diagram we routinely reduce questions about $\R[[\ell_n^{\R}]]^{\E}$ involving $\der$ to corresponding questions about $\Texp$.  
We are now finally ready for the proof of our main result. 

\section{Proof of Theorem~\ref{thC}} 

\noindent
Let $K\supseteq \R$ be a truncation closed differential subfield of $\T$ such that
$K\cap G^{\LE}$ splits. Our job is to show that then the Liouville
closure $K^{\Li}$ of $K$ in $\T$ is truncation closed, and that
$K^{\Li}\cap G^{\LE}$ splits. First we observe that $x\in K^{\Li}$, and thus the iterated logarithms
$\ell_n$ are also in $K^{\Li}$. 
Moreover, $L:=K(\ell_n:\ n=0,1,2,\dots)$ is still a truncation closed
differential subfield of $\T$ and $L\cap G^{\LE}$ splits, so replacing $K$ by $L$ we arrange that $\ell_n\in K$ for all $n$. We can also assume $K$ is not Liouville closed. For such $K$ it is enough to show: there exists a differential field extension $L\subseteq K^{\Li}$ of $K$ such that $K\ne L$, $L$ is truncation closed, and $L\cap G^{\LE}$ splits. We now treat separately the 
various cases that may occur because of the $H$-field $K$ not being Liouville closed:

\medskip\noindent
{\em Case 1:\ $K$ is not real closed.} Then we can take for $L$ the real closure of $K$ in $\T$,  by Lemma~\ref{sp5}.

\medskip\noindent
{\em Case 2:\ $K$ is not closed under exponentiation.} 
Then we can take for $L$ the exponential closure of $K$ in $\T$,
by Propositions~\ref{trexp} and ~\ref{trexpsplit}. 

\medskip\noindent
{\em Case 3:\ $f\in K$ is such that $f\notin \der(K)$.}
We begin by observing that $\fe_n{\downarrow}_n=\ell_1\cdots \ell_n$, and that if $m < n$, then no element of $\R[[\ell_m^{\R}]]^{\E}$ has $\frac{1}{\ell_1\cdots \ell_n}$ in its support. We set
$K_n:= K\cap \R[[\ell_n^{\R}]]^{\E}$, and take $n$ so high that
$f\in K_n$ and $\frac{1}{\ell_1\cdots \ell_n}\notin\supp(f)$. Note that $K_n$ is truncation closed and
$K_n\cap G^{\E,n}$ splits. Since $K$ contains all $\ell_m$, we also have
$\frac{1}{\ell_1\cdots \ell_n}\in K_n\cap G^{\E,n}$.  Thus $F:=K_n{\uparrow}^n$ is a truncation closed subfield of $\Texp$ such that $F\supseteq \R$, $f{\uparrow}^n\in F$, $\frac{1}{\fe_n}\in (F\cap G^{\E})\setminus \supp(f{\uparrow}^n)$, and 
$F\cap G^{\E}$ splits. We now use the second commutative diagram at the end of the previous section, and let $\derdelta:=\frac{1}{\fe_n}\der$ be the derivation on $\Texp$ corresponding to the bottom horizontal line of that diagram. 

Then Theorem \ref{thmB} and the remark following its proof yields a truncation closed differential subfield $F_{\infty}$ of $F^{\Li}$ such that
$f{\uparrow}^n\in \derdelta(F_{\infty})$ and $F_{\infty}\cap G^{\E}$ splits. Applying ${\downarrow}_n$ we obtain a truncation closed
differential subfield $H:= F_{\infty}{\downarrow}_n$ of $K_n^{\Li}\subseteq K^{\Li}$ such that $H\supseteq \R$, $f\in \der(H)$ and $H\cap G^{\E,n}$ splits. Then it follows from Lemma~\ref{sp4} that $K(H)$ is a truncation closed
differential subfield of $K^{\Li}$ such that 
$f\in \der K(H)$
and $K(H)\cap G^{\LE}$ splits. So $L:=K(H)$ does the job.  

\medskip\noindent
This concludes the proof of Theorem~\ref{thC}.
 
\section{Future Directions}

\noindent
In this section we discuss possible ways of extending the results in this dissertation. First, there are other natural extension procedures for differential subfields
$K\supseteq \R$ of $\T$  besides taking the Liouville closure in $\T$. For example, $\T$ is {\em linearly surjective}, and every $K$ as above has a {\em linearly surjective closure} in $\T$. (A differential field $E$ is said to be {\bf linearly surjective\/}  if every linear differential equation $y^{(n)} + a_1y^{(n-1)} + \cdots + a_ny=b$ with $n\ge 1$ and $a_1,\dots,a_n, b\in E$ has a solution in $E$.  For $K$ as above there is a smallest linearly surjective differential subfield of $\T$ containing $K$,  the
\textbf{linearly surjective closure of $K$ in $\T$}.) The obvious question is whether Theorem~\ref{thC} goes through for the
linearly surjective closure of $K$ in $\T$ instead of its Liouville closure  in $\T$.  Going beyond linear differential equations, the most powerful
elementary property of $\T$ is that it is {\em newtonian}. (We refer to \cite[Chapter 14]{ADH} for a precise definition.)  If $K$ as above has {\em asymptotic integration} (as defined in \cite[Section 9.1]{ADH}), then $K$ has a {\em newtonization}
in $\T$ (a smallest newtonian differential subfield of $\T$ that contains $K$), and one can ask if Theorem~\ref{thC} goes through for the newtonization
instead of the Liouville closure. 

Another natural direction: try to connect the material above to surreal numbers. We refer to \cite{ADHNo} for the description of a canonical embedding of 
$\T$ as an exponential ordered field into the exponential ordered field $\No$ of surreal numbers, with $x\in \T$ corresponding to $\omega\in \No$.  For subsets of $\No$, the notion of being {\em initial} plays a key role, and is similar to being truncation closed for subsets of Hahn fields. Elliot Kaplan has shown that the image of $\T$ in $\No$ under the canonical embedding is initial. This raises natural questions for truncation closed subfields $K\supseteq \R$ of $\T$:

\begin{itemize}
\item if $K\cap G^{\LE}$ splits, does it follow that the image of $K$ in $\No$ is initial? 
\item if the image of $K$ in $\No$ is initial, does it follow that $K\cap G^{\LE}$ splits?
\end{itemize}

\end{document}